\documentclass[12pt,a4paper]{amsart}

\usepackage{amssymb}
\usepackage[numeric, abbrev, nobysame]{amsrefs}
\usepackage{amscd}

\def\frak{\mathfrak}
\def\Bbb{\mathbb}
\def\Cal{\mathcal}

\newcommand{\llbr}{[\mid}
\newcommand{\rrbr}{\mid]}

\let\phi\varphi

\newcommand{\x}{\times}
\renewcommand{\o}{\circ}

\newcommand{\al}{\alpha}
\newcommand{\be}{\beta}
\newcommand{\ga}{\gamma}

\newcommand{\ep}{\epsilon}
\newcommand{\ka}{\kappa}

\newcommand{\om}{\omega}
\newcommand{\ph}{\phi}
\newcommand{\ps}{\psi}
\renewcommand{\th}{\theta}
\newcommand{\si}{\sigma}
\newcommand{\ze}{\zeta}
\newcommand{\Ga}{\Gamma}
\newcommand{\La}{\Lambda}

\newcommand{\Ps}{\Psi}
\newcommand{\Om}{\Omega}

\newcommand{\tsum}{\textstyle\sum}

\newcommand{\im}{\operatorname{im}}

\newcommand{\End}{\operatorname{End}}
\newcommand{\gr}{\operatorname{gr}}
\newcommand{\id}{\operatorname{id}}

\newcommand{\ad}{\operatorname{ad}}
\newcommand{\Alt}{\operatorname{Alt}}

\newcommand{\tbv}{\widetilde{\mathbb V}}
\newcommand{\tcf}{\widetilde{\mathcal F}}
\newcommand{\tcv}{\widetilde{\mathcal V}}

\newcounter{theorem}
\numberwithin{theorem}{section}
\numberwithin{equation}{section}

\newtheorem{thm}[theorem]{Theorem}
\newtheorem*{thm*}{Theorem \thesubsection}
\newtheorem{lemma}[theorem]{Lemma}
\newtheorem{prop}[theorem]{Proposition}
\newtheorem{cor}[theorem]{Corollary}
\newtheorem*{lemma*}{Lemma \thesubsection}
\newtheorem*{prop*}{Proposition \thesubsection}
\newtheorem*{cor*}{Corollary \thesubsection}

\theoremstyle{definition}
\newtheorem{definition}[theorem]{Definition}
\newtheorem*{definition*}{Definition \thesubsection}

\newtheorem*{example*}{Example \thesubsection}
\theoremstyle{remark}
\newtheorem{remark}[theorem]{Remark}
\newtheorem*{remark*}{Remark \thesubsection}

\def\sideremark#1{\ifvmode\leavevmode\fi\vadjust{\vbox to0pt{\vss% the remark
 \hbox to 0pt{\hskip\hsize\hskip1em%                          will appear only
 \vbox{\hsize3cm\tiny\raggedright\pretolerance10000%          on the side
  \noindent #1\hfill}\hss}\vbox to8pt{\vfil}\vss}}}%
\newcommand\ddd[3]{%
  \begin{picture}(46,12)\put(5,3){\line(1,0){16}}\put(25,3){\line(1,0){16}}%
    \put(3,3){\makebox(0,0){$\o$}}\put(23,3){\makebox(0,0){$\o$}}%
    \put(43,3){\makebox(0,0){$\o$}}%
    \put(3,10){\makebox(0,0){\scriptsize $#1$}}%
    \put(23,10){\makebox(0,0){\scriptsize $#2$}}%
    \put(43,10){\makebox(0,0){\scriptsize $#3$}}\end{picture}}

\newcommand\xxd[3]{%
  \begin{picture}(46,12)\put(3,3){\line(1,0){38}}%
    \put(3,3){\makebox(0,0){$\times$}}\put(23,3){\makebox(0,0){$\times$}}%
    \put(43,3){\makebox(0,0){$\o$}}%
    \put(3,10){\makebox(0,0){\scriptsize $#1$}}%
    \put(23,10){\makebox(0,0){\scriptsize $#2$}}%
    \put(43,10){\makebox(0,0){\scriptsize $#3$}}\end{picture}}

\newcommand\xdd[3]{%
  \begin{picture}(46,12)\put(3,3){\line(1,0){18}}\put(25,3){\line(1,0){16}}%
    \put(3,3){\makebox(0,0){$\times$}}\put(23,3){\makebox(0,0){$\o$}}%
    \put(43,3){\makebox(0,0){$\o$}}%
    \put(3,10){\makebox(0,0){\scriptsize $#1$}}%
    \put(23,10){\makebox(0,0){\scriptsize $#2$}}%
    \put(43,10){\makebox(0,0){\scriptsize $#3$}}\end{picture}}

\begin{document}
\renewcommand{\today}{} 

\title{Relative BGG sequences;\\II.\ BGG machinery and invariant
  operators} 

\author{Andreas \v Cap and Vladimir Sou\v cek}

\address{A.\v C.: Faculty of Mathematics\\
  University of Vienna\\
  Oskar--Morgenstern--Platz 1\\
  1090 Wien\\
  Austria\\
  V.S.:Mathematical Institute\\ Charles University\\ Sokolovsk\'a
  83\\Praha\\Czech Republic} 

\email{Andreas.Cap@univie.ac.at} 
\email{soucek@karlin.mff.cuni.cz}

%\subjclass{primary: 58J10; secondary: 53C07, 53C15, 58J60, 58J70}

\begin{abstract}
  For a real or complex semisimple Lie group $G$ and two nested
  parabolic subgroups $Q\subset P\subset G$, we study parabolic
  geometries of type $(G,Q)$.  Associated to the group $P$, we
  introduce a class of relative natural bundles and relative
  tractor bundles and construct some basic invariant differential
  operators on such bundles. We define a (rather weak) notion of
  ``compressability'' for operators acting on relative differential
  forms with values in a relative tractor bundle. The we develop a
  general machinery which converts a compressable operator to an
  operator on bundles associated to completely reducible
  representations on relative Lie algebra homology groups.

  Applying this machinery to a specific compressable invariant
  differential operator of order one, we obtain a relative version of
  BGG (Bernstein--Gelfand--Gelfand) sequences. All our constructions
  apply in the case $P=G$, producing new and simpler proofs in the
  case of standard BGG sequences. We characterize cases in which the
  relative BGG sequences are complexes or even fine resolutions of
  certain sheaves and describe these sheaves. We show that this gives
  constructions of new invariant differential operators as well as of
  new subcomplexes in certain curved BGG sequences. The results are
  made explicit in the case of generalized path geometries.
\end{abstract}

\thanks{First author supported by projects P23244--N13 and P27072--N25
  of the Austrian Science Fund (FWF), second author supported by the
  grant P201/12/G028 of the Grant Agency of the Czech Republic
  (GACR). Discussions with David M.J.~Calderbank have been very
  helpful.}

\maketitle

\pagestyle{myheadings}\markboth{\v Cap and Sou\v cek}{Relative
  BGG--Sequences II}

\section{Introduction}\label{1}
The main motivation for the construction of
Bernstein--Gelfand--Gelfand sequences (or BGG sequences) came from
questions on differential operators which are naturally associated to
certain geometric structures. In particular, conformally invariant
differential operators were studied in Riemannian geometry for a long
time with rather limited success. Starting from the 1970s, it became
clear that questions on invariant differential operators are closely
related to questions in representation theory. More precisely, for
locally flat conformal structures (with the round sphere as the model
example), conformally invariant differential operators are equivalent
to homomorphisms of generalized Verma modules. A basic source for such
homomorphisms is Lepowsky's generalization \cite{Lepowsky} of the BGG
resolution of a finite dimensional representation by homomorphisms of
Verma modules, see \cite{BGG}.

The generalized BGG resolutions and the Jantzen--Zuckermann
translation principle for homomorphisms were used in the pioneering
article \cite{Eastwood--Rice} to construct conformally invariant
differential operators. Later on, these ideas were combined with
tractor calculus (see \cite{BEG} and \cite{Cap-Gover}) and extended to
the family of \textit{parabolic geometries} under the name
\textit{curved translation principle}.

The basic results on BGG sequences were obtained in \cite{CSS-BGG} and
with an improved construction in \cite{Calderbank--Diemer} following a
slightly different approach. Rather than using results on generalized
Verma modules, these articles gave a direct construction of invariant
differential operators, based on tractor bundles and the algebraic
setup introduced by Kostant for the proof of his theorem on Lie
algebra cohomology in \cite{Kostant}, which is commonly known as
Kostant's version of the Bott--Borel--Weil theorem. Kostant's theorem
itself is then used to identify the natural vector bundles which show
up in the sequence. This provided a general construction for a large
class of differential operators naturally associated to a broad class
of geometric structures, which apart from conformal structures also
contains other well known examples like projective and quaternionic
structures, CR structures and path geometries.

In the applications of BGG sequence that were found during the
subsequent years, a certain change of perspective evolved. On the one
hand, it turned out to be very important that one not only obtains a
construction of higher order operators on (relatively) simple bundles
but also a relation to first order operators on more complicated
bundles. This gives the possibility to work with the operators in a
BGG sequence without knowing explicit formulae for them (which tend to
become very complicated if the order gets high). On the other hand,
already in the original construction, two possible operators on forms
with values in a tractor bundle were used as a starting point for the
construction. In \cite{deformations} and the subsequent generalization
\cite{HSSS}, the construction was applied to certain modifications of
the canonical tractor connection. Corresponding to these developments, 
the terminology \textit{BGG machinery} started turning up.

\medskip

The aim of this article is to develop a relative version of BGG
sequences and at the same time to convert the vague idea of a ``BGG
machinery'' into precise statements. In the notation we are going to
use, the starting point for usual BGG sequences is a pair $(G,Q)$,
where $G$ is a real or complex semi--simple Lie group and $Q\subset G$
is a parabolic subgroup. The construction then associates each
representation of $G$ a sequence of invariant differential operators
on the category of parabolic geometries of type $(G,Q)$. For the
relative version we develop, one in addition chooses an intermediate
parabolic subgroup $P$ lying between $G$ and $Q$. The construction
then starts from a completely reducible representation of $P$, again
producing operators on parabolic geometries of type $(G,Q)$. We
emphasize at this point that one may always choose $P=G$ to obtain a
construction for the usual BGG sequences, which contains several new
features and strong improvements compared to the constructions in
\cite{CSS-BGG} and \cite{Calderbank--Diemer}. This is also crucial for
some of the applications discussed in the end of the article.

The relative construction needs an algebraic background, a relative
version of Kostant's harmonic theory and a description of relative Lie
algebra homology groups parallel to Kostant's theorem. These results
belong to the realm of finite dimensional representation theory and
should be of independent interest, so they have been put into the
separate article \cite{part1}.

Building on this algebraic background, we describe the setup needed
for the relative BGG construction in Section \ref{2}. It turns out
that the intermediate subgroup $P$ can be used to single out a
subclass of natural bundles that we call \textit{relative natural
  bundles}. These contain all irreducible bundles (i.e.~those
associated to irreducible representations of $Q$) and the class of
\textit{relative tractor bundles}, which are associated to completely
reducible representations of $P$. One obtains natural relative
versions of the tangent and cotangent bundle and a relative adjoint
tractor bundle. The main results of Section 2 are a construction of a
relative version of the fundamental derivative in Proposition
\ref{prop2.2}, and a relative version of the curved Casimir operator
from \cite{Casimir}, see Section \ref{2.3}.

The actual relative BGG construction is carried out in two steps. In
Section \ref{3}, we establish a machinery to construct operators on
bundles induced by relative Lie algebra homology groups from operators
defined on relative differential forms with values in a relative
tractor bundle. This construction can be applied to a single operator
mapping $k$--forms to $(k+1)$--forms, and apart from being linear, the
only requirement on this operator is compressability as defined in
Definition \ref{def3.1}. This means that the operator preserves a
natural filtration on the space of forms and has a specific induced
action on the associated graded. Beyond that, it is not even required
to be a differential operator. The main feature of the construction is
that it entirely works with (universal) polynomials in the composition
of a natural bundle map with the given operator. Hence it always
produces operators which are ``as nice'' as the operator one starts
from. 

The key step for this is the construction of a \textit{splitting
  operator}, for which we give two equivalent descriptions. One is
parallel to the general constructions of splitting operators from
curved Casimirs in \cite{Casimir} and \cite{examples}, the other
construction is closer to the one used in
\cite{Calderbank--Diemer}. We also prove that the kernel of the
initial operator naturally corresponds to a subspace in the kernel of
the induced operator on Lie algebra homology groups (Proposition
\ref{prop3.5}), which is a general version of the concept of ``normal
solutions'' of first BGG equations.

In Section \ref{3.6}, we discuss the case that one starts with a
sequence of operators on forms of all degrees rather than just a
single operator. In particular, we show that if two operators in the
sequence have trivial composition then the same is true for the
induced operators on Lie algebra homology bundles, and we precisely
analyze the relation between the cohomologies, see Theorem
\ref{thm3.6}.

\smallskip

The second part of the construction is carried out in Section
\ref{4}. Using the relative fundamental derivative, we construct a
first order invariant differential operator called the
\textit{relative twisted exterior derivative} on relative differential
forms with values in any relative tractor bundle. This operator is
shown to be compressable, so the machinery of Section \ref{3} leads to
a sequence of invariant differential operators on relative homology
bundles, see Theorem \ref{thm4.1}. In the course of the further
developments, it is shown in Proposition \ref{prop4.5} that for $P=G$,
this operator coincides with the one constructed in \cite{CSS-BGG} via
semi--holonomic jet modules.

Next, we treat the question when a relative BGG sequence is a complex
or even a fine resolution of some sheaf. Apart from a computation of
the square of the relative twisted exterior derivative (which seems to
be a new result, even for $P=G$), this requires an interpretation in
terms of a relative analog of tractor connections, see Theorem
\ref{thm4.3} and Proposition \ref{prop4.5}. This is done even in the
case where the relative tangent bundle (which always is a smooth
subbundle in the tangent bundle) is non--involutive, so the naive way
to extend a partial connection to higher order forms fails. 

Involutivity of the relative tangent bundle is necessary but not
sufficient for BGG sequences being resolutions, additional conditions
on the (relative) curvature have to be satisfied. To interpret the
sheaves resolved by a BGG resolution, one has to use the theory of
correspondence spaces and local twistor spaces for parabolic
geometries as developed in \cite{twistor}. The main general results on
BGG resolutions we prove is Theorem \ref{thm4.6}, for the
interpretation of the sheaves being resolved also Theorem \ref{thm4.6a}
is important. 

The last topic in Section \ref{4} are algebraic properties of
splitting operators which generalize the results for usual BGG
sequences in \cite{twistor}, also giving simpler proofs for those
results. The main topic here is to systematically obtain restrictions
on the image of splitting operators, given information on the sections
that they are applied to and/or on the curvature of the geometry. In
particular, these results can be used to obtain information on the
curvature of a geometry from information on its harmonic part. A
crucial result in this context is the description of the Laplacian
determined by the twisted exterior derivative in terms of the relative
curved Casimir in Proposition \ref{prop4.7}, which completes and
extends partial results in this direction from \cite{Casimir}.

\smallskip

In the last Section \ref{5}, we discuss some applications of our
results and make them explicit for one structure. First, we discuss the
case in which the bundles showing up in a relative BGG sequence also
arise in a standard BGG sequence. In representation theory terms, this
means that the weight determining the relative tractor bundle which
gives rise to the sequence is in the affine Weyl orbit of a $\frak
g$--dominant integral weight. In this case, we are able to prove in
general that the operators in the relative BGG sequence are the same
as the operators between the bundles in question that are obtained in
the absolute BGG sequence, see Theorem \ref{thm5.3}. Under the
appropriate curvature conditions, which are much weaker than local
flatness of the geometry, one thus obtains subcomplexes in curved
BGG sequences, which are different from those constructed in
\cite{subcomplexes}. These results also show how strong the
characterization results relating BGG operators to the (relative)
twisted exterior derivative are. Initially, the statement that the
bundles occur in both sequences only comes from the fact that they are
induced by representations with the same highest weight and already
finding an explicit bundle map relating absolute and relative homology
bundles is a quite non--trivial problem.

Second, there is the case in which we obtain operators that cannot
occur in a standard BGG sequence. In representation theory terms this
means that either the representation inducing the relative tractor
bundle has singular infinitesimal character or its highest weight is
non--integral. The latter is not a rare case at all because there are
density weights involved, which can be non--integral without
problems. In all these cases, we obtain a systematic and general
construction for invariant differential operators, for which up to now
there were only construction principles (which usually need
case--by--case verifications, even to decide whether they apply)
available in the literature.

We conclude the article by making our results explicit in the case of
generalized path geometries. This example of parabolic geometries is
of particular interest, since the geometric theory of systems of
second order ODEs as developed in \cite{Fels} is a special case of such
structures. This is just one example, however, and we see potential
for many further applications of relative BGG sequences. In
particular, we hope that relative BGG resolutions provide a starting
point for a curved version of the Penrose transform as described in
\cite{BEastwood}.

\section{Relative natural bundles}\label{2}
We start by briefly recalling the setup of two nested parabolic
subalgebras $\frak q\subset\frak p$ in a semisimple Lie algebra $\frak
g$ with a compatible choice $Q\subset P\subset G$ of groups as
discussed in \cite{part1}. The intermediate parabolic $\frak p$ gives
rise to a class of natural bundles on parabolic geometries of type
$(G,Q)$, which we call relative natural bundles. We show that there
are natural analogs of two of the basic differential operators
available for parabolic geometries, the fundamental derivative and the
curved Casimir operator, which are adapted to the relative
setting. Then we describe the geometric counterpart of the algebraic
setup developed in \cite{part1}, which sets the stage for the relative
BGG--machinery we develop in the next section.

\subsection{Relative natural bundles}\label{2.1}
Throughout this article, we consider a real or complex semisimple Lie
algebra $\frak g$ endowed with two nested parabolic subalgebras $\frak
q\subset\frak p\subset\frak g$. Moreover, we assume that we have
chosen a Lie group $G$ with Lie algebra $\frak g$ and a parabolic
subgroup $P\subset G$ corresponding to $\frak p$. As discussed in
Section 2.1 of \cite{part1} the normalizer $Q$ of $\frak q$ in $P$ has
Lie algebra $\frak q$, so we obtain closed subgroups $Q\subset
P\subset G$ corresponding to $\frak q\subset\frak p\subset\frak
g$. For each of the parabolic subalgebras, we have the nilradical, and
we denote these by $\frak p_+\subset\frak p$ and $\frak
q_+\subset\frak q$. It turns out that $\frak p_+\subset\frak q_+$ and
that the exponential map restricts to diffeomorphisms from these
subalgebras onto closed subgroups $P_+\subset Q_+\subset Q\subset P$
such that $Q_+$ is normal in $Q$ and $P_+$ is normal in $P$.

In this setting, we will study parabolic geometries of type $(G,Q)$,
and use $P$ (or $\frak p$) as an additional input. By definition,
these are Cartan geometries of type $(G,Q)$ and hence can exist on
smooth manifolds of dimension $\dim(G/Q)$. Explicitly, such a geometry
on a smooth manifold $M$ is given by a principal fiber bundle $p:\Cal
G\to M$ with structure group $Q$ together with a Cartan connection
$\om\in\Om^1(\Cal G,\frak g)$. This means that $\om$ defines a
trivialization $T\Cal G\cong M\x\frak g$ which is $P$--equivariant and
reproduces the generators of fundamental vector fields, see section
1.5 of \cite{book} for details on Cartan geometries. There is a
general theory exhibiting such parabolic geometries as equivalent
encodings of underlying structures. For the purposes of this article,
we may however simply take the Cartan geometry as a given input.

From this description it is clear, that a representation $\Bbb W$ of
the Lie group $Q$ gives rise to a natural vector bundle on parabolic
geometries of type $(G,Q)$. If $(p:\Cal G\to M,\om)$ is such a
geometry, then we simply form the associated bundle $\Cal G\x_Q\Bbb
W$. Via the Cartan connection $\om$, one can identify some of these
natural bundles with more traditional geometric objects like tensor
bundles.

\begin{definition}
  Suppose that $Q\subset P\subset G$ are nested parabolic subgroups
  and let $\Bbb W$ be a representation of $Q$, and consider the
  corresponding natural vector bundle $\Cal W$ on parabolic geometries
  of type $(G,Q)$.

  (1) $\Cal W$ is called a \textit{relative natural bundle} if the
  subgroup $P_+\subset Q$ acts trivially on $\Bbb W$.

  (2) $\Cal W$ is called a \textit{relative tractor bundle} if $\Bbb
  W$ is the restriction to $Q$ of a representation of $P$ on which
  $P_+$ acts trivially.
\end{definition}

Observe that trivial action of $P_+$ is equivalent to trivial action
of $\frak p_+$ under the infinitesimal representation. Moreover, on
irreducible (and hence on completely reducible) representations of any
parabolic subgroup, the nilradical always acts trivially. Hence any
completely reducible representation of $Q$ gives rise to a relative
natural bundle (these are the usual completely reducible bundles) and
any completely reducible representation of $P$ gives rise to a
relative tractor bundle.

Beyond the class of completely reducible natural bundles, we can
immediately construct some fundamental examples of relative natural
bundles. Recall that for any parabolic geometry $(p:\Cal G\to M,\om)$,
the tangent bundle $TM$ is the associated bundle $\Cal G\x_Q(\frak
g/\frak q)$. This is not a relative natural bundle in
general. However, the additional parabolic subalgebra $\frak
p\subset\frak g$ is a $Q$--invariant subspace, which gives rise to a
smooth subbundle $\Cal G\x_Q(\frak p/\frak q)=:T_\rho M\subset
TM$. Since $\frak p_+$ is an ideal in $\frak p$ we get $[\frak
p_+,\frak p]\subset\frak p_+\subset\frak q$. Thus $\frak p_+$ acts
trivially on $\frak p/\frak q$, so $T_\rho M$ is a relative natural
bundle, which we will call the \textit{relative tangent bundle}.

From the definition it is clear, that the class of relative natural
vector bundles is closed under forming natural subbundles and
quotients and under the usual functorial constructions like sums,
tensor products, duals and so on. In particular, the dual $T^*_\rho M$
of $T_\rho M$ is also a relative natural bundle, which we call the
\textit{relative cotangent bundle}. As discussed in Section 2.3 of
\cite{part1}, the Killing form of $\frak g$ induces dualities between
$\frak g/\frak q$ and $\frak q_+$ and between $\frak g/\frak p$ and
$\frak p_+$, which implies that it also gives rise to a duality
between $\frak p/\frak q$ and $\frak q_+/\frak p_+$. Thus $T^*_\rho
M\cong\Cal G\x_Q (\frak q_+/\frak p_+)$, so in particular, this is naturally a
bundle of nilpotent Lie algebras. Having the relative tangent bundle
and the relative cotangent bundle at hand, we can of course form
relative tensor bundles, and in particular, there is the bundle
$\La^kT^*_\rho M$ of relative $k$--forms, which is the associated
bundle corresponding to the $Q$--module $\La^k(\frak q_+/\frak p_+)$.

\subsection{Relative adjoint tractor bundle and relative fundamental
  derivative}\label{2.2}
Recall that for a parabolic geometry $(p:\Cal G\to M,\om)$ of type
$(G,P)$, the \textit{adjoint tractor bundle} is the natural bundle
$\Cal AM:=\Cal G\x_Q\frak g$. This is a fundamental example of a
tractor bundle (since it is induced by the restriction to $Q$ of a
representation of $G$), but of course not a relative natural
bundle. There is a relative analog of this bundle, however. The group
$P$ acts on its Lie algebra $\frak p$ by the adjoint representation
and the nilradical $\frak p_+\subset\frak p$ is invariant under this
action. Hence there is an induced action on the quotient $\frak
p/\frak p_+$. Since $\frak p_+$ is an ideal in $\frak p$, it acts
trivially on this quotient, so $\Cal A_\rho M:=\Cal G\x_Q(\frak
p/\frak p_+)$ is a relative tractor bundle called the \textit{relative
  adjoint tractor bundle}.

The bundle $\Cal A_\rho M$ has properties similar to $\Cal AM$ in many
respects. First, $\frak p/\frak p_+$ naturally is a Lie algebra and
the bracket is $P$--invariant and hence $Q$--invariant. Thus we get an
induced bilinear bundle map $\{\ ,\ \}:\Cal A_\rho M\x\Cal A_\rho
M\to\Cal A_\rho M$. Second, as discussed in Section 2.5 of
\cite{part1}, the Lie algebra $\frak g$ carries a natural
$Q$--invariant filtration, which restricts to $Q$--invariant
filtrations on $\frak p$ and $\frak p_+$. The resulting $Q$--invariant
filtration of $\frak p/\frak p_+$ induces a filtration of $\Cal A_\rho
M$ by smooth subbundles $\Cal A^i_\rho M$. Since the initial
filtration is compatible with the Lie bracket, we conclude that
$\{\Cal A^i_\rho M,\Cal A^j_\rho M\}\subset\Cal A^{i+j}_\rho M$, so
$\Cal A_\rho M$ is a bundle of filtered Lie algebras. By definition,
we further have $\frak p^0=\frak q$ and $\frak p^1=\frak q_+$. Passing
to associated graded bundles, this implies that $\Cal A_\rho M/\Cal
A^0_\rho M=T_\rho M$ and that $\Cal A^1_\rho M=T^*_\rho M$. We will
denote by $\Pi_\rho$ the projection $\Ga(\Cal A_\rho M)\to\Ga(T_\rho
M)\subset\frak X(M)$ induced by the first isomorphism.

Having the relative adjoint tractor bundle at hand, we can construct a
relative version of the most basic differential operator available on
any Cartan geometry, the so--called fundamental derivative. Recall
that via the Cartan connection $\om$, sections of the adjoint tractor
bundle $\Cal AM$ can be identified with $Q$--invariant vector fields
on the total space $\Cal G$ of the Cartan bundle. Given any associated
bundle $E$ to $\Cal G$, one can identify its sections with
$Q$--equivariant functions on $\Cal G$, and differentiating such a
function with a $Q$--invariant vector field, the result is
$Q$--equivariant again. Thus one obtains a natural bilinear
differential operator $D:\Ga(\Cal AM)\x\Ga(E)\to\Ga(E)$, which, to
emphasize the analogy to a covariant derivative, is written as
$(s,\si)\mapsto D_s\si$. By construction, this operator is linear over
smooth functions in the $\Cal AM$--slot, so it can also be interpreted
as a natural linear operator $\Ga(E)\to\Ga(\Cal A^*M\otimes E)$, and
in this form it can evidently be iterated.

To construct a relative version of this operator, we need another
property of the fundamental derivative. The $Q$--invariant filtration
of $\frak g$ induces a filtration of $\Cal AM$ by smooth subbundles
$\Cal A^iM$ and in particular $\Cal A^0M=\Cal G\x_Q\frak q$. Now if
$E=\Cal G\x_Q\Bbb W$ for a representation $\Bbb W$ of $Q$, then the
infinitesimal representation defines a $Q$--equivariant, bilinear map
$\frak q\x\Bbb W\to\Bbb W$. Passing to associated bundles, we get a
bilinear bundle map $\Cal A^0M\x E\to E$, which we write as
$(s,\si)\mapsto s\bullet\si$. Now in the above picture of vector
fields on $\Cal G$, sections of $\Cal A^0M$ correspond to vertical
vector fields, and equivariancy implies that $D_s\si=-s\bullet\si$ for
$s\in\Ga(\Cal A^0M)$ and $\si\in\Ga(E)$.

Now the $Q$--invariant subspaces $\frak p\subset\frak g$ and $\frak
p_+\subset\frak q$ give rise to smooth subbundles $\Cal G\x_Q\frak
p_+\subset\Cal A^0M\subset\Cal G\x_Q\frak p\subset\Cal AM$ and the
quotient of the third of these bundles by the first one can be
identified with $\Cal A_\rho M$. Moreover, if $\Bbb W$ is a
representation of $Q$ inducing a relative natural bundle, then the
infinitesimal representation $\frak q\otimes\Bbb W\to\Bbb W$ descends
to $\frak q/\frak p_+$ in the first factor, and the latter
representation induces the subbundle $\Cal A^0_\rho M\subset\Cal
A_\rho M$. If $E$ is the relative natural bundle determined by $\Bbb
W$, then we get an induced bilinear bundle map $\bullet:\Cal A^0_\rho
M\x E\to E$. Having all that at hand, we can construct the
\textit{relative fundamental derivative} and prove that it has the
same strong naturality properties as the fundamental derivative.

\begin{prop}\label{prop2.2}
  For a relative natural bundle $E$, the fundamental derivative
  induces a well defined operator $D^\rho:\Ga(\Cal A_\rho
  M)\x\Ga(E)\to\Ga(E)$ which has the following properties.

  (1) For $s\in\Ga(\Cal A^0_\rho M)$ and $\si\in\Ga(E)$, we get
  $D^\rho_s\si=-s\bullet\si$.

  (2) For $E=M\x\Bbb R$, we get $D_s f=\Pi_\rho(s)\cdot f$ for all
  $s\in\Ga(\Cal A_\rho M)$ and $f\in C^\infty(M,\Bbb R)$.

  (3) The operators $D^\rho$ are compatible with any bundle map which
  comes from a $Q$--equivariant linear map on the inducing
  representation. In particular, one obtains Leibniz rules both for
  the multiplication by functions and for tensor products and
  compatibility on dual bundles in the usual sense (c.f.~Proposition
  1.5.8 in \cite{book}).
\end{prop}
\begin{proof}
  We can first restrict the fundamental derivative to a operation
   $$
 \Ga(\Cal G\x_Q\frak p)\x\Ga(E)\to \Ga(E).
 $$ 
 Since $\frak p_+\subset\frak q$, this coincides with the negative of
 $\bullet$ on $\Ga(\Cal G\x_Q\frak p_+)\x\Ga(E)$ and hence vanishes
 identically if $E$ is a relative natural bundle. Hence we get a well
 defined operator as claimed. The claimed properties of $D^\rho$ then
 follow readily from the analogous properties of the fundamental
 derivative as proved in Proposition 1.5.8 of \cite{book}.
\end{proof}

\subsection{The relative curved Casimir operator}\label{2.3}
The relative fundamental derivative leads to a relative version of
another basic tool for parabolic geometries, the curved Casimir
operator originally introduced in \cite{Casimir}. It is a general fact
that for parabolic subalgebras, the nilradical coincides with the
annihilator under the Killing form. Hence the Killing form of $\frak
g$ descends to a non--degenerate bilinear form $B$ on $\frak p/\frak
p_+$, which of course is $Q$ invariant (and even
$P$--invariant). Since $B$ then identifies $\frak p/\frak p_+$ with
its dual, we can view $B^{-1}$ as an invariant, non--degenerate
bilinear form on the dual. Given a representation $\Bbb W$ of $Q$
inducing a relative natural bundle $E$, we get an induced bundle map
$$
 B^{-1}\otimes\id:\Cal A^*_\rho M\otimes\Cal A^*_\rho M\otimes E\to E,
$$
and we denote by the same symbol the corresponding tensorial operator
on smooth sections.

\begin{definition}
  Given a relative natural bundle $E$, we define the \textit{relative
    curved Casimir operator} $\Cal C_\rho :\Ga(E)\to\Ga(E)$ by
 $$
 \Cal C_\rho (\si):=(B^{-1}\otimes\id)(D^\rho D^\rho\si).
 $$  
\end{definition}

By construction, $\Cal C_\rho$ has, in the category of relative
natural bundles, analogous naturality properties as proved for the
curved Casimir in Proposition 2 of \cite{Casimir}.

The simplest way to evaluate the relative curved Casimir is via dual
frames. Choose a local frame $\{s_\ell\}$ for $\Cal A_\rho M$ and
denote by $\{t_\ell\}$ the dual frame with respect to $B$, so
$B(s_i,t_j)=\delta_{ij}$. Then by definition, for a section
$\si\in\Ga(E)$, one can compute $\Cal C_\rho (\si)$ on the domain of
definition of the frame as
$$
 \textstyle\sum_\ell (D^\rho D^\rho\si)(t_\ell,s_\ell)=\sum_\ell
 \big(D^\rho_{t_\ell}D^\rho_{s_\ell}\si-D^\rho_{D^\rho_{t_\ell}s_\ell}\si).
$$ 
As in the case of the ordinary curved Casimir, this expression can be
simplified considerably by considering a special class of so--called
adapted local frames. Recall that $\Cal A_\rho M$ is filtered by
smooth subbundles $\Cal A^i_\rho M$ and for $i=0$ and $i=1$, the
filtration components correspond to the subspaces $\frak q/\frak p_+$
and $\frak q_+/\frak p_+$ of $\frak p/\frak p_+$, respectively. Since
$B$ is induced by the Killing form of $\frak g$ and the $Q$--invariant
filtration of $\frak p/\frak p_+$ is induced by the filtration on
$\frak g$, the usual compatibilities between the two structures hold
in this case. In particular, for $i>0$, the degree $i$ filtration
component coincides with the annihilator with respect to $B$ of the
component of degree $-i+1$. Moreover, as noted before, the filtration
is compatible with the induced Lie bracket.

\begin{definition}\label{def2.3.2}
 An \textit{adapted local frame} for $\Cal A_\rho M$ is a local frame
 of the form $\{X_i,A_r,Z^i\}$ with the following properties:
 \begin{itemize}
 \item $Z^i\in\Ga(\Cal A^1_\rho M)$ for all $i$ and $A_r\in\Ga(\Cal
   A^0_\rho M)$ for all $r$. 
 \item We have $B(X_i,X_j)=0$, $B(A_r,X^i)=0$, and
   $B(X_i,Z^j)=\delta_i^j$ for all $i$, $j$, and $r$.
 \item For all $i$, the algebraic bracket $\{Z^i,X_i\}$ is a section of
   $\Cal A^0_\rho M$.  
 \end{itemize}
\end{definition}

Note that the last condition in this definition is not explicitly
stated in \cite{Casimir} but used afterwards. The proof of existence
of such frames and of their fundamental properties is parallel to the
case of the usual curved Casimir.

\begin{lemma}\label{lem2.3}
  Adapted local frames for $\Cal A_\rho M$ exist for each parabolic
  geometry of type $(G,Q)$. Moreover, if $\{X_i,A_r,Z^i\}$ is such a
  frame, then there are local sections $A^r\in\Ga(\Cal A^0_\rho M)$
  such that the dual frame has the form $\{Z^i,A^r,X_i\}$. Finally,
  $B$ descends to a non--degenerate bilinear form on $\Cal A^0_\rho
  M/\Cal A^1_\rho M\cong G\x_Q(\frak q/\frak q_+)$ and $\{A_r\}$ and
  $\{A^r\}$ descend to dual local frames for this quotient bundle.
\end{lemma}
\begin{proof}
  Consider an open subset over which the Cartan bundle is
  trivial. Then all natural bundles are trivial and hence admit local
  frames there. We start by choosing a local frame $\{Z^i\}$ for $\Cal
  A^1_\rho M$ which starts with a local frame for the smallest
  filtration component, then continues with the next larger filtration
  component and so on. In particular, this implies that for each $i$
  and $j$, the algebraic bracket $\{Z_i,Z_j\}$ can be written as a
  linear combination of elements $Z_\ell$ where $\ell<i$ and $\ell<j$.

  Since $\Cal A^1_\rho M$ is the annihilator with respect to $B$ of
  $\Cal A^0_\rho M$, we get a duality between $\Cal A^1_\rho M$ and
  $\Cal A_\rho M/\Cal A^0_\rho M$ induced by $B$. Now we consider the
  frame $\{\underline{X}_i\}$ of that bundle which is dual to
  $\{Z_i\}$ and for each $i$ choose a preimage
  $\widetilde{X}_i\in\Ga(\Cal A_\rho M)$ of
  $\underline{X}_i$. Finally, choose any local frame
  $\{\underline{A}_r\}$ for the quotient bundle $\Cal A^0_\rho M/\Cal
  A^1_\rho M$ and for each $r$ choose a preimage $\widetilde{A}_r\in
  \Ga(\Cal A^0_\rho M)$ of $\underline{A}_r$. Then by construction
  $\{\widetilde{X}_i,\widetilde{A}_r,Z^i\}$ is a local frame for $\Cal
  A_\rho M$ and since $B$ vanishes on $\Cal A^0_\rho M\x\Cal A^1_\rho
  M$, we have $B(Z_i,Z_j)=0$, $B(\widetilde{A}_r,Z_i)=0$ and
  $B(\widetilde{X}_i,Z^j)=\delta_i^j$.

  Putting $X_i:=\widetilde{X}_i-\sum_j\frac12B(\widetilde{X}_i,
  \widetilde{X}_j)Z^j$, we see that we still get
  $B(X_i,Z^j)=\delta_i^j$ but also $B(X_i,X_j)=0$ for all $i$ and
  $j$. Defining
  $A_r:=\widetilde{A}_r-\sum_iB(\widetilde{A}_r,X_i)Z^i$, we see that
  $\{X_i,A_r,Z^i\}$ is a local frame which satisfies the first two
  conditions of an adapted frame. But for the algebraic bracket
  $\{Z^i,X_i\}$, invariance of the the Killing form implies that
  $B(\{Z^i,X_i\},Z_j)=-B(X_i,\{Z_i,Z_j\})$. As we have noted above,
  for any $\{Z_i,Z_j\}$ is a linear combination of elements $Z_\ell$
  with $\ell<i$, so $B(X_i,\{Z_i,Z_j\})=0$. Since this holds for all
  $j$, $\{Z^i,X_i\}$ lies in the annihilator of $\Cal A^1_\rho M$, so
  we have constructed an adapted frame.

  From the behavior of $B$ with respect to our frame, it follows
  immediately that the dual frame must be of the form
  $\{Z^i,A^r,X_i\}$ for some sections $A^r$ of $\Cal A_\rho M$. But by
  definition $B(A^r,Z^j)=0$ for all $r$ and $j$, so the $A^r$ are
  sections of $\Cal A^0_\rho M$. The last claim is then obviously
  true.
\end{proof}

In terms of an adapted local frame, the action of the relative curved
Casimir is easy to compute.

\begin{prop}\label{prop2.3}
  In terms of an adapted local frame $\{X_i,A_r,Z^i\}$ and the
  elements $A^r$ in the dual frame, the relative curved Casimir is
  given by
$$ 
 \Cal C_\rho(\si)=-2\textstyle\sum_iZ^i\bullet
 D^\rho_{X_i}\si-\sum_i\{Z^i,X_i\}\bullet\si+\sum_rA^r\bullet
 A_r\bullet\si.
$$ 
In particular, the relative curved Casimir has at most order one and
it has order zero on relative natural bundles induced by completely
reducible representations of $Q$.
\end{prop}
\begin{proof}
  We evaluate the relative curved Casimir with respect to the dual
  frames $\{X_i,A_r,Z^i\}$ and $\{Z^i,A^r,X_i\}$ as described
  above. For the first summands, we use
  $D^\rho_{Z^i}D^\rho_{X_i}\si=-Z^i\bullet D^\rho_{X_i}\si$ and
  $-D^\rho_{D^\rho_{Z^i}X_i}\si=D^\rho_{\{Z^i,X_i\}}\si=-\{Z^i,X_i\}\bullet\si$.
  Similarly, for the second summands, we get
  $D^\rho_{A^r}D^\rho_{A_r}\si=A^r\bullet A_r\bullet\si$ and
  $-D^\rho_{D^\rho_{A^r}A_r}\si=-\{A^r,A_r\}\bullet\si$, so these add
  up to $\sum_rA_r\bullet A^r\bullet\si$. Finally, for the last
  summands, we get
  $D^\rho_{X_i}D^\rho_{Z^i}\si=-D^\rho_{X_i}(Z^i\bullet\si)$ and
  $-D^\rho_{D^\rho_{X_i}Z^i}\si=(D^\rho_{X_i}Z^i)\bullet\si$, so these
  add up to $-\sum_iZ^i\bullet D^\rho_{X_i}\si$, so the claimed
  formula for $\Cal C_\rho (\si)$ follows.

  From this formula, it is evident that $\Cal C_\rho$ is an operator
  of order at most one, with the first order part coming only from the
  terms $Z^i\bullet D^\rho_{X_i}\si$. But on completely reducible
  representations of $Q$, $\frak q_+$ acts trivially, so $\bullet$
  vanishes identically on $\Cal A^1_\rho M\x E$ in this case, and
  hence $\Cal C_\rho$ is tensorial.
\end{proof}

In the case of a natural bundle $E$ induced by a complex irreducible
representation of $Q$, the last property in the proposition readily
implies that $\Cal C$ acts by a scalar on $\Ga(E)$. The corresponding
eigenvalue can be computed in terms of representation theory data in
complete analogy to \cite{Casimir}. We do not go into details here,
since we will not need this result.

\subsection{The setup for the relative BGG machinery}\label{2.4}
We next discuss operations on relative natural bundles coming from the
relative version of Kostant's algebraic harmonic theory from
\cite{part1}. We start with a representation $\Bbb V$ of $P$, such
that $\frak p_+$ acts trivially under the infinitesimal
representation. Given a parabolic geometry $(p:\Cal G\to M,\om)$ of
type $(G,Q)$, the associated bundle $\Cal VM:=\Cal G\x_Q\Bbb V$ by
definition is a relative tractor bundle. The main objects we will
study are the bundles $\La^kT_\rho^*M\otimes\Cal VM$ of relative
differential forms with coefficients in $\Cal VM$. By definition,
these bundles are induced by the representations $\La^k(\frak
q_+/\frak p_+)\otimes\Bbb V$ of $Q$.

\begin{prop}\label{prop2.4.1}
   The relative Kostant codifferential introduced in Section 2.2 of
   \cite{part1} gives rise to morphisms
$$
 \partial^*_\rho : \La^kT^*_\rho M\otimes\Cal VM\to \La^{k-1}T^*_\rho
 M\otimes\Cal VM.
$$
of natural bundles such that
$\partial^*_\rho\o\partial^*_\rho=0$. Hence we obtain smooth
subbundles
$\im(\partial^*_\rho)\subset\ker(\partial^*_\rho)\subset\La^kT^*_\rho
M\otimes\Cal VM$. The quotient bundle
$\ker(\partial^*_\rho)/\im(\partial^*_\rho)$ is a completely reducible
bundle, which can be identified with the bundle induced by the Lie
algebra homology space $H_k(\frak q_+/\frak p_+,\Bbb V)$.
\end{prop}
\begin{proof}
  The relative Kostant codifferential is a $Q$--equivariant map
$$
\La^k(\frak q_+/\frak p_+)\otimes\Bbb V\to \La^{k-1}(\frak
  q_+/\frak p_+)\otimes\Bbb V,
$$ 
which immediately implies the first statement. The fact that the
composition of two codifferentials is zero of course carries over to
the induced bundle maps, and hence kernel and image are nested smooth
subbundles. The last statement follows by definition of Lie algebra
homology and the fact that Lie algebra homology groups are completely
reducible representations, see Proposition 2.1 of \cite{part1}.
\end{proof}

The Lie algebra homology interpretation can be carried over to the
bundle level. As we have observed in Section \ref{2.1}, $T^*_\rho M$
is a bundle of Lie algebras, and by construction $\Cal VM$ is a bundle
of modules over this bundle of Lie algebras. In this language, the
bundle maps $\partial^*_\rho$ are just the point--wise Lie algebra
homology differentials as in formula (2.2) of \cite{part1}. Hence the
quotient bundle $\ker(\partial_*)/\im(\partial_*)$ can be interpreted
as forming a Lie algebra homology group in each point. Thus we will
denote these bundles by $H_k(T^*_\rho M,\Cal VM)$ in what follows.

\medskip

As discussed in Section 2.5 of \cite{part1}, the representations $
\La^k(\frak q_+/\frak p_+)\otimes\Bbb V$ carry natural filtrations by
$Q$--invariant subspaces, which induce filtrations of the bundles
$\La^kT_\rho^*M\otimes\Cal VM$ by smooth subbundles. Now one can pass
to the associated graded both on the level of representations and on
the level of bundles, and this is compatible with forming induced
bundles. Moreover, forming the associated graded is compatible (on
both sides) with tensorial operations, compare with Section 3.1.1 of
\cite{book}. Hence we can view the associated graded bundles as
$\La^k\gr(T^*_\rho M)\otimes\gr(\Cal VM)$ and they are induced by the
representations $\La^k\gr(\frak q_+/\frak p_+)\otimes\gr(\Bbb V)$. As
a vector space, the latter representation can be identified with
$\La^k(\frak q_+/\frak p_+)\otimes\Bbb V$ but $\frak q_+$ acts
trivially on the associated graded, whence this descends to a
representation of $Q/Q_+$, see again Section 2.5 of
\cite{part1}. Hence $\La^k\gr(T^*_\rho M)\otimes\gr(\Cal VM)$ is a
completely reducible natural bundle, and thus a much simpler
geometric object than the original bundle of relative forms.

\begin{prop}\label{prop2.4.2}
  (1) The bundle maps $\partial^*_\rho$ from Proposition
  \ref{prop2.4.1} are compatible with the natural filtrations on the
  bundles $\La^k(T^*_\rho M)\otimes\Cal VM$ and thus induce, for each
  $k$, natural bundle maps
  $$
 \underline{\partial}^*_\rho:\La^k\gr(T^*_\rho M)\otimes\gr(\Cal VM)\to
 \La^{k-1}\gr(T^*_\rho M)\otimes\gr(\Cal VM).  
  $$

  (2) The relative Lie algebra cohomology differential from Section
  2.3 of \cite{part1} induces, for each $k$, natural bundle maps
$$
\partial_\rho:\La^k\gr(T^*_\rho M)\otimes\gr(\Cal VM)\to
\La^{k+1}\gr(T^*_\rho M)\otimes\gr(\Cal VM)
$$
such that $\partial_\rho\o\partial_\rho=0$.

(3) Defining
$\square_\rho:=\underline{\partial}^*_\rho\o\partial_\rho+
 \partial_\rho\o\underline{\partial}^*_\rho$, we get, for each $k$, a
 decomposition
$$
 \La^k\gr(T^*_\rho M)\otimes\gr(\Cal
 VM)=\im(\underline{\partial}^*_\rho)\oplus\ker(\square_\rho)
 \oplus\im(\partial_\rho) 
$$
as a direct sum of natural subbundles. Moreover, the first two
summands add up to $\ker(\underline{\partial}^*_\rho)$, while the last
two summands add up to $\ker(\partial_\rho)$.
\end{prop}
\begin{proof}
  Part (1) follows directly from the algebraic properties of the
  Kostant codifferential. As discussed in Section 2.5 of \cite{part1},
  the Lie algebra cohomology differential can be viewed as a
  $Q$--equivariant map on the associated graded representation
  $\La^*\gr(\frak q_+/\frak p_+)\otimes\gr(\Bbb V)$, so part (2)
  follows. Finally, part (3) is a direct consequence of the algebraic
  Hodge decomposition proved in Lemma 2.2 of \cite{part1} and the
  discussion in Section 2.5 of that reference.
\end{proof}

\section{The relative BGG machinery}\label{3}
Having the necessary setup at hand, we can take the first step towards
the construction of relative BGG sequences. We develop a machinery to
compress operators (with a certain property) defined on relative
differential forms with values in a relative tractor bundle to
operators defined on the corresponding Lie algebra homology
bundles. This procedure is very general (not even requiring the
initial operator to be differential) but set up in such a way that
nice properties of the initial operators carry over to the compressed
operators.

\subsection{Compressable operators}\label{3.1}
Consider a relative tractor bundle $\Cal VM=\Cal G\x_Q\Bbb V$. Then we
denote the spaces of relative differential forms with values in $\Cal
VM$ by $\Om^k_\rho(M,\Cal VM):=\Ga(\La^kT^*_\rho M\otimes\Cal
VM)$. The input needed for the relative BGG machinery is a linear
operator $\Cal D=\Cal D_k:\Om^k_\rho(M,\Cal
VM)\to\Om^{k+1}_\rho(M,\Cal VM)$, satisfying a certain condition,
respectively a sequence of such operators. The condition in question
is compatibility with a natural filtration together with a condition
on the induced operator on the associated graded. Let us explain the
necessary background and at the same time make things more explicit.

We have already noted that $T_\rho M$ and $\Cal VM$ are filtered by
smooth natural subbundles. Let us denote these filtrations by
\begin{gather*}
 T_\rho M=T^{-\mu}_\rho M\supset T^{-\mu+1}_\rho M\supset\dots\supset
 T^{-1}_\rho M\\
 \Cal VM=\Cal V^0M\supset\Cal V^1M\supset\dots\supset \Cal V^NM.
\end{gather*}
Then we call a form $\ph\in\Om^k_\rho(M,\Cal VM)$
(filtration--)homogeneous of degree $\geq\ell$ if for all
$\xi_1,\dots,\xi_k\in\frak X(M)$ such that $\xi_i\in\Ga(T^{j_i}_\rho
M)$ we have $\ph(\xi_1,\dots,\xi_k)\in \Cal
V^{j_1+\dots+j_k+\ell}M$. Note that this is a purely pointwise
condition, so it just means that $\ph$ is a section of the filtration
component of degree $\ell$ of the bundle $\La^kT^*_\rho M\otimes\Cal
VM$.

Now we say that $\Cal D:\Om^k(M,\Cal VM)\to\Om^{k+1}(M,\Cal VM)$ is
\textit{compatible with the natural filtration} if for each $\ell$ and
any $k$--form $\ph$ which is homogeneous of degree $\geq\ell$, also
$\Cal D(\ph)$ is homogeneous of degree $\geq\ell$. If this holds, then
we can take a form $\ph$, which is homogeneous of degree $\geq\ell$,
and form the projection $\gr_\ell(\Cal D(\ph))$. This is a section of
the degree--$\ell$ part $\gr_\ell(\La^{k+1}T^*_\rho M\otimes\Cal VM)$
of the associated graded bundle. If we add to $\ph$ a form
$\ps\in\Om^k(M,\Cal VM)$, which is homogeneous of degree $\geq\ell+1$,
then $\gr_\ell(\Cal D(\ph+\ps))=\gr_\ell(\Cal D(\ph))$.

On the other hand, given a section $\al$ of $\gr_\ell(\La^kT^*_\rho
M\otimes\Cal VM)$, we can choose $\ph\in\Om^k_\rho(M,\Cal VM)$ which
is homogeneous of degree $\geq\ell$ such that $\gr_\ell(\ph)=\al$, and
$\ph$ is unique up to addition of a form $\ps$ which is homogeneous of
degree $\geq\ell+1$. Consequently, we see that $\gr_\ell(\Cal D(\ph))$
depends only on $\al$ and not on the choice of $\ph$. Thus we conclude
that any operator $\Cal D:\Om^k_\rho(M,\Cal
VM)\to\Om^{k+1}_\rho(M,\Cal VM)$ which is compatible with the natural
filtration induces an operator
$$
 \gr_0(\Cal D_k):\Ga(\gr(\La^kT^*_\rho M\otimes\Cal
 VM))\to\Ga(\gr(\La^{k+1}T^*_\rho M\otimes\Cal VM))  
$$
which is homogeneous of degree zero. Using this, we can now formulate
a crucial definition.

\begin{definition}\label{def3.1}
  A linear operator $\Cal D:\Om^k_\rho(M,\Cal
  V)\to\Om^{k+1}_\rho(M,\Cal VM)$ is called \textit{compressable} if
  and only if it preserves the natural filtration and the induced
  operator $\gr_0(\Cal D)$ is the tensorial operator induced by the
  Lie algebra cohomology differential $\partial_\rho$ from part (2) of
  Proposition \ref{prop2.4.2}.
\end{definition}

Once one has found one compressable operator, it is easy to describe
all of them. Indeed, if $\Cal D$ and $\tilde{\Cal D}$ both are
compressable operators defined on $\Om^k_\rho(M,\Cal VM)$, then
consider the difference $\tilde{\Cal D}-\Cal D$. Of course, this
preserves the natural filtration on $\Cal VM$--valued forms and the
induced operator $\gr_0(\tilde{\Cal D}-\Cal D)$ is identically
zero. But this exactly means that for any $\ph\in\Om^k_\rho(M,\Cal
VM)$ which is homogeneous of degree $\geq\ell$, the form $(\tilde{\Cal
  D}-\Cal D)(\ph)\in\Om^{k+1}_\rho(M,\Cal VM)$ is homogeneous of
degree $\geq \ell+1$.

Conversely if $\Cal D$ is compressable and $\Cal E:\Om^k_\rho(M,\Cal
V)\to\Om^{k+1}_\rho(M,\Cal VM)$ is any linear operator, which strictly
increases homogeneous degrees, then $\Cal D+\Cal E$ is again
compressable.

\subsection{}\label{3.2}
Given a compressable linear operator $\Cal D:\Om^k_\rho (M,\Cal VM)\to
\Om^{k+1}_\rho (M,\Cal VM)$, the key point is to study the operator
$\partial^*_\rho\o\Cal D$, which maps $\Om^k_\rho (M,\Cal VM)$ to
itself. To simplify notation, we will write this composition as
$\partial^*_\rho\Cal D$ from now on. Note that there are the natural
subbundles
$\im(\partial^*_\rho)\subset\ker(\partial^*_\rho)\subset\La^kT^*_\rho
M\otimes\Cal VM$, and the restriction of $\partial^*_\rho\Cal D$
defines an operator
$\Ga(\ker(\partial^*_\rho))\to\Ga(\im(\partial^*_\rho))$. We next
prove a property of this restriction, which is the main technical
input for what follows.
 \begin{lemma}\label{lemma3.2}
 Let $\Cal VM$ be a relative tractor bundle and let $\Cal D$ be a
 compressable operator defined on $\Om^k_\rho(M,\Cal VM)$. Further let
 $\pi_H:\Ga(\ker(\partial^*_\rho))\to\Ga(\Cal H_k(T^*_\rho M,\Cal
 VM))$ be the tensorial operator induced by the canonical projection
 to the Lie algebra homology bundle. Then we have

 (1) The restriction of $\partial^*_\rho\Cal D$ to
 $\Ga(\im(\partial^*_\rho))$ is injective.

 (2) The restriction of $\pi_H$ to $\ker(\partial^*_\rho\Cal
 D)\cap\Ga(\ker(\partial^*_\rho))$ is injective.
 \end{lemma}
 \begin{proof}
 (1) Suppose that $\ph\in\Ga(\im(\partial^*_\rho))$ is such that
   $\partial^*_\rho\Cal D(\ph)=0$, and suppose that $\ph$ is
   homogeneous of degree $\geq\ell$ for some $\ell$. From the
   construction of the natural filtrations it follows that for each
   $\ell$, the restriction of $\partial^*_\rho$ to the component of
   degree $\geq\ell$ in $\La^{k+1}T^*_\rho M\otimes\Cal VM$ maps onto
   the component of degree $\geq\ell$ in $\im(\partial^*_\rho)$,
   compare with Section 2.5 of \cite{part1}. Consequently, we can find
   a form $\ps\in\Om^{k+1}_\rho(M,\Cal VM)$, which is homogeneous of
   degree $\geq\ell$ such that $\ph=\partial^*_\rho\ps$. By Proposition
   \ref{prop2.4.2}, this implies that
   $\gr_\ell(\ph)=\underline{\partial}^*_\rho(\gr_\ell(\ps))$. On the
   other hand, we get
 $$
 0=\gr_\ell(\partial^*_\rho\Cal D(\ph))=\underline{\partial}^*_\rho\gr_\ell(\Cal
 D(\ph))=\underline{\partial}^*_\rho\partial_\rho(\gr_\ell(\ph)). 
 $$
 But since $\gr_\ell(\ph)\in\Ga(\im(\underline{\partial}^*_\rho))$, the
 latter expression coincides with $\square_\rho(\gr_\ell(\ph))$. Hence
 the Hodge decomposition in part (3) of Proposition \ref{prop2.4.2}
 implies that $\gr_\ell(\ph)=0$ and thus $\ph$ is homogeneous of degree
 $\geq\ell+1$. Iterating this argument finitely many times, we get
 $\ph=0$, which completes the proof of (1). 

 (2) By definition, the kernel of $\pi_H$ coincides with
 $\Ga(\im(\partial^*_\rho))\subset\Ga(\ker(\partial^*_\rho))$. But we
 have just shown that this subspace has zero intersection with
 $\ker(\partial^*_\rho\Cal D)$, which implies the claim.  
 \end{proof}

 \subsection{The splitting operator}\label{3.3}
 As a next step, we construct an operator from $\Ga(\Cal
 H_k(T^*M_\rho,\Cal VM))$ to $\ker(\partial^*_\rho\Cal
 D)\subset\Ga(\ker(\partial^*_\rho))$, which is right invariant to (the
 restriction of) $\pi_H$, thus proving that this restriction is a
 linear isomorphism. This so--called splitting operator can be
 constructed from polynomials in $\partial^*_\rho\Cal D$, which implies
 that it inherits nice properties from $\Cal D$. We put 
 $$
 \widetilde{\Bbb W}:=\im(\partial^*_\rho)\subset\ker(\partial^*_\rho):=\Bbb
 W\subset \La^k(\frak q_+/\frak p_+)\otimes\Bbb V,
 $$ and denote the filtration components of degree $\ell$ by $\Bbb
 W^\ell$ and $\widetilde{\Bbb W}^\ell=\Bbb W^\ell\cap\widetilde{\Bbb
   W}$, respectively. Consequently, for each $\ell$, the quotient
 $\gr_\ell(\widetilde{\Bbb W})=\widetilde{\Bbb W}^\ell/\widetilde{\Bbb
   W}^{\ell+1}$ is naturally a subspace of $\gr_\ell(\Bbb W)$. From
 Sections 2.4 and 2.6 of \cite{part1}, we see that the relative
 Kostant Laplacian acts on each of these spaces, it acts
 diagonalizably, and $\gr_\ell(\widetilde{\Bbb W})\subset\gr_\ell(\Bbb W)$
 coincides with the direct sum of the eigenspaces corresponding to
 non--zero eigenvalues. Let $a^\ell_1,\dots, a^\ell_{j_\ell}$ be the
 different non--zero eigenvalues which occur in homogeneity $\ell$.

 Now we take the corresponding induced bundles $\Cal WM$ and
 $\widetilde{\Cal W}M$ which are just the subbundles
 $\ker(\partial^*_\rho)$ and $\im(\partial^*_\rho)$ of $\La^kT^*_\rho
 M\otimes\Cal VM$. The corresponding filtration components give rise
 to smooth subbundles $\widetilde{\Cal W}^\ell M\subset\Cal W^\ell
 M\subset\Cal WM$. Now for each possible homogeneity $\ell$, we define
 an operator $S_\ell:\Ga(\Cal WM)\to \Ga(\Cal WM)$ by
 $$ S_\ell:=\frac{(-1)^{j_\ell}}{\prod_{r=1}^{j_\ell}
   a^\ell_r}\prod_{r=1}^{j_\ell}(\partial^*_\rho\Cal D-a^\ell_r\id).
 $$ 
 The basic properties of these operators are now easy to prove.
 \begin{lemma}\label{lemma3.3}
   For each $\ell$, the operator $S_\ell:\Ga(\Cal W)\to\Ga(\Cal W)$ is
   compatible with the natural filtration and satisfies $\pi_H\o
   S_\ell=\pi_H$, $\partial^*_\rho\Cal D\o
   S_\ell=S_\ell\o\partial^*_\rho\Cal D$, and
   $S_\ell(\Ga(\widetilde{\Cal W}^\ell))\subset\Ga(\widetilde{\Cal
     W}^{\ell+1})$.
 \end{lemma}
 \begin{proof} 
   By definition, $S_\ell$ is a polynomial in operator
   $\partial^*_\rho\Cal D$, so it commutes with $\partial^*_\rho\Cal
   D$. Moreovoer, since $\partial^*_\rho\Cal D$ is compatible with the
   natural filtration, the same holds for $S_\ell$. By definition
   $\pi_H\o \partial^*_\rho=0$, so $\pi_H\o (\partial^*_\rho\Cal
   D-a^\ell_r\id)=a^\ell_r\pi_H$ for each $r$, which immediately
   implies $\pi_H\o S_\ell=\pi_H$.

 To prove the last property, take $\ph\in\Ga(\widetilde{\Cal
   W}^\ell)$. Then we have already observed that $\partial^*_\rho\Cal
 D(\ph)\in\Ga(\widetilde{\Cal W}^\ell)$, and we compute
 $$ \gr_\ell(\partial^*_\rho\Cal
 D(\ph))=\underline{\partial}^*_\rho\gr_\ell(\Cal
 D(\ph))=\underline{\partial}^*_\rho\partial_\rho\gr_\ell(\ph).
 $$ 
Since $\gr_\ell\ph$ is a section of the subbundle
 $\im(\underline{\partial}^*_\rho)$ the last term coincides with
 $\square_\rho(\gr_\ell(\ph))$.  Now we can write $\ph$ as a finite sum
 of sections $\ph_i$ of $\Ga(\widetilde{\Cal W}^\ell)$ such that each
 $\gr_\ell(\ph_i)$ is a section of the bundle induced by one of the
 eigenspaces for $\square_\rho$. If $a^\ell_r$ is the corresponding
 eigenvalue, then the above computation shows that
 $\gr_\ell((\partial^*_\rho\Cal D-a^\ell_r\id)(\ph_i))=0$. Since the
 factors in the composition defining $S_\ell$ can be permuted
 arbitrarily, we see that $\gr_\ell(S_\ell(\ph_i))=0$. This implies
 $\gr_\ell(S_\ell(\ph))=0$ and thus $S_\ell(\ph)\in\Ga(\widetilde{\Cal
   W}^{\ell+1})$.
 \end{proof}

 Now we define the splitting operator
 $S:\Ga(\ker(\partial^*_\rho))\to\Ga(\ker(\partial^*_\rho))$ as the
 composition of the operators $S_\ell$ for the finitely many possible
 homogeneities $\ell$, which show up in
 $\ker(\partial^*_\rho)\subset\La^kT^*_\rho M\otimes\Cal VM$. Since
 the operators $S_\ell$ all are polynomials in $\partial^*_\rho\Cal
 D$, the order in which they are composed plays no role. 

 \begin{thm}\label{thm3.3}
 (1) The operator $S$ satisfies $\pi_H\o S=\pi_H$ and
   $\partial^*_\rho\Cal D\o S=0$, and its restriction to
   $\Ga(\im(\partial^*_\rho))$ vanishes identically. Thus it descends
   to an operator
 $$
 \Ga(\ker(\partial^*_\rho))/\Ga(\im(\partial^*_\rho))\cong \Ga(\Cal
 H_k(T^*_\rho M,\Cal VM))\to\ker(\partial^*_\rho\Cal
 D)\cap\Ga(\ker(\partial^*_\rho)),
 $$ which is right inverse to the tensorial projection $\pi_H$. In
 particular, $\pi_H$ restricts to a linear isomorphism on
 $\ker(\partial^*_\rho\Cal D)\cap\Ga(\ker(\partial^*_\rho))$.

 (2) For $\al\in \Ga(\Cal H_k(T^*_\rho M,\Cal VM))$ the form
 $\ph=:S(\al)$ is uniquely determined by $\partial^*_\rho(\ph)=0$,
 $\pi_H(\ph)=\al$, and $\partial^*_\rho\Cal D(\ph)=0$.

 (3) If the operator $\Cal D$ is such that $\partial^*_\rho\Cal D$
 belongs to a class of linear operators which is stable under forming
 polynomials, then also $S$ belongs to this class.
 \end{thm}
 \begin{proof}
 (1) Since $\pi_H\o S_\ell=\pi_H$ holds for each $\ell$ by Lemma
   \ref{lemma3.3}, we see that $\pi_H\o S=\pi_H$. If $\ph$ is a
   section of $\im(\partial^*)$, then $\ph$ is homogeneous of degree
   $\geq\ell$ for some $\ell$. Then Lemma \ref{lemma3.3} shows
   iteratively that $S_\ell(\ph)\in\Ga(\widetilde{\Cal W}^{\ell+1})$,
   $S_{\ell+1}(S_\ell(\ph))\in\Ga(\widetilde{\Cal W}^{\ell+2})$ and
   continuing up to the maximal possible homogeneity, we conclude
   that the composition of the $S_i$ for $i\geq\ell$ annihilates
   $\ph$. This of course implies $S(\ph)=0$, so $S$ vanishes on
   $\Ga(\im(\partial^*_\rho))$.

 Next Lemma \ref{lemma3.3} iteratively implies that $S$ commutes with
 $\partial^*_\rho\Cal D$. But since we have just seen that
 $S\o\partial^*_\rho=0$, this implies $\partial^*_\rho\Cal D\o
 S=0$. From this, the rest of (1) is evident.

 (2) By part (1), the form $\ph=S(\al)$ for $\al\in \Ga(\Cal
 H_k(T^*_\rho M,\Cal VM))$ has the claimed properties. Conversely, if
 $\ph\in\Om^k_\rho(M,\Cal VM)$ has the three properties, then
 $\partial^*_\rho\ph=0$ implies that we can form $\pi_H(\ph)\in
 \Ga(\Cal H_k(T^*_\rho M,\Cal VM))$. But then $\partial^*_\rho\Cal
 D(\ph)=0$ immediately implies $S_\ell(\ph)=\ph$ for all $\ell$ and
 hence $\ph=S(\pi_H(\ph))$. 

 (3) This is clear, since $S$ is given by a universal polynomial in
 $\partial^*_\rho\Cal D$.
 \end{proof}

 \subsection{The compressed operator}\label{3.5}
 Having the splitting operator at hand, it is now easy to complete the 
 relative version of the BGG construction. 
 \begin{definition}\label{def3.5}
   Given a compressable operator $\Cal D:\Om^k_\rho(M,\Cal
   VM)\to\Om^{k+1}_\rho (M,\Cal VM)$, the \textit{compression} of
   $\Cal D$, respectively the \textit{BGG--operator} induced by $\Cal
   D$, is the operator
 $$ 
D:\Ga(\Cal H_k(T^*_\rho M,\Cal VM))\to \Ga(\Cal H_{k+1}(T^*_\rho
 M,\Cal VM))
 $$ 
 defined by
 $$
 D(\al):=\pi_H(\Cal D(S(\al))),
 $$ where $S$ denotes the splitting operator associated to $\Cal
 D$.
 \end{definition}

 Notice that this definition makes sense since by Theorem \ref{thm3.3},
 $\Cal D(S(\al))$ is a section of the bundle $\ker(\partial^*_\rho)$,
 so $\pi_H$ can be applied to it. 

 We can easily prove that nice properties of a compressable operator
 carry over to the corresponding compressed operator. Moreover, the
 notion of a normal solution of a first BGG operator (see
 e.g.~\cite{polynomiality}) has a nice analog in general. 

 \begin{prop}\label{prop3.5}
   Let $\Cal D:\Om^k_\rho(M,\Cal VM)\to \Om^{k+1}_\rho(M,\Cal VM)$ be a
   compressable operator and $D:\Ga(\Cal H_k(T^*_\rho M,\Cal VM))\to
   \Ga(\Cal H_{k+1}(T^*_\rho M,\Cal VM))$ the corresponding compressed
   operator.

   (1) If $\Cal D$ and $\partial^*_\rho\Cal D$ belong to some class
   of operators which is stable under forming polynomials then also the
   compressed operator $D$ belongs to this class.

   (2) The projection $\pi_H$ maps $\ker(\Cal
   D)\cap\Ga(\ker(\partial^*_\rho))$ bijectively onto a subspace of
   $\ker(D)\subset\Ga(\Cal H_k(T^*_\rho M,\Cal VM))$.
 \end{prop}
 \begin{proof}
   (1) follows directly from part (3) of Theorem \ref{thm3.3}. 

 (2) Suppose that $\ph\in\Om^k_\rho(M,\Cal VM)$ satisfies
 $\partial^*_\rho(\ph)=0$ and $\Cal D(\ph)=0$. Then
 $\partial^*_\rho\Cal D(\ph)=0$, which by part (2) of Theorem
 \ref{thm3.3} implies $\ph=S(\pi_H(\ph))$. On the other hand, we also
 get $0=\Cal D(S(\pi_H(\ph)))$, which implies $\pi_H(\ph)\in\ker(D)$. 
 \end{proof}

 \begin{remark}\label{rem3.5}
   The construction of the compressed operator as a polynomial in the
   original operator is a major advantage compared to earlier
   constructions of BGG sequences. This is expressed by part (1) of
   the proposition, which shows that, while compressability is the
   only property of a linear operator required to apply the
   BGG-machinery, many nice properties of such an operator manifestly
   carry over to the corresponding BGG--operator.

   Most easily, if $\Cal D$ is a differential operator, then so is
   $D$. The condition also applies to the usual definition of an
   invariant differential operator, which requires a universal formula
   in terms of certain distinguished connections, their torsion and
   curvature. In this setting, we obtain the first construction of
   BGG--operators which are manifestly invariant. There are also
   notions of ``strong invariance'' to which this condition applies.

   A drawback of the construction is that in the construction of the
   splitting operator as a polynomial in the compressable operator,
   there is a lot of cancellation. The degree of the polynomial used
   to define $S$ roughly equals the number of $\frak q_0$--irreducible
   components in the representation $\im(\partial^*_\rho)\subset
   \La^k(\frak q_+/\frak p_+)\otimes\Bbb V$. In the important special
   case that $\Cal D$ is a first order differential operator, one
   might therefore expect that this number of components is the order
   of the splitting operator. However, as we shall see in Remark
   \ref{rem3.4} below, the order of the splitting operator basically
   is given by the length of the $\frak q$--invariant filtration on
   $\im(\partial^*_\rho)\subset \La^k(\frak q_+/\frak p_+)\otimes\Bbb
   V$, which is much smaller than the number of irreducible
   components.
 \end{remark}

 \subsection{An alternative construction}\label{3.4}
 There is an alternative construction for the splitting operators,
 which at the same time leads to a relative analog of a further
 element of the ``BGG--calculus'' as developed in
 \cite{Calderbank--Diemer}. This will be particularly useful in the
 case of sequences of compressable operators discussed below. We
 preferred to first present the direct construction from Section
 \ref{3.3} since it seems more intuitive to us.

 We continue using the notation from above, so $\widetilde{\Bbb
   W}=\im(\partial^*_\rho)\subset \La^k(\frak q_+/\frak
 p_+)\otimes\Bbb V$, $\widetilde{\Bbb W}^\ell\subset\widetilde{\Bbb
   W}$ denotes the filtration component of degree $\ell$ and the
 different (non--zero) eigenvalues of $\square_\rho$ on
 $\gr_\ell(\widetilde{\Bbb W})$ are denoted by
 $a^\ell_1,\dots,a^\ell_{j_\ell}$. Now we observe that by construction
 the operator $\partial^*_\rho\Cal D$ maps $\Ga(\widetilde{\Cal
   W}M)$ to itself and preserves each of the filtration components
 $\Ga(\widetilde{\Cal W}^\ell M)$. So for each homogeneity $\ell$, we
 can define an operator $\tilde Q^\ell$ on $\Ga(\widetilde{\Cal
   W}^\ell M)$ by
 $$ \tilde
 Q^\ell:=\textstyle\sum_{r=1}^{j_\ell}\tfrac{1}{a^\ell_r\prod_{s\neq
     r}(a^\ell_s-a^\ell_r)}\textstyle\prod_{s\neq
   r}(\partial^*_\rho\Cal D-a^\ell_s\id).
 $$

 \begin{lemma}\label{lemma3.4}
   For each homogeneity $\ell$ and each section
   $\ph\in\Ga(\widetilde{\Cal W}^\ell M)$, we have
 $$
 \ph-\partial^*_\rho\Cal D\tilde Q^\ell(\ph)\in\Ga(\widetilde{\Cal W}^{\ell+1}M).
 $$
 \end{lemma}
 \begin{proof}
   As in the proof of Lemma \ref{lemma3.3}, we can write a section
   $\ph\in\Ga(\widetilde{\Cal W}^\ell M)$ as a finite sum
   $\ph=\ph_1+\dots+\ph_{j_\ell}$ in such a way that for each $i$ the
   image $\gr_\ell(\ph_i)$ satisfies
   $\square_\rho(\gr_\ell(\ph_i))=a^\ell_i \ph_i$. But this implies
   that $\partial^*_\rho\Cal D(\ph_i)\in\Ga(\widetilde{\Cal W}^\ell
   M)$ and $\gr_\ell(\partial^*_\rho\Cal
   D(\ph_i))=a^\ell_i\gr_\ell(\ph_i)$. Now in the definition of
   $\tilde Q^\ell$, we can permute the factors in the composition of
   operators in each summand arbitrarily. This shows that $\prod_{s\neq
     r}(\partial^*_\rho\Cal D-a^\ell_s\id)$ maps each $\ph_s$ for
   $s\neq r$ to a section of $\widetilde{\Cal W}^{\ell+1}M$. On the
   other hand, applying this composition to $\ph_r$ and applying
   $\gr_\ell$, we obtain $\prod_{s\neq
     r}(a^\ell_r-a^\ell_s)\gr_\ell(\ph_r)$. Hence we conclude from the
   definition that $\tilde Q^\ell(\ph)\in\Ga(\widetilde{\Cal W}^\ell
   M)$ and $\gr_\ell(\tilde
   Q^\ell(\ph))=\sum_{r=1}^{j_\ell}\frac{1}{a^\ell_r}\gr_\ell(\ph_r)$. Together
   with the above observation on the action of $\partial^*_\rho\Cal
   D$, this implies the claim of the lemma.
 \end{proof}

 \begin{thm}\label{thm3.4}
 There is an operator
 $Q:\Ga(\im(\partial^*_\rho))\to\Ga(\im(\partial^*_\rho))$ which
 can be written as a universal polynomial in $\partial^*_\rho\Cal D$
 such that $\partial^*_\rho\Cal D\o Q=\id$ on $\Ga(\im(\partial^*_\rho))$.  
 \end{thm}
 \begin{proof}
 We construct $Q$ recursively. Denoting by $N$ the maximal possible
 homogeneity occurring in $\widetilde{\Bbb W}$, Lemma \ref{lemma3.4} shows
 that $Q^N:=\tilde Q^N$ has the property that $\ph-\partial^*_\rho\Cal
 DQ^N(\ph)\in\Ga(\widetilde{\Cal W}^{N+1}M)=\{0\}$ for each
 $\ph\in\Ga(\widetilde{\Cal W}^NM)$.

 Let us inductively assume that for some $\ell<N$ we have found an
 operator $Q^\ell$ on $\Ga(\widetilde{\Cal W}^\ell M)$ which is a universal
 polynomial in $\partial^*_\rho\Cal D$ and satisfies
 $\partial^*_\rho\Cal D\o Q^\ell=\id$ on $\Ga(\widetilde{\Cal W}^\ell M)$. Then
 for $\ph\in\Ga(\widetilde{\Cal
   W}^{\ell-1}M)$, we again invoke Lemma \ref{3.4} to conclude that 
 $\ph-\partial^*_\rho\Cal D\tilde Q^{\ell-1}(\ph)\in\Ga(\widetilde{\Cal
   W}^\ell M)$, and hence we can define 
 $$
 Q^{\ell-1}(\ph):=\tilde Q^{\ell-1}(\ph)-Q^\ell(\ph-
 \partial^*_\rho\Cal D\tilde Q^{\ell-1}(\ph)).
 $$

 Then we immediately conclude that $\partial^*_\rho\Cal
 DQ^{\ell-1}(\ph)=\ph$. This leads to an operator $Q$ with the
 required properties in finitely many steps.
 \end{proof}

 It is easy to describe the splitting operator $S$ in terms of $Q$:

 \begin{cor}\label{cor3.4}
 Viewed as an operator on $\Ga(\ker(\partial^*))$, the splitting
 operator $S$ from Theorem \ref{thm3.3} is given by
 $S=\id-Q\partial^*_\rho\Cal D$.
 \end{cor}
 \begin{proof}
 Let us put $\tilde S:=\id-Q\partial^*_\rho\Cal D$. Since $Q$ has
 values in $\Ga(\im(\partial^*_\rho))$ we see that $\tilde S$ maps the
 space $\Ga(\ker(\partial^*_\rho))$ to itself and that $\pi_H\o\tilde
 S=\pi_H$. Moreover, from Lemma \ref{lemma3.4}, we conclude
 immediately that $\partial^*_\rho\Cal D\o\tilde S=0$. This already
 shows that $S-\tilde S$ maps $\Ga(\ker(\partial^*_\rho))$ to
 $\Ga(\im(\partial^*_\rho))\cap\ker(\partial^*_\rho\Cal D)$, and this
 intersection is trivial by part (1) of Lemma \ref{lemma3.2}.
 \end{proof}

 \begin{remark}\label{rem3.4} 
While the operator $Q$ (in the case $\frak p=\frak g$) was the crucial
ingredient for the construction of BGG sequences in
\cite{Calderbank--Diemer}, our construction as a polynomial in
$\partial^*\Cal D$, which gives a manifestly invariant operator in
case $\Cal D$ is invariant, is new even for this special case. 
Similarly to the case of $S$ discussed in Remark \ref{rem3.5}, also
the construction of $Q$ as a polynomial involves a lot of
cancellation. There is an alternative construction for $Q$ (and thus
via Corollary \ref{cor3.4} also for $S$), however, which needs much
fewer composition factors. This is a relative version of the
construction of \cite{Calderbank--Diemer} for ordinary BGG--sequences.

One first fixes a splitting of the natural filtration on the bundle
$\im(\partial^*_\rho)$ and thus an identification with its associated
graded bundle. A conceptual way to obtain such a splitting is via the
choice of a Weyl--structure, see section 5.1 of \cite{book} or
\cite{Weyl}. Now $\gr(\im(\partial^*_\rho))$ is a subbundle of
$\gr(\La^kT^*_\rho M\otimes\Cal VM)$, which is invariant under the
bundle map $\square_\rho$ from Proposition \ref{prop2.4.2}. On this
subbundle $\square_\rho$ coincides with
$\underline{\partial}^*_\rho\o\partial_\rho$ and it is invertible, so
we may form $(\square_\rho)^{-1}$. Via the chosen isomorphism, we can
now define bundle maps $\square_\rho$ and $(\square_\rho)^{-1}$ on
$\im(\partial^*_\rho)$ and we use the same symbols to denote the
resulting tensorial operators.

   In terms of these operators defined on $\Ga(\im(\partial^*_\rho))$,
   we can write
 $$ \partial^*_\rho\Cal D=\square_\rho(\id+
   \square_\rho^{-1}(\partial^*_\rho\Cal D-\square_\rho)).
 $$ 

 From the construction it is clear that $\partial^*_\rho\Cal
 D-\square_\rho$ raises the filtration degree by one, i.e.~in the
 notation of Sections \ref{3.3} and \ref{3.4} it maps each of the
 spaces $\Ga(\widetilde{\Cal W}^\ell M)$ to $\Ga(\widetilde{\Cal
   W}^{\ell+1}M)$. This remains true after composing
 $\square_\rho^{-1}$ and it of course implies that the resulting
 operator is nilpotent of degree $N+1$ where $N$ describes the length
 of the natural filtration of the bundle
 $\im(\partial^*_\rho)$. Adding the identity to this nilpotent
 operator, the result is invertible, and there is a usual Neumann
 series for the inverse, which actually is a finite sum by
 nilpotency. Thus, one can construct an inverse of
 $\partial^*_\rho\Cal D$ on $\Ga(\im(\partial^*_\rho))$ as
 $$
 \left(\textstyle\sum_{k=0}^{N+1}(-1)^k(\square_{\rho}^{-1}(\partial^*_\rho\Cal
 D-\square_\rho))^k\right)\o\square^{-1}_\rho .
 $$
 As discussed above, this gives a much smaller number of factors in a
 composition that the construction in \ref{3.4}. For example, if $\Cal
 D$ is a first order differential operator, we conclude that $Q$ is a
 differential operator of order at most $N+1$, so $S$ is a differential
 operator of order at most $N+2$.

 The disadvantage of this construction is that one has to use a
 non--natural tensorial operation, which makes it much more difficult
 to see that nice properties of $\Cal D$ carry over to $D$.  In
 particular, this applies to concepts of strong invariance.  One
 solution to this problem is provided by the original construction in
 \cite{CSS-BGG} in terms of semi--holonomic jet modules, which however
 is significantly more complicated.
 \end{remark}

 \subsection{Sequences of compressable operators}\label{3.6}
 For the next step we have to assume that rather than a single
 compressable operator, we have a whole sequence $\Cal
 D_k:\Om^k_\rho(M,\Cal VM)\to\Om^{k+1}_\rho(M,\Cal VM)$ of such
 operators. In this case, we can rephrase what we have done so far in
 a nice way, using the Laplacians associated to the sequence, which
 are defined as follows.

 \begin{definition}\label{def3.6}
 Given a sequence $\Cal D_k:\Om^k_\rho(M,\Cal
 VM)\to\Om^{k+1}_\rho(M,\Cal VM)$ of compressable operators, we define
 the associated Laplacians 
 $$
 \square^{\Cal D}=\square^{\Cal D}_k:\Om^k_\rho(M,\Cal VM)\to \Om^k_\rho(M,\Cal VM)
 $$
 by $\square^{\Cal D}_k:=\partial^*_\rho\o\Cal D_k+\Cal
 D_{k-1}\o\partial^*_\rho$. 
 \end{definition}

 \begin{prop}\label{prop3.6}
   (1) If $\ph\in\Om^k_\rho(M,\Cal VM)$ satisfies $\square^{\Cal
     D}_k(\ph)=0$, then $\partial^*_\rho(\ph)=0$. Hence
   $\ker(\square^{\Cal D}_k)\subset\Om^k_\rho(M,\Cal VM)$ coincides
   with $\ker(\partial^*_\rho\Cal D_k)\cap\Ga(\ker(\partial^*_\rho))$.

 (2) On $\Ga(\ker(\partial^*_\rho))$, the operator $\square^{\Cal
     D}_k$ coincides with $\partial^*_\rho\Cal D_k$, so in the
   definitions of the operators $S_\ell$ in Section \ref{3.3} and
   $\tilde Q^j $ in Section \ref{3.4}, in the statements of Theorems
   \ref{thm3.3} and \ref{thm3.4} and of Lemma \ref{lemma3.4}, and in
   the construction of Remark \ref{rem3.4}, one may always replace
   $\partial^*_\rho\Cal D_k$ by $\square^{\Cal D}_k$.
 \end{prop}
 \begin{proof}
 (1) If $0=\square^{\Cal D}_k(\ph)$, then applying $\partial^*_\rho$
   and using $\partial^*_\rho\o\partial^*_\rho=0$, we get
   $0=\partial^*_\rho\Cal D_k\partial^*_\rho(\ph)$. But from part (1)
   of Lemma \ref{lemma3.2} we know that $\partial^*_\rho\Cal D_k$ is
   injective on $\Ga(\im(\partial^*_\rho))$. Thus
   $\partial^*_\rho\ph=0$ and hence also $\partial^*_\rho\Cal
   D_k(\ph)=0$. 

 (2) now follows immediately from
   $\partial^*_\rho\o\partial^*_\rho=0$. 
 \end{proof}

 We can now look at conditions related to the operators $\Cal D_k$
 forming a complex. 

 \begin{thm}\label{thm3.6}
 Consider a sequence $\Cal D_k:\Om^k_\rho(M,\Cal VM)\to
 \Om^{k+1}_\rho(M,\Cal VM)$ of compressable operators. Then we have

 (1) Let $Q$ be the operator on
 $\Ga(\im(\partial^*_\rho))\subset\Om^{k-1}_\rho (M,\Cal VM)$
 constructed in Theorem \ref{thm3.4}. Then for each
 $\ph\in\Om^k_\rho(M,\Cal VM)$, we have $\partial^*_\rho(\ph-\Cal
 D_{k-1}Q\partial^*_\rho(\ph))=0$.

 (2) If $\Cal D_k\o\Cal D_{k-1}=0$, then $D_k\o D_{k-1}=0$. If in
 addition $\Cal D_{k+1}\o\Cal D_k=0$, then the splitting operator
 induces a surjective linear map,
 $$
 \ker(D_k)/\im(D_{k-1}) \to \ker(\Cal D_k)/\im(\Cal D_{k-1}),
 $$ 
 which is a linear isomorphism provided that also $\Cal D_{k-1}\o\Cal
 D_{k-2}=0$. 
 \end{thm}
 \begin{proof}
 (1) follows immediately from the fact that $\partial^*_\rho\Cal
   D_{k-1}Q=\id$ on $\Ga(\im(\partial^*_\rho))$ which we proved in
   Theorem \ref{thm3.4}.

 (2) Suppose that $\Cal D_k\o\Cal D_{k-1}=0$, take a section
   $\al\in\Ga(\Cal H_{k-1}(T^*_\rho M,\Cal VM))$ and consider
   $\ph:=\Cal D_{k-1}(S(\al))$. Then by definition,
   $\partial^*_\rho(\ph)=0$ and $\pi_H(\ph)=D_{k-1}(\al)$. By
   assumption $\Cal D_k(\ph)=0$, so part (2) of Theorem \ref{thm3.3} shows
   that $\ph=S(\pi_H(\ph))$. Thus we get $\Cal D_{k-1}\o S_{k-1}=S_k\o
   D_{k-1}$ and hence 
 $$
 D_k\o D_{k-1}=\pi_H\o\Cal D_k\o S_k\o D_{k-1}=\pi_H\o\Cal D_k\o\Cal
 D_{k-1}\o S_{k-1} =0. 
 $$ 
 Now suppose that in addition $\Cal D_{k+1}\o\Cal D_k=0$, let
 $\al\in\Ga(\Cal H_k(T^*_\rho M,\Cal VM))$ be such that
 $D_k(\al)=0$ and consider $\ph:=S_k(\al)$. Then from above we see that
 $\Cal D_k(\ph)=S(D_k(\al))=0$. Moreover, if $\al=D_{k-1}(\be)$
 for some $\be\in \Ga(\Cal H_{k-1}(T^*_\rho M,\Cal VM))$, then
 $S_k(\al)=\Cal D_{k-1}(S_{k-1}(\be))$. This shows that $S_k$ induces a
 well defined map $\ker(D_k)/\im(D_{k-1})\to \ker(\Cal D_k)/\im(\Cal
 D_{k-1})$ in cohomology. 

 Supposing that $\ph\in\Om^k_\rho(M,\Cal VM)$ satisfies $\Cal
 D_k(\ph)=0$, we can use part (1) to find $\ps\in\Om^{k-1}_\rho(M,\Cal
 VM)$ such that $\tilde\ps:=\ph+\Cal D_{k-1}(\ps)$ satisfies
 $\partial^*_\rho(\tilde\ph)=0$. By assumption $\Cal
 D_k(\tilde\ph)=\Cal D_k(\ph)=0$ and thus
 $\tilde\ph=S(\pi_H(\tilde\ph))$ and $D_k(\pi_H(\tilde\ph))=0$. This
 implies surjectivity of the map in cohomology.

 So let us finally assume that $\Cal D_{k-1}\o\Cal D_{k-2}=0$ and that
 we have given $\al\in\Ga(\Cal H_k(T^*_\rho M,\Cal VM))$ such that
 $D_k(\al)=0$ and $S_k(\al)=\Cal D_{k-1}(\ph)$ for some
 $\ph\in\Om^{k-1}_\rho(M,\Cal VM)$. Then again by part (1), we can
 find an element $\ps\in\Om^{k-2}_\rho(M,\Cal VM)$ such that
 $\tilde\ph:=\ph+\Cal D_{k-2}\ps$ satisfies
 $\partial^*_\rho(\tilde\ph)=0$. By assumption $\Cal
 D_{k-1}(\tilde\ph)=\Cal D_{k-1}(\ph)=S(\al)$, so in particular
 $\partial^*_\rho\Cal D_{k-1}(\tilde\ph)=0$ and hence
 $\tilde\ph=S(\pi_H(\tilde\ph))$. Moreover,
 $$
 D_{k-1}(\pi_H(\tilde\ph))=\pi_H(\Cal D_{k-1}(\tilde\ph)))=\pi_H(S(\al))=\al,
 $$ 
which implies injectivity of the map in cohomology induced by $S$.
 \end{proof}

 \section{The relative twisted exterior derivative}\label{4}
 In this section, we construct a sequence of compressable first order
 differential operators on relative forms with values in an arbitrary
 relative tractor bundle, which has strong naturality properties. Thus
 we can run the BGG machinery as developed in Section \ref{3} on this
 sequence to obtain invariant differential operators defined on the
 relative Lie algebra homology bundles. In view of the discussion in
 Section \ref{3.1}, this also gives a description of all compressable
 operators.

 \subsection{Definition of the relative twisted exterior
   derivative}\label{4.1}
 Given a relative tractor bundle $\Cal VM$, we start by defining an
 operator
 $$
 \tilde d^{\Cal V}:\Ga(\La^k\Cal A^*_\rho M\otimes\Cal VM)\to
 \Ga(\La^{k+1}\Cal A^*_\rho M\otimes\Cal VM).
 $$ Here $\Cal A^*_\rho M$ is the dual to the relative adjoint tractor
 bundle introduced in Section \ref{2.2}. By definition, the
 representation $\Bbb V$ inducing $\Cal VM$ is the restriction to $Q$
 of a representation of $P$, which in addition has the property that
 the ideal $\frak p_+\subset\frak p$ acts trivially in the
 infinitesimal representation. Thus we can view $\Bbb V$ as a
 representation of the Lie algebra $\frak p/\frak p_+$. Hence there is
 the standard Lie algebra cohomology differential, compare with
 Section 2.3 of \cite{part1}, which, for each $k$, is a linear map
 $$ 
 \partial_{\frak p/\frak p_+}:\La^k(\frak p/\frak p_+)^*\otimes\Bbb V\to
 \La^{k+1}(\frak p/\frak p_+)^*\otimes\Bbb V.
 $$ 
In the picture of multilinear maps, this differential is given by
 \begin{equation}
   \label{part-def}
   \begin{aligned}
 \partial \ph(A_0&,\dots,A_k):=\textstyle\sum_{i=0}^k(-1)^iA^i\cdot
 \ph(A_0,\dots,\widehat{A_i},\dots,A_k)\\
 &+\textstyle\sum_{i<j}(-1)^{i+j}\ph([A_i,A_j],A_0,\dots,\widehat{A_i},\dots,
 \widehat{A_j},\dots,A_k),
 \end{aligned}
 \end{equation}
 for $A_0,\dots,A_k\in\frak p/\frak p_+$. This map is evidently
 $Q$--equivariant, so it induces a bundle map between the
 corresponding associated bundles. We denote this bundle map as well
 as the corresponding tensorial operator on sections by the same
 symbol.

 On the other hand, applying the relative fundamental derivative from
 Section \ref{2.2} to $\ph\in\Ga(\La^k\Cal A^*_\rho M\otimes\Cal VM)$,
 we obtain
 $$
 D^\rho\ph\in\Ga(\Cal A^*_\rho M\otimes\La^k\Cal A^*_\rho
 M\otimes\Cal VM).
 $$ 
 Then we define $\tilde d_1^{\Cal V}\ph$ by
 \begin{equation}\label{td1def}
   \tilde d_1^{\Cal V}\ph(s_0,\dots,s_k):=\tsum_{i=0}^k(-1)^i(D^\rho_{s_i}\ph)
   (s_0,\dots,\widehat{s_i},\dots,s_k).
 \end{equation} 
 Observe that this is alternating in all entries, so $\tilde d_1^{\Cal
   V}\ph\in\Ga(\La^{k+1}\Cal A^*_\rho M\otimes\Cal VM)$. Having this at
 hand, we finally put $\tilde d^{\Cal V}\ph:=\tilde d_1^{\Cal
   V}\ph+\partial_{\frak p/\frak p_+}\ph$.

 Note that the formula  \eqref{td1def} for $\tilde d_1^{\Cal V}\ph$ 
 can be further expanded using the naturality properties of $D^\rho$
 derived in Proposition \ref{3.2}. These imply that for
 $s,t_1,\dots,t_k\in\Ga(\Cal A_\rho M)$ we have
 $$
 (D^\rho_s\ph)(t_1,\dots,t_k)=D^\rho
 _s(\ph(t_1,\dots,t_k))-\tsum_{i=1}^r\ph(t_1,\dots,D^\rho_st_i,\dots,t_k).
 $$ 
 On the other hand, we can explicitly express $\partial_{\frak p/\frak
   p_+}\ph(s_0,\dots,s_k)$ using the definition in formula
 \eqref{part-def}.  We only have to replace the action $\cdot:(\frak
 p/\frak p_+)\x\Bbb V\to\Bbb V$ by the induced bundle map $\bullet:\Cal
 A_\rho M\x\Cal VM\to\Cal VM$ and the Lie bracket on $\frak p/\frak
 p_+$ by the (induced) algebraic bracket $\{\ ,\ \}$ on (sections of)
 $\Cal A_\rho M$.

 As we have noted in Section \ref{2.2}, the relative tangent bundle
 $T_\rho M$ can be identified with the quotient $\Cal A_\rho M/\Cal
 A^0_\rho M$, so dually $T^*_\rho M$ is a subbundle of $\Cal A^*_\rho
 M$. Consequently, we can view $\Om^k_\rho (M,\Cal VM)$ as a subspace
 of $\Ga(\La^k\Cal A^*_\rho M\otimes\Cal VM)$. The elements of this
 subspace can evidently be characterized by the fact that they vanish
 upon insertion of a single section of the subbundle $\Cal A^0_\rho
 M\subset\Cal A_\rho M$. Using this, we can now prove:
 \begin{thm}\label{thm4.1}
   The operators $\tilde d^{\Cal V}$ restrict to first order invariant 
   differential operators
 $$
 d^{\Cal V}:\Om^k_\rho (M,\Cal VM)\to\Om^{k+1}_\rho (M,\Cal VM),
 $$ 
 which are compressable in the sense of Definition \ref{def3.1}.

 Via the construction in Section \ref{3}, we thus obtain a relative
 BGG--sequence
 \begin{equation}\label{rel-BGG}
 D_k:\Ga(\Cal H_k(T^*_\rho M,\Cal VM))\to\Ga(\Cal H_{k+1}(T^*_\rho
 M,\Cal M))\quad k=0,\dots,\dim(\frak q_+/\frak p_+)-1 
 \end{equation}
 of invariant differential operators.  
 \end{thm}
 \begin{proof}
   We first show that if $\ph$ vanishes upon insertion of one section
   of the subbundle $\Cal A^0_\rho M\subset\Cal A_\rho M$, then the same is
   true for $\tilde d^{\Cal V}\ph$. Since we know that $\tilde d^{\Cal
     V}\ph$ is alternating, we may assume that $s_0\in\Ga(\Cal A^0_\rho
   M)$ and prove vanishing of $(\tilde d^{\Cal V}\ph)(s_0,\dots,s_k)$
   for arbitrary $s_1,\dots,s_k\in\Ga(\Cal A_\rho M)$. By part (3) of
   Proposition \ref{prop2.2}, $(D^\rho_{s_i}\ph)$ lies in
   $\Om^k_\rho(M,\Cal VM)$ for each $i$ and thus vanishes upon
   insertion of $s_0$. Hence we conclude that
 $$
 \tilde d_1^{\Cal V}(s_0,\dots,s_k)=(D^\rho_{s_0}\ph)(s_1,\dots,s_k),
 $$ 
 and we can expand this as noted above. Since $s_0$ is a section of
 $\Cal A^0_\rho M\subset\Cal A_\rho M$, part (1) of Proposition
 \ref{prop2.2} shows that $D^\rho_{s_0}$ coincides with the negative
 of the algebraic action by $s_0$ on the appropriate bundle. This
 action is $\bullet$ on $\Cal VM$ and the adjoint action via $\{\ ,\
 \}$ on $\Cal A_\rho M$, so we see that
 $$ \tilde d_1^{\Cal V}\ph(s_0,\dots,s_k)=-s_0\bullet
 \ph(s_1,\dots,s_k)+\textstyle\sum_{i=1}^k\ph(s_1,\dots,\{s_0,s_i\},\dots,s_k).
 $$ Rewriting the last term as
 $(-1)^{i+1}\ph(\{s_0,s_i\},s_1,\dots,\widehat{s_i},\dots,s_k)$, we
 see that this coincides with $-\partial_{\frak p/\frak
   p_+}\ph(s_0,\dots,s_k)$ up to terms in which $s_0$ is inserted
 directly into $\ph$ and which thus vanish by assumption. Hence we
 obtain the operators $d^{\Cal V}$ as claimed.

 To prove compressability, assume that $\ph\in\Om^k_\rho (M,\Cal VM)$
 is homogeneous of degree $\geq\ell$ and take sections $s_j\in\Ga(\Cal
 A^{i_j}_\rho M)$ with $i_j<0$ for $j=0,\dots,k$. To prove that
 $d^{\Cal V}$ preserves the natural filtration, we have to show that
 for each such choice we have
 $$
 (\tilde d^{\Cal V}\ph)(s_0,\dots,s_k)\in\Ga(\Cal V^{i_0+\dots+i_k+\ell}M).
 $$ 
 By naturality of the fundamental derivative,
 $D^\rho_{s_j}\ph\in\Om^k_\rho(M,\Cal VM)$ is homogeneous of degree
 $\geq\ell$. Hence
 $(D^\rho_{s_j}\ph)(s_0,\dots,\widehat{s_j},\dots,s_k)$ lies in the
 filtration component of $\Cal VM$ of degree
 $$
 i_0+\dots+\widehat{i_j}+\dots+i_k+\ell>i_0+\dots+i_k+\ell.
 $$ 
 So we conclude that $\tilde d_1^{\Cal V}$ is not only filtration
 preserving but also does not contribute to the action on the
 associated graded.

 On the other hand, both the action $\frak p/\frak p_+\x\Bbb V\to\Bbb
 V$ and the Lie bracket on $\frak p/\frak p_+$ are $Q$--homomorphisms
 and thus are homogeneous of degree zero for the grading element of
 $\frak q$. Together with the formula for the Lie algebra differential
 in \eqref{part-def}, this implies that $\partial_{\frak p/\frak
   p_+}\ph$ is filtration homogeneous of the same degree as $\ph$. This
 implies that $d^{\Cal V}$ is filtration preserving and the induced
 operator on sections of the associated graded bundle coincides with
 the one induced by $\partial_{\frak p/\frak p_+}$. From the
 definition in Section 2.3 of \cite{part1} it is evident that this is
 the bundle map induced by $\partial_\rho$, so compressability follows.
 \end{proof}

 \subsection{The square of the relative twisted exterior
   derivative}\label{4.2}
 Having constructed relative BGG--sequences, we next move to the
 question when we obtain complexes or even resolutions of some
 sheaves. The first step towards this is computing the composition
 $d^{\Cal V}\o d^{\Cal V}$. In the case of non--vanishing torsion, this
 result is new even for standard BGG sequences. It is based on the
 naturality properties of the fundamental derivative, which are very
 well understood, but some care is needed in the computations.

 From the properties of the inducing Lie algebra cohomology
 differential we conclude that $\partial_{\frak p/\frak
   p_+}\o\partial_{\frak p/\frak p_+}=0$, and thus the composition
 $d^{\Cal V}\o d^{\Cal V}$ is induced by
 $$ 
 \tilde d^{\Cal V}\o\tilde d^{\Cal V}=\tilde d_1^{\Cal V}\o\tilde
 d_1^{\Cal V}+\tilde d_1^{\Cal V}\o\partial_{\frak p/\frak
   p_+}+\partial_{\frak p/\frak p_+}\o\tilde d_1^{\Cal V}.
 $$
 We start by computing the sum of the last two terms in this formula:

 \begin{lemma}\label{lemma4.2}
   For $\ph\in\Ga(\La^{k-1}\Cal A_{\rho}^*\otimes\Cal VM)$ and
   $s_0,\dots,s_k\in\Ga(\Cal A_\rho M)$, we can express $\left(\tilde
     d_1^{\Cal V}(\partial_{\frak p/\frak p_+}\ph)+\partial_{\frak
       p/\frak p_+}(\tilde d_1^{\Cal V}\ph)\right)(s_0,\dots,s_k)$ as
 $$ 
 \textstyle\sum_{i<j}(-1)^{i+j}(D_{\{s_i,s_j\}}\ph)
 (s_0,\dots,\widehat{s_i},\dots,\widehat{s_j},\dots,s_k).
 $$
 \end{lemma}
 \begin{proof}
 First we observe that by definition
 $$
 \tilde d_1^{\Cal V}\partial_{\frak p/\frak p_+}\ph(s_0,\dots,s_k)=
 \textstyle\sum_{i=0}^r(D^\rho_{s_i}(\partial_{\frak
   p/\frak p_+}\ph))(s_0,\dots,\widehat{s_i},\dots,s_k).
 $$
 Now since $\partial_{\frak p/\frak p_+}$ is a natural bundle map
 between relative natural bundles, part (3) of Proposition \ref{2.2}
 implies that $D^\rho_{s_i}(\partial_{\frak p/\frak
   p_+}\ph)=\partial_{\frak p/\frak p_+}(D^\rho_{s_i}\ph)$. Hence we
 may write $\tilde d_1^{\Cal V}\partial_{\frak p/\frak
   p_+}\ph(s_0,\dots,s_k)$ as
 \begin{equation}\label{Dpartial}
   \textstyle \sum_{i=0}^k(-1)^i(\partial_{\frak
     p/\frak p_+}(D^\rho_{s_i}\ph))(s_0,\dots,\widehat{s_i},\dots,s_k).
 \end{equation}

 On the other hand, we can compute $(\partial_{\frak p/\frak p_+}(\tilde
 d_1^{\Cal V}\ph))(s_0,\dots,s_k)$ as
 \begin{equation}\label{partialD}
 \begin{aligned} 
   \textstyle\sum_{i=0}^k(-1)^i&s_i\bullet (\tilde d_1^{\Cal
     V}\ph)(s_0,\dots,\widehat{s_i},\dots,s_k)+\\
   &\textstyle\sum_{i<j}(-1)^{i+j}(\tilde d_1^{\Cal V}\ph)
   (\{s_i,s_j\},s_0,\dots,\widehat{s_i},\dots,\widehat{s_j},\dots,s_k).
 \end{aligned}
 \end{equation}
 Inserting the definition of $d_1^{\Cal V}$ in the first sum in
 \eqref{partialD}, we get a sum over all $i\neq j$ of terms of the form
 $s_i\bullet ((D^\rho_{s_j}\ph)(s_0,\dots,s_k))$ with $s_i$ and $s_j$
 omitted between $s_0$ and $s_k$. The sign of this term is $(-1)^{i+j}$
 if $j<i$ and $(-1)^{i+j+1}$ if $j>i$. This is exactly the opposite of
 the sign with which the same terms occur when inserting the definition
 of $\partial_{\frak p/\frak p_+}$ in \eqref{Dpartial}.

 Next, we insert the definition of $\tilde d_1^{\Cal V}$ in the second
 sum in \eqref{partialD}. On the one hand, this gives
 $\sum_{i<j}(-1)^{i+j}(D^\rho_{\{s_i,s_j\}}\ph)(s_0,\dots,\widehat{s_i},\dots,
 \widehat{s_j},\dots,s_k)$. On the other hand, for each $\ell$ different
 from $i$ and $j$, we obtain a summand of the form
 $$
 (D_{s_\ell}\ph)(\{s_i,s_j\},s_0,\dots,s_k),
 $$
 where between $s_0$ and $s_k$ the entries $s_i$, $s_j$, and $s_\ell$ are
 omitted. This term comes with a sign $(-1)^{i+j+\ell+1}$ if $\ell<i$ or
 $\ell>j$ and with a sign $(-1)^{i+j+\ell}$ if $i<\ell<j$. Again, this sign is
 opposite to the one with which the same term occurs after inserting
 the definition of $\partial_{\frak p/\frak p_+}$ in \eqref{Dpartial},
 and the result follows.
 \end{proof}

 Using this, we can completely compute $(d^{\tilde V})^2$.

 \begin{thm}\label{thm4.2}
   Consider $\ph\in\Om^{k-1}_\rho(M,\Cal VM)\subset\Ga(\La^{k-1}\Cal
   A_\rho^*M\otimes\Cal VM)$ then $d^{\Cal V}(d^{\Cal V}\ph)$ is
   induced by the section of $\La^{k+1}\Cal A_\rho^*M\otimes\Cal VM$
   which maps $s_0,\dots,s_k$ to
 $$
 \textstyle\sum_{i<j}(-1)^{i+j}(D_{\ka(\Pi(s_i),\Pi(s_j))}\ph)
 (s_0,\dots,\widehat{s_i},\dots,
 \widehat{s_j},\dots,s_r), 
 $$
 where $\ka\in\Om^2(M,\Cal AM)$ is the curvature of the geometry.
 \end{thm}
 \begin{proof}
   For any $\ell$ consider the complete alternation defined by
 $$
 \Alt_\ell\ps(s_1,\dots,s_\ell)=\tfrac1{\ell!}\sum_{\si\in\frak
   S_\ell}\operatorname{sgn}(\si)\ps(s_{\si_1},\dots,s_{\si_\ell}).
 $$ This can be viewed as a natural bundle map on various bundles, in
 particular as mapping $\Cal A_\rho^*M\otimes\La^{\ell-1}\Cal
 A_\rho^*M\otimes\Cal VM$ to $\La^\ell\Cal A_\rho^*M\otimes\Cal
 VM$. Now our definition of $\tilde d_1^{\Cal V}$ can be recast as
 $\tilde d_1^{\Cal V}\ph=k\Alt_k(D\ph)$ since $\ph$ has degree $k-1$
 and is already alternating in all its entries. But then by definition
 \begin{align*}
 \tilde d_1^{\Cal V}(\tilde d_1^{\Cal
   V}\ph)&(s_0,\dots,s_k)=k\tsum_{i=0}^k(-1)^i(D^\rho_{s_i}(\Alt_kD^{\rho}\ph))
 (s_0,\dots,\widehat{s_i},\dots,s_k)=\\
 &k\tsum_{i=0}^k(-1)^i\Alt_k(D^\rho _{s_i}D^\rho\ph)
 (s_0,\dots,\widehat{s_i},\dots,s_k)=\\
 &(k+1)k\Alt_{k+1}(D^\rho D^\rho \ph)(s_0,\dots,s_k).
 \end{align*}
 Now $D^\rho D^\rho \ph$ is a section of $\otimes^2\Cal
 A_\rho^*M\otimes\La^{k-1}\Cal A_\rho^*M\otimes\Cal VM$, so in forming
 the alternation, we do not have to permute the last $r-1$
 entries. Hence we can express $\tilde d_1^{\Cal V}(\tilde d_1^{\Cal
   V}\ph)(s_0,\dots,s_k)$ as 
 $$
 \tsum_{i<j}(-1)^{i+j+1}(D^\rho D^\rho\ph (s_i,s_j)-D^\rho
 D^\rho\ph(s_j,s_i))(s_0,\dots,\widehat{s_i},\dots,\widehat{s_j},\dots,s_k). 
 $$
 Now $D^\rho\ph$ is induced by the restriction of $D\ph$ to a natural
 subbundle, whose sections are preserved by a fundamental
 derivative. Thus, the usual Ricci identity from Proposition 1.5.9 of
 \cite{book} implies that
 $$
 D^\rho D^\rho\ph (s_i,s_j)-D^\rho
 D^\rho\ph(s_j,s_i)=-D_{\ka(\Pi(s_i),\Pi(s_j))}\ph+D^\rho_{\{s_i,s_j\}}\ph. 
 $$
 Together with Lemma \ref{lemma4.2}, this implies the claim.
 \end{proof}

 \subsection{An alternative description}\label{4.3} 
To prove that a standard BGG--sequence is a resolution, one usually
relates it to a twisted de--Rham resolution. As a next step towards
relative versions of such results, we derive a description of the
twisted relative exterior derivative which is closer to the standard
analogs of the exterior derivative on bundle valued differential
forms. To do this, we first need an analog of tractor connections,
which we can obtain using the relative fundamental
derivative. Consider the relative twisted exterior derivative on
$\Om^0_\rho(M,\Cal VM)=\Ga(\Cal VM)$. Viewed as an operator $\Ga(\Cal
A_\rho M)\x\Ga(\Cal VM)\to\Ga(\Cal VM)$ this is induced by
$(s,\si)\mapsto D^\rho_s\si+s\bullet\si$. Thus from part (1) of
Proposition \ref{prop2.2}, we conclude that the induced operator
$\Ga(T_\rho M)\x\Ga(\Cal VM)\to\Ga(\Cal VM)$ also satisfies a Leibniz
rule. Hence it defines a partial connection, which is called the
(normal) \textit{relative tractor connection} on the relative tractor
bundle $\Cal VM$ and denoted by $\nabla^{\rho,\Cal V}$.

A linear connection on a vector bundle can be coupled to the exterior
derivative to obtain an operator on differential forms with values in
that vector bundle. This has an analog for partial connections,
provided that the subbundle of the tangent bundle in question is
involutive, see Section \ref{4.4} below. In our setting, the relative
tangent bundle $T_\rho M$ is not involutive in general, but we can
overcome this problem by using a modification of the Lie bracket of
vector fields.
 \begin{prop}\label{prop4.3}
   The bilinear operator $\Ga(\Cal A_\rho M)\x\Ga(\Cal A_\rho
   M)\to\Ga(\Cal A_\rho M)$ defined by $(s_1,s_2)\mapsto
   D^\rho_{s_1}s_2-D^\rho_{s_2}s_1+\{s_1,s_2\}$ descends to a skew
   symmetric bilinear operator
 $$
 \llbr\ ,\ \rrbr:\Ga(T_\rho M)\x\Ga(T_\rho M)\to\Ga(T_\rho M). 
 $$  
 This satisfies a Leibniz rule, i.e.~$\llbr\xi,f\eta\rrbr=(\xi\cdot
 f)\eta+f\llbr\xi,\eta\rrbr$ holds for any $f\in C^\infty(M,\Bbb R)$
 and all $\xi,\eta\in\Ga(T_\rho M)$.  
 \end{prop}
 \begin{proof}
   It is evident that the operator on $\Ga(\Cal A_\rho M)$ is skew
   symmetric. If $s_1\in\Ga(\Cal A^0_\rho M)$, then
   $D^\rho_{s_1}s_2=-\{s_1,s_2\}$, while $D^\rho_{s_2}s_1\in\Ga(\Cal
   A^0_\rho M)$ by naturality of $D^\rho$. Thus the values lie in
   $\Ga(A^0_\rho M)$ for any choice of $s_2$, which implies that the
   operation descends to sections of $T_\rho M$ as claimed. The
   Leibniz rule for $D^\rho$ from Proposition \ref{prop2.2} together
   with the fact that the algebraic bracket $\{\ ,\ \}$ is bilinear
   over smooth functions implies the Leibniz rule for $\llbr\ ,\
   \rrbr$.
 \end{proof}
 Now it is easy to derive a formula for the twisted exterior
 derivative which is analogous to the formula for the covariant
 exterior derivative induced by a linear connection on a vector
 bundle.

 \begin{thm}\label{thm4.3}
   Let $\Cal VM$ be a relative tractor bundle and let
   $\nabla^{\rho,\Cal V}$ be the associated relative tractor
   connection. Then for $\ph\in\Om^k_\rho (M,\Cal VM)$ the twisted
   exterior derivative satisfies
 \begin{align*}
   d^{\Cal
     V}\ph&(\xi_0,\dots,\xi_k)=\tsum_{i=0}^r(-1)^i\nabla^{\rho,\Cal
     V}_{\xi_i}\ph(\xi_0,\dots,\widehat{\xi_i},\dots,\xi_k)\\
   &+\tsum_{i<j}(-1)^{i+j}\ph(\llbr\xi_i,\xi_j\rrbr,\dots,\widehat{\xi_i},
   \dots,\widehat{\xi_j},\dots,\xi_k)
 \end{align*}
 for all $\xi_0,\dots,\xi_k\in\Ga(T_\rho M)$. 
 \end{thm}
 \begin{proof}
   Let $\Pi_\rho:\Cal A_\rho M\to\Cal A_\rho M/\Cal A^0_\rho M=T_\rho
   M$ be the natural projection. For each $i$, choose $s_i\in\Ga(\Cal
   A_\rho M)$ such that $\Pi(s_i)=\xi_i$ and view $\ph$ as a section of
   $\La^k\Cal A_\rho^*M\otimes\Cal VM$. Then by definition $d^{\Cal
     V}\ph(\xi_0,\dots,\xi_k)=\tilde d^{\Cal V}\ph(s_0,\dots,s_k)$.

   Now we expand
   $(D^\rho_{s_i}\ph)(s_0,\dots,\widehat{s_i},\dots,s_k)$ as discussed
   in Section \ref{4.2}. In each of terms in which $D_{s_i}$ acts on
   $s_j$ for $j\neq i$, we can move $D_{s_i}s_j$ to the first entry of
   $\ph$, picking up a sign $(-1)^{j-1}$ if $j<i$ and $(-1)^j$ if
   $j>i$. Using this, we obtain the following alternative expression
   for $\tilde d_1^{\Cal V}$:
 \begin{equation}\label{td1alt}
 \begin{aligned}
 \tilde d_1^{\Cal V}\ph(s_0&,\dots,s_k)=
 \tsum_{i=0}^r(-1)^iD^\rho_{s_i}(\ph(s_0,\dots,\widehat{s_i},\dots,s_k))+\\
 &\tsum_{i<j}(-1)^{i+j}\ph(D^\rho_{s_i}s_j-D^\rho_{s_j}s_i,s_0,\dots,
 \widehat{s_i},\dots,\widehat{s_j},\dots,s_k). 
 \end{aligned}
 \end{equation}
 But now it is evident that the first sum in the right hand side
 combines with the part
 $\sum_{i=0}^k(-1)^is_i\bullet\ph(s_0,\dots,\widehat{s_i},\dots,s_k)$
 in $\partial_{\frak p/\frak p_+}\ph(s_0,\dots,s_k)$ to produce the
 first sum in the claimed formula. Likewise, the second sum in the
 right hand side combines with the second sum in $\partial_{\frak
   p/\frak p_+}\ph(s_0,\dots,s_k)$ to the second sum in the claimed
 formula.
 \end{proof}

 \subsection{The case of involutive relative tangent bundle}\label{4.4}
 The considerations in Section \ref{4.3} suggest studying parabolic
 geometries for which the relative tangent bundle is involutive. In
 view of the relation to twistor spaces discussed in Section \ref{4.6}
 below, this property is sometimes phrased as ``existence of a twistor
 space corresponding to $P\supset Q$ as a manifold''. Involutivity of
 $T_\rho M$ for a given parabolic geometry $(p:\Cal G\to M,\om)$ is
 easy to characterize in terms of the curvature $\ka\in\Om^2(M,\Cal
 AM)$ of the geometry. Indeed, it depends only on its torsion
 $\tau\in\Om^2(M,TM)$, which by definition is obtained by projecting
 the values of $\ka$ to $\Cal AM/\Cal A^0M\cong TM$. The result in
 Proposition 2.5 of \cite{twistor} can be phrased as follows.

 \begin{prop}\label{prop4.4.1}
   For a parabolic geometry $(p:\Cal G\to M,\om)$ of type $(G,Q)$ the
   relative tangent bundle is involutive if and only if the curvature
   $\ka$ maps $T_\rho M\x T_\rho M$ to sections of the subbundle $\Cal
   G\x_Q\frak p\subset\Cal AM$ or equivalently if $\tau$ maps $T_\rho
   M\x T_\rho M$ to $T_\rho M$.
 \end{prop}

 Using this result, we can define relative versions of the curvature
 and the torsion in the case of involutive relative tangent
 bundle. Observe that there is a natural projection $\Cal G\x_Q\frak
 p\to \Cal G\x_Q(\frak p/\frak p_+)=\Cal A_\rho M$.

 \begin{definition}\label{def4.4} 
   Consider a parabolic geometry $(p:\Cal G\to M,\om)$ of type $(G,Q)$
   for which the relative tangent bundle $T_\rho M\subset TM$ is
   involutive.

   Then we define the \textit{relative torsion}
   $\tau_\rho\in\Ga(\La^2T^*_\rho M,T_\rho M)$ as the restriction of
   the torsion $\tau$ to entries from $T_\rho M\subset TM$. Further,
   we define the \textit{relative curvature}
   $\ka_\rho\in\Ga(\La^2T_\rho^* M,\Cal A_\rho M)$ to be the image of
   the restriction of the curvature $\ka$ (which has values in $\Cal
   G\x_Q\frak p$ by Proposition \ref{prop4.4.1}) under the natural
   bundle map $\Cal G\x_Q\frak p\to\Cal A_\rho M$ from above.
 \end{definition}

 If the relative tangent bundle $T_\rho M$ is involutive, then one can
 associate a curvature $R^{\rho,\Cal V}\in\Om^2(M,L(\Cal VM,\Cal VM))$
 to the partial tractor connection $\nabla^{\rho,\Cal V}$. This can
 simply be defined by the usual formula
 $$
 R^{\rho,\Cal V}(\xi,\eta)(s)=\nabla^{\rho,\Cal V}_\xi\nabla^{\rho,\Cal
   V}_\eta s-\nabla^{\rho,\Cal V}_\eta\nabla^{\rho,\Cal V}_\xi s-
 \nabla^{\rho,\Cal V}_{[\xi,\eta]}s
 $$
 for $\xi,\eta\in\Ga(T_\rho M)$ and the usual proof shows that the
 right hand side is linear over smooth functions in all arguments. To
 compute this curvature, we first need a lemma. 

 \begin{lemma}\label{lemma4.4}
 For any parabolic geometry $(p:\Cal G\to M,\om)$ of type $(G,Q)$ and
 $\xi_1,\xi_2\in\Ga(T_\rho M)$ we have
 $$
 \llbr\xi_1,\xi_2\rrbr=[\xi_1,\xi_2]+\tau(\xi_1,\xi_2),
 $$ 
 where $\tau$ denotes the torsion of the geometry. 

 If the relative tangent bundle $T_\rho M$ is involutive, then in the
 right hand side, we can replace $\tau$ by the relative torsion
 $\tau_\rho$ and the summands in the right hand side both are sections
 of $T_\rho M$.
 \end{lemma}
 \begin{proof}
   For $i=1,2$ choose $s_i\in\Ga(\Cal A_\rho M)$ such that
   $\Pi_\rho(s_i)=\xi_i$. Then by definition,
 \begin{equation}
 \label{brr}
 \llbr\xi_1,\xi_2\rrbr=\Pi_\rho(D^\rho_{s_1}s_2-D^\rho_{s_2}s_1+\{s_1,s_2\}). 
 \end{equation}
 For $i=1,2$, choose a section $\tilde s_i$ of the subbundle $\Cal
 G\x_Q\frak p\subset\Cal AM$ which descends to $s_i$. Then by
 definition, $D_{\tilde s_1}\tilde s_2\in\Ga(\Cal G\x_Q\frak p)$
 descends to $D^\rho_{s_1}s_2$ and similarly for the other terms in the
 right hand side of \eqref{brr}.

 On the other hand, as we have noted in Section \ref{2.2}, sections of
 $\Cal AM$ can be identified with $Q$--invariant vector fields on
 $\Cal G$ and in this picture, $\Pi$ is just the projection of such
 vector fields to $M$. In particular, since the Lie bracket of
 $Q$--invariant vector fields is again $Q$--invariant, there is an
 operation $[\ ,\ ]:\Ga(\Cal AM)\x\Ga(\Cal AM)\to\Ga(\Cal AM)$
 corresponding to the Lie bracket of vector fields. But by
 construction, we then have $\Pi([\tilde s_1,\tilde s_2])=[\Pi(\tilde
   s_1),\Pi(\tilde s_2)]=[\xi_1,\xi_2]$. Finally, by Corollary 1.5.8
 of \cite{book}, one has
 $$
 [\tilde s_1,\tilde s_2]=D_{\tilde s_1}\tilde s_2- D_{\tilde s_2}\tilde
 s_1+\{\tilde s_1,\tilde s_2\}- \ka(\Pi(\tilde s_1),\Pi(\tilde s_2)),
 $$ 
 which immediately implies the result. 
 \end{proof}

 \begin{prop}\label{prop4.4.2}
 Let $(p:\Cal G\to M,\om)$ be a parabolic geometry of type $(G,Q)$ such
 that the relative tangent bundle $T_\rho M$ is involutive and let
 $\Cal VM\to M$ be a relative tractor bundle. 

(1) For $\ph\in\Om^{k-1}_\rho (M,\Cal VM)$, $d^{\Cal V}(d^{\Cal
   V}\ph)$ is induced by the section of $\La^{k+1}\Cal
 A_\rho^*M\otimes\Cal VM$ which maps $s_0,\dots,s_k$ to 
 $$
 \textstyle\sum_{i<j}(-1)^{i+j}(D^\rho_{\ka_\rho(\Pi(s_i),\Pi(s_j))}\ph)
 (s_0,\dots,\widehat{s_i},\dots, \widehat{s_j},\dots,s_k).
 $$
 If in addition the relative torsion vanishes, then this equals
 $$
 \textstyle\sum_{i<j}(-1)^{i+j+1}(\ka_\rho(\Pi(s_i),\Pi(s_j))\bullet
 \ph) (s_0,\dots,\widehat{s_i},\dots, \widehat{s_j},\dots,s_k).
 $$

(2) The curvature of the relative tractor connection
 $\nabla^{\rho,\Cal V}$ is given by
 $$
 R^{\rho,\Cal V}(\xi,\eta)(s)=\ka_\rho(\xi,\eta)\bullet s. 
 $$  
 \end{prop}
 \begin{proof}
(1) We have noted above that involutivity of $T_\rho M$ implies that
   $\ka$ has values in $\Cal G\x_Q\frak p$ if both its entries are
   from $T_\rho M$. Together with the fact that $\ph$ is a section of a relative
   natural bundle, this shows that
   $D_{\ka(\Pi(s_i),\Pi(s_j))}\ph=D^\rho_{\ka_\rho(\Pi(s_i),\Pi(s_j))}\ph$
   for all $i$ and $j$. Hence the first formula follows directly from
   Theorem \ref{thm4.2}. 

   If $\tau_\rho$ vanishes, then the values of $\ka_\rho$ lie in $\Cal
   A^0_\rho M$ and the second formula follows from part (1) of
   Proposition \ref{prop2.2}. 

(2) For $s\in\Ga(\Cal VM)$, part (1) implies that for
   $\xi,\eta\in\Ga(T_\rho M)$, we get $(d^{\Cal
     V})^2s(\xi,\eta)=-D^\rho_{\ka_\rho(\xi,\eta)}s$. On the other
   hand, applying Theorem \ref{thm4.3} to $d^{\Cal V}s=\nabla^{\Cal
     V,\rho}s$ which lies in $\Om^1_\rho(M,\Cal VM)$, we see that 
$$
(d^{\Cal
     V})^2s(\xi,\eta)= R^{\rho,\Cal V}(\xi,\eta)(s)+\nabla^{\rho,\Cal
     V}_{[\xi,\eta]}s- \nabla^{\rho,\Cal V}_{\llbr\xi,\eta\rrbr}s.
 $$ By Lemma \ref{lemma4.4}, the last two terms add up to
   $-\nabla^{\rho,\Cal V}_{\tau_\rho(\xi,\eta)}s$. Since
   $\ka_\rho(\xi,\eta)$ projects onto $\tau_\rho(\xi,\eta)$, we get
 $$
 -\nabla^{\rho,\Cal
   V}_{\tau_\rho(\xi,\eta)}s=-D^\rho_{\ka_\rho(\xi,\eta)}s-
 \ka_\rho(\xi,\eta)\bullet s, 
 $$ and the result follows.
 \end{proof}

 \subsection{The relative covariant exterior derivative}\label{4.5}
 If the relative tangent bundle $T_\rho M$ is involutive, then as
 mentioned before, one can follow the standard approach of extending a
 covariant derivative to an operator on bundle valued differential
 forms in our setting. Namely, for $\ph\in\Om^k_{\rho}(M,\Cal VM)$ and
 $\xi_0,\dots,\xi_k\in\Ga(T_\rho M)$, we define
 \begin{equation}\label{covext}
 \begin{aligned}
 (d^\nabla \ph)(\xi_0&,\dots,\xi_k):=\tsum_{i=0}^k(-1)^i\nabla^{\rho,\Cal
   V}_{\xi_i}\ph(\xi_0,\dots,\widehat{\xi_i},\dots,\xi_k)+\\
 &\tsum_{i<j}(-1)^{i+j}\ph([\xi_i,\xi_j],\xi_0,\dots,\widehat{\xi_i},\dots,
 \widehat{\xi_j},\dots,\xi_k). 
 \end{aligned}
 \end{equation}
 Note that involutivity of $T_\rho M$ is needed for this definition to
 make sense, since only sections of this subbundle may be inserted
 into $\ph$. The right hand side of \eqref{covext} is obviously
 alternating in $\xi_0,\dots,\xi_k$ and the same argument as for usual
 connections shows that $d^\nabla \ph$ is linear over smooth functions
 in each entry. Thus $d^\nabla \ph\in\Om^{k+1}_\rho(M,\Cal VM)$ and we
 have defined an operator
 $$
 d^{\nabla}:\Om^k_\rho(M,\Cal VM)\to\Om^{k+1}_\rho(M,\Cal VM) 
 $$
 called the \textit{relative covariant exterior derivative}.

 The relation between this operator and the relative twisted exterior
 derivative can be easily described using Lemma \ref{lemma4.4}. There
 is a natural insertion operator on $\Cal VM$--valued relative forms
 associated to the relative torsion. Namely, for $\ph\in\Om^k_\rho
 (M,\Cal VM)$ and $\xi_0,\dots,\xi_k\in\Ga(T_\rho M)$, we define
 \begin{equation}\label{idef}
   (i_{\tau_\rho}\ph)(\xi_0,\dots,\xi_k):=\tsum_{i<j}(-1)^{i+j+1}
   \ph(\tau_\rho(\xi_i,\xi_j),\xi_0,\dots,\widehat{\xi_i},\dots,
   \widehat{\xi_j},\dots,\xi_k), 
 \end{equation}
 so this coincides with the complete alternation of the insertion up to
 a positive factor. Having this at hand, we can now formulate:

 \begin{prop}\label{prop4.5}
   Let $(p:\Cal G\to M,\om)$ be a parabolic geometry of type $(G,Q)$
   such that $T_\rho M\subset TM$ is involutive.

   Then the relative twisted exterior derivative defined in Section
   \ref{4.1} is related to the relative covariant exterior derivative
   by
 $$
 d^{\Cal V}\ph=d^\nabla\ph+i_{\tau_\rho}\ph
 $$
 for all $\ph\in\Om^*_\rho(M,\Cal VM)$. 

 In particular, the relative covariant exterior derivatives define a
 compressable sequence of operators and in the case that $\frak
 p=\frak g$ (i.e.~for usual BGG sequences), $d^{\Cal V}$ coincides
 with the twisted exterior derivative as defined in \cite{CSS-BGG}.
 \end{prop}
 \begin{proof}
   By the last part of Lemma \ref{lemma4.4}, both in the right hand
   side of the formula for $\llbr\xi_1,\xi_2\rrbr$ in that lemma are
   sections of $T_\rho M$. Hence they can be inserted individually
   into a relative form, and using this, the first claim follows
   immediately from the formula for $d^{\Cal V}$ in Theorem
   \ref{thm4.3}.

   The second claim easily follows from the discussion on top of
   p.~105 in \cite{CSS-BGG} discussing the relation between the
   twisted exterior derivative (as defined there) and the covariant
   exterior derivative.
 \end{proof}

 \subsection{Relative BGG resolutions}\label{4.6}
 One of the key features of BGG--sequences is that on locally flat
 geometries, they are complexes defining fine resolutions of certain
 sheaves. We are next aiming at analogs of this results in the
 relative setting. There is a model case of this situation provided by
 so--called correspondence spaces, a special class of parabolic
 geometries of type $(G,Q)$ associated to the parabolic subalgebra
 $\frak p\supset\frak q$.

 Consider an arbitrary regular normal parabolic geometry $(p:\Cal G\to
 N,\om)$ of type $(G,P)$. Then we can define $M:=\Cal CN:=\Cal G/Q$,
 the orbit space under the restriction of the principal right action
 to the subgroup $Q\subset P$. Then $\Cal G/Q$ can be identified with
 $\Cal G\x_P(P/Q)$, so it is a smooth manifold and the total space of
 a natural fiber bundle over $N$ with typical fiber the generalized
 flag manifold $P/Q$. By construction, $\Cal G\to M$ is a
 $Q$--principal bundle and it is easy to see that the Cartan
 connection $\om\in\Om^1(\Cal G,\frak g)$ also defines a Cartan
 connection on $\Cal G\to M$. Hence one obtains a parabolic geometry
 of type $(G,Q)$ on $M$, which turns out to be automatically
 normal. In this context, $M=\Cal CN$ is called the
 \textit{correspondence space} associated to $N$ and the subalgebra
 $\frak q\subset\frak p$.

 The bundle projection $\pi:M\to N$ gives rise to the vertical subbundle
 $\ker(T\pi)\subset TM$. From the above construction it is clear that
 this vertical subbundle is given by $\Cal G\x_Q(\frak p/\frak
 q)\subset\Cal G\x_Q(\frak g/\frak q)=TM$. For a correspondence space
 $\Cal CN$, the relative tangent bundle $T_\rho\Cal CN$ thus coincides
 with the vertical subbundle of $\Cal CN\to N$. Hence $T_\rho\Cal CN$
 is globally integrable with global leaf--space $N$. Observe that in
 this setting a completely reducible representation $\Bbb V$ of $P$
 gives rise to both a relative tractor bundle $\Cal VM\to M$ and a
 (completely reducible) natural bundle $\underline{\Cal V}N:=\Cal
 G\x_P V\to N$ for the underlying geometry.

 Now we are ready to prove a general criterion ensuring that relative
 BGG sequences are resolutions and an interpretation of the sheaf that
 gets resolved in the model case of a correspondence space.  To
 formulate this, recall from Section \ref{4.3} that for a relative
 tractor bundle $\Cal VM\to M$, there is the relative tractor
 connection $\nabla^{\rho,\Cal V}$. This can be considered as an
 operator $\Ga(\Cal VM)\to\Om^1_\rho (M,\Cal VM)$ and clearly the
 kernel of this operator gives rise to a well defined subsheaf of the
 sheaf of smooth sections of $\Cal VM$.

 \begin{thm}\label{thm4.6}
   Suppose that $(p:\Cal G\to M,\om)$ is a parabolic geometry of type
   $(G,Q)$ for which the relative tangent bundle is involutive and that
   $P/Q$ is connected. Let $\Bbb V$ be any completely reducible
   representation of $P$ and let $\Cal VM\to M$ be the corresponding
   relative tractor bundle.

   (1) If the relative curvature $\ka_\rho$ of the geometry vanishes
   identically, then the relative BGG--sequence \eqref{rel-BGG} from
   Theorem \ref{thm4.1} determined by $\Bbb V$ is a complex and a fine
   resolution of the sheaf $\ker(\nabla^{\rho,\Cal V})$.

   (2) If $M$ is the correspondence space $\Cal CN$ of a parabolic
   geometry of type $(G,P)$, then the condition in (1) is always
   satisfied, and the sheaf $\ker(\nabla^{\rho,\Cal V})$ can be
   (globally) identified with the pullback of the sheaf of smooth
   sections of $\underline{\Cal V}N\to N$.
 \end{thm}
 \begin{proof}
   From part (1) of Proposition \ref{prop4.4.2}, we see that vanishing
   of the relative curvature implies that $(d^{\Cal V})^2=0$, so the
   twisted de--Rahm sequence determined by $\Cal VM$ is a
   complex. Moreover, by Proposition \ref{prop4.5} the twisted
   exterior derivative coincides with the covariant exterior
   derivative associated to the relative tractor connection. Thus on
   sufficiently small open subsets we are dealing with a twisted de
   Rham sequence along the fibers of the projection to a local leaf
   space, and it is a standard result that this is a fine resolution
   of the kernel of the first operator.

   By Theorem \ref{thm3.6}, the fact that the relative
   twisted de--Rham sequence is a complex implies that the same is true
   for the corresponding relative BGG sequence and that both sequences
   compute the same cohomology. Since this can be applied locally, (1)
   follows.

   To prove (2), consider a section $\si\in\Ga(\Cal VM)$ and let
   $f:\Cal G\to\Bbb V$ be the corresponding $Q$--equivariant
   function. Further, let $s\in\Ga(\Cal G\x_Q\frak p)$ be a smooth
   section. Then we can compute $\nabla^{\rho,\Cal V}_{\Pi(s)}\si$ as
   $D_s\si+s\bullet \si$. This easily implies that $\si$ is a section
   of $\ker(\nabla^{\rho,\Cal V})$ if and only if for each $A\in\frak
   p$ and $u\in\Cal G$ we have $(\om^{-1}(A)\cdot f)(u)=-A\cdot
   (f(u))$, where in the right hand side $A$ acts via the infinitesimal
   representation.

   If $M$ is a correspondence space $\Cal CN$, then $M=\Cal G/Q$ for a
   principal $P$--bundle $\Cal G\to N$ and the Cartan connection $\om$
   actually is a Cartan connection on this $P$--bundle. This shows that
   for $A\in\frak p$, the vector field $\om^{-1}(A)$ is the fundamental
   vector field of this $P$--bundle. Since such vector fields insert
   trivially into the curvature, it follows that the condition from (1)
   is satisfied. On the other hand, the above discussion shows that
   $\si$ lies in $\ker(\nabla^{\rho,\Cal V})$ if and only if $f$ is
   $\frak p$--equivariant. Since $P/Q$ is assumed to be connected,
   equivariancy under both $Q$ and $\frak p$ is equivalent to
   equivariancy under $P$, so $f$ corresponds to a smooth section of
   $\underline{\Cal V}N\to N$. The converse direction is obvious.
 \end{proof}

 This shows that the situation is significantly better than for
 ordinary BGG sequences, since in order to get a resolution, we only
 need vanishing of the relative curvature, which is much weaker than
 local flatness. In particular, for \textit{any} parabolic geometry of
 type $(G,P)$, any relative BGG sequence for the induced geometry of
 type $(G,Q)$ on the correspondence space is a resolution. However, as
 we shall also see in the example of path geometries in Section
 \ref{5.6}, the class of geometries for which we obtain resolutions is
 much larger than the class of correspondence spaces. Indeed, Theorem
 2.7 of \cite{twistor} shows that a geometry of type $(G,Q)$ is
 locally isomorphic to a correspondence space if and only if its
 curvature $\ka$ vanishes upon insertion of any element of $T_\rho M$,
 and this condition is much stronger than vanishing of $\ka_\rho$.

 \subsection{Local twistor spaces}\label{4.6a}
 We continue working in the setting of vanishing relative curvature so
 that part (1) of Theorem \ref{thm4.6} shows that the relative BGG
 sequence associated to $\Cal VM$ is a fine resolution of the sheaf
 $\ker(\nabla^{\rho,\Cal V})$. We are looking for curvature conditions
 which are weaker than the ones which locally characterize
 correspondence spaces but still allow us to obtain an explicit
 description of the sheaf $\ker(\nabla^{\rho,\Cal V})$. 

 For a parabolic geometry of type $(G,Q)$ over $M$ such that $T_\rho M$
 is involutive, we can consider local leaf spaces for the foliation
 defined by $T_\rho M$, which are then called \textit{local twistor
   space} for $M$ corresponding to $\frak p\supset\frak q$. Of course
 there is the hope to interpret $\ker(\nabla^{\rho,\Cal V})$ as the
 pullback of some sheaf on $N$. The proof of Theorem \ref{thm4.6}
 suggests that the key question here is, loosely speaking, whether the
 Cartan bundle $\Cal G\to M$ can be viewed as a principal $P$--bundle
 over a local twistor space. This question has been studied in
 \cite{twistor}, which suggests the following technical formulation of
 the concepts and directly provides some results.

 \begin{lemma}\label{lemma4.6a}
   Let $(p:\Cal G\to M,\om)$ be a parabolic geometry of type $(G,Q)$,
   such that $\ka_\rho=0$ (so $T_\rho M$ is involutive), and suppose
   that $U\subset M$ is open and $\ps:U\to N$ is a local leaf space
   for $T_\rho M$. Suppose further that $\pi:\Cal F\to N$ is a
   $P$--principal bundle and that $W\subset\Cal F$ is a $Q$--invariant
   open subset such that for each $x\in N$ the set $W_x/Q\subset \Cal
   F_x/Q\cong P/Q$ is non--empty and connected. Finally, suppose that
   $\ph$ is a $Q$--equivariant diffeomorphism from $W$ onto a
   $Q$--invariant subset of $p^{-1}(U)$ such that for each $A\in\frak
   p$, $\ph^*(\om^{-1}(A))$ is the fundamental vector field
   $\ze_A\in\frak X(\Cal F)$.

   Then over the open subset $p(\ph(W))\subset U$, the sheaf
   $\ker(\nabla^{\rho,\Cal V})$ is isomorphic to the pullback of the
   sheaf of smooth sections of $\Cal F\x_P\Bbb V$.
 \end{lemma}
 \begin{proof}
   From the proof of Theorem \ref{thm4.6} we see that sections of
   $\ker(\nabla^{\rho,\Cal V})$ are in bijective correspondence with
   $Q$--equivariant smooth functions $f$ such that $(\om^{-1}(A)\cdot
   f)(u)=-A\cdot f(u)$, where in the left hand side the vector field
   $\om^{-1}(A)$ is used to differentiate the function $f$, while in
   the right hand side $A$ acts on $\Bbb V$ via the infinitesimal
   representation. Pulling back via $\ph$, we by assumption get the
   sheaf of $Q$--equivariant smooth functions $W\to\Bbb V$ which in
   addition are $\frak p$--equivariant for the infinitesimal action.

   Now we claim that restriction to $W$ defines an isomorphism from the
   sheaf of $P$--equivariant functions $\Cal F\to\Bbb V$ onto the above
   sheaf. Then the result follows from the standard correspondence
   between equivariant functions and sections of an associated
   bundle. By assumption, $W$ meets each fiber of $\Cal F\to N$, so
   restriction to $W$ is injective (on the level of sheaves on
   $N$). 

   Surjectivity of the restriction can be proved locally (on $N$). So
   we can restrict to an open subset in $N$ over which $\Cal F$ admits
   a smooth section. Since locally such a section can be assumed to
   have values in $W$, we may restrict to the case that there is a
   global section $\tau:N\to W$ of $\Cal F$. But then of course for any
   smooth function $g:W\to\Bbb V$, the smooth function $g\o\tau$ can be
   uniquely extended to a $P$--equivariant function $f:\Cal F\to\Bbb
   V$. It suffices to prove that if $g$ is both $Q$--equivariant and
   $\frak p$--equivariant, then it coincides with the restriction of
   $f$. To do this, consider $\{u\in W:g(u)=f(u)\}$, which by
   definition is $Q$--invariant and meets each fiber $W_x:=W\cap \Cal
   F_x$ (in $\tau(x)$). Since both $f$ and $g$ are $\frak p$
   equivariant, the intersection with $W_x$ is open. Hence it projects
   onto a non--empty open subset of $W_x/Q$. But since the complement
   evidently is open and $Q$--invariant, too, connectedness of $W_x/Q$
   implies that $f=g$ on $W_x$. Since $x$ is arbitrary, the result
   follows.
 \end{proof}

 Now a result of \cite{twistor} provides a curvature condition
 ensuring that Lemma \ref{lemma4.6a} can be applied locally. This
 leads to trivial $P$--bundles and initially does not give a geometric
 description of the sheaf $\ker(\nabla^{\rho,\Cal V})$. However, under
 a slightly stronger condition, we can prove that under the
 assumptions of Lemma \ref{lemma4.6a} the Cartan connection $\om$
 induces a soldering form on $\Cal F$, which leads to a geometric
 interpretation of sections of associated bundles.

 \begin{thm}\label{thm4.6a}
   Suppose that $(p:\Cal G\to M,\om)$ is a parabolic geometry of type
   $(G,Q)$, that $P/Q$ is connected, and that the curvature $\ka$ of
   the geometry vanishes on $T_{\rho}M\x T_\rho M$ (so that $T_\rho M$
   is involutive). Let $\Bbb V$ be any completely reducible
   representation of $P$ and let $\Cal VM\to M$ be the corresponding
   relative tractor bundle.

   (1) For sufficiently small local leaf spaces $N$ for $T_\rho M$ the
   assumptions of lemma \ref{lemma4.6a} are satisfied for $\Cal F=N\x
   P$. Also, Theorem \ref{thm4.6} applies and the sheaf
   $\ker(\nabla^{\rho,\Cal V})$ resolved by the relative BGG sequence
   associated to $\Bbb V$ can be locally identified with the pullback
   of the sheaf $C^\infty(N,\Bbb V)$.

   (2) Assume in addition that inserting one element from $T_\rho M$
   into $\ka$ forces the values to lie in $\Cal G\x_Q\frak
   p\subset\Cal AM$. Then whenever the assumptions of Lemma
   \ref{lemma4.6a} are satisfied, the Cartan connection $\om$ induces
   a strictly horizontal, $P$--equivariant one--form $\th\in\Om^1(\Cal
   F,\frak g/\frak p)$.
 \end{thm}
 \begin{proof}
   Part (1) is proved in Proposition 2.6 of \cite{twistor}. (The
   connectedness of $W_x/Q$ is not explicitly stated there, but $W$ is
   constructed as $N\x U$ with $U\subset P$ the pre--image of an open
   neighborhood of $eQ$ in $P/Q$, and one may choose this neighborhood
   to be connected.)

   For part (2), we consider the form $\tilde\th\in\Om^1(W,\frak
   g/\frak p)$, which is obtained by projecting the values of
   $\ph^*\om$ to the quotient by $\frak p$. Since $\ph$ and $\om$ are
   $Q$--equivariant, also $\tilde\th$ is $Q$--equivariant. We claim
   that $\tilde\th$ is also $\frak p$--equivariant in the sense that
   $\Cal L_{\ze_A}\tilde\th=-\underline{\ad}(A)\o\tilde\th$. Here by
   $\underline{\ad}(A)$ we denote the action of $\frak p$ on $\frak
   g/\frak p$ induced by the adjoint action.

   By assumption we have $\ph^*\om(\ze_A)=A$, so
   $\tilde\th(\ze_A)=0$. Hence for any vector field $\eta$, we can
   compute $(L_{\ze_A}\tilde\th)(\eta)$ as $d\tilde\th (\ze_A,\eta)$,
   which in turn can be computed as the projection to $\frak g/\frak
   p$ of $(\ph^*d\om)(\ze_A,\eta)$. But now the assumption on the
   curvature implies that for any vector field $X$ on $\Cal G$, the
   expression $d\om(\om^{-1}(A),X)+[A,\om(X)]$ has values in $\frak
   p$. This shows that $(\ph^*d\om)(\ze_A,\eta)$ is congruent to
   $-[A,\ph^*\om(\eta)]$ modulo $\frak p$, which implies the claim.

   Having verified $\frak p$--equivariancy, one proceeds exactly as in
   the proof of Lemma \ref{lemma4.6a} to show that there is a unique
   $P$--equivariant one--form $\th\in\Om^1(\Cal F,\frak g/\frak p)$
   which restricts to $\tilde\th$ on $W$. By construction, the kernel
   of $\tilde\th$ in each point of $W$ is the vertical subspace of
   $\Cal F\to N$, and of course, this continues to hold for the
   equivariant extension $\th$.
 \end{proof}

 Observe that the result of part (2) in particular implies that $\Cal
 F\x_P(\frak g/\frak p)\cong TN$, which gives a description of all
 tensor bundles as associated bundles to $\Cal F$.

 \subsection{The Laplacian associated to the relative twisted exterior
   derivative}\label{4.7}

 We conclude this section with a finer analysis of the relative
 BGG--sequences induced by the relative twisted exterior derivative as
 obtained in Theorem \ref{thm4.1}. Since we are dealing with a
 sequence of compressable operators here, Section \ref{3.6} suggests
 looking at the Laplacians
 $$
 \square^{d^{\Cal V}}_r=\partial^*_\rho d^{\Cal V}+d^{\Cal
   V}\partial^*_\rho: \Om^k_\rho(M,\Cal VM)\to \Om^k_\rho(M,\Cal VM).
 $$
 In the case of ordinary BGG sequences it was already indicated in
 \cite{Casimir} that these Laplacians should be closely related to
 curved Casimir operators. In particular, it was proved there that the
 two operators coincide up to a constant in degree zero. To prove this
 in all degrees, we need a bit of preparation.

 Consider the Lie algebra $\frak p/\frak p_+$ endowed with the
 non--degenerate invariant bilinear form $B$ induced by the Killing
 form of $\frak g$. This bilinear form gives rise to a Casimir
 operator $C_0$ acting on representations of $\frak p/\frak p_+$. Via
 $B$, one can view the identity map as defining an element in $(\frak
 p/\frak p_+)\otimes (\frak p/\frak p_+)$ which is invariant under the
 natural action of $\frak p/\frak p_+$. Projecting this into the
 universal enveloping algebra $\Cal U(\frak p/\frak p_+)$, one obtains
 an element in the center. Acting by this element then defines a
 $\frak p/\frak p_+$--equivariant map on any representation of $\frak
 p/\frak p_+$. Of course, on a complex irreducible representation,
 this must be a multiple of the identity. Passing to associated
 bundles, we obtain an endomorphism of any relative tractor bundle,
 which we also denote by $C_0$.

 \begin{prop}\label{prop4.7}
   The Laplacians associated to the relative twisted exterior
   derivative are given by $\square^{d^{\Cal V}}_k=\tfrac12(\Cal
   C_\rho-\id_{\La^k\Cal A_\rho M}\otimes C_0)$.  In particular, if
   $\Cal V$ is induced by a complex irreducible representation $\Bbb
   V$, then $\square^{d^{\Cal V}}_k=\tfrac12(\Cal C_\rho-c_0\id)$
   where $c_0\in\Bbb C$ is the eigenvalue of $C_0$ on $\Bbb V$.
 \end{prop}
 \begin{proof}
   In Section \ref{4.1}, we have obtained $d^{\Cal V}$ as the
   restriction of an operator $\tilde d^{\Cal V}$ defined on sections
   of the bundle $\La^k\Cal A_\rho^*M\otimes\Cal VM$. The bundle
   $\La^kT_\rho^*M\otimes\Cal VM$ can be viewed as a subbundle in
   there, corresponding to the inclusion 
$$
\La^k(\frak q_+/\frak p_+)\otimes\Bbb V\hookrightarrow \La^k(\frak
p/\frak p_+)\otimes\Bbb V.
$$ 
Moreover, the bundle map $\partial^*_\rho$ is induced by a Lie algebra
homology differential for the smaller algebra, and thus can be viewed
as the restriction of a bundle map $\partial^*_{\frak p/\frak p_+}$
induced by the Lie algebra homology differential for the bigger
algebra. Thus we can complete the proof by showing that the operator
 \begin{equation}\label{bigsquare}
   \partial^*_{\frak p/\frak p_+}\o \tilde d^{\Cal V}+\tilde d^{\Cal
     V}\o \partial^*_{\frak p/\frak p_+}
 \end{equation}
 on $\Ga(\La^k\Cal A_\rho^*M\otimes\Cal VM)$ coincides with
 $\frac12(\Cal C_\rho -\id\otimes C_0)$.

 Now we write $\tilde d^{\Cal V}$ as $\tilde d_1^{\Cal
   V}+\partial_{\frak p/\frak p_+}$ and accordingly split
 \eqref{bigsquare} into two parts. For the rest of this proof, we will
 omit the subscripts and just write $\partial^*$ for
 $\partial^*_{\frak p/\frak p_+}$ and $\partial$ for $\partial_{\frak
   p/\frak p_+}$. To compute the first part, we observe that by
 definition, in terms of a local frame $\{\xi_\al\}$ for $\Cal A_\rho
 M$ and the dual frame $\{\eta_\al\}$ for $\Cal A^*_\rho M$, we can
 write $\tilde d_1^{\Cal V}\ph$ as $\sum_\al \eta_\al\wedge
 D^\rho_{\xi_\al}\ph$ for $\ph\in\Ga(\La^k\Cal A^*_\rho M\otimes \Cal
 VM)$. Now naturality of the relative fundamental derivative implies
 that $D^\rho_{\xi_\al}\partial^*\ph=\partial^*D^\rho_{\xi_\al}\ph$. On
 the other hand, we observe that Lemma 3.3.2 of \cite{book} continues
 to hold for the reductive algebra $\frak p/\frak p_+$ without any
 changes. Using part (1) of this lemma, we immediately conclude that
 we can write $\partial^*\tilde d_1^{\Cal V}\ph+\tilde d^{\Cal
   V}\partial^*\ph$ as
 $$
 -\textstyle\sum_\al \eta_\al\bullet D_{\xi_\al}\ph, 
 $$  
 where we now view $\{\xi_\al\}$ and $\{\eta_\al\}$ as local frames of
 $\Cal A_\rho M$, which are dual with respect to $B$. Using an adapted
 frame $\{X_i,A_r,Z^i\}$ for $\{\xi_\al\}$ (see Definition
 \ref{def2.3.2}), this can be rewritten as
 \begin{equation}
   \label{d1term}
 -\textstyle\sum_i Z^i\bullet D_{X_i}\ph+\sum_rA^r\bullet
 A_r\bullet\ph+\sum_iX_i\bullet Z^i\bullet\ph.   
 \end{equation}
 On the other hand, we can compute
 $\partial^*\partial+\partial\partial^*$ using the algebraic formulae
 from Sections 3.3.2 and 3.3.3 of \cite{book}. Similarly to the
 notation used there, we write $\Cal L$ for the algebraic action, so
 for $s\in\Ga(\Cal A_\rho M)$, we write $\Cal L_s$ for the operator
 $s\bullet\_$. Further we write $\Cal L^V_s$ for the tensor product of
 the identity on $\La^r\Cal A_\rho^*M$ and the action of $s$ on $\Cal
 VM$. Finally, we write $\ep_s$ for the wedge product by $s$, viewed as
 a section of $\Cal A_\rho^*M$ via $B$. Then from Section 3.3.3 of
 \cite{book}, we see that for dual frames $\{\xi_\al\}$ and
 $\{\eta_\al\}$ as above, we can write $\partial$ as
 \begin{equation}
   \label{partform}
   \textstyle\sum_\al\ep_{\eta_\al}\o \Cal
   L^V_{\xi_\al}+\frac12(\sum_\al \ep_{\eta_\al}\o (\Cal
   L_{\xi_\al}-\Cal L^V_{\xi_\al}))=\tfrac12\textstyle\sum_\al\ep_{\eta_\al}\o\Cal
   L_{\xi_\al}+\tfrac12\sum_\al\ep_{\eta_\al}\o\Cal L^V_{\xi_\al} 
   \end{equation}
   Now $\partial^*$ commutes with $\Cal L_s$ for any $s$, and then part
   (1) of Lemma 3.3.2 of \cite{book} shows that the anti--commutator of
   the first sum in the right hand side of \eqref{partform} with
   $\partial^*$ can be written as $-\tfrac12\sum_{\al}\Cal
   L_{\eta_\al}\o\Cal L_{\xi_\al}$, so up to the factor this is just
   the action of the Casimir element of $\frak p/\frak p_+$. Now we can
   use an adapted frame $\{X_i,A_r,Z^i\}$ as above for
   $\{\xi_\al\}$. Acting on $\ph$, we obtain terms
 $$
 -\tfrac12\textstyle\sum_i (Z^i\bullet X_i\bullet\ph+X_i\bullet Z^i\bullet\ph),
 $$ 
 which add up with the last term in \eqref{d1term} to
 $-\tfrac12\sum_i\{Z^i,X_i\}\bullet\ph$. On the other hand, we get a
 term $-\frac12\sum_rA^r\bullet A_r\bullet \ph$, which adds up with the
 middle sum in \eqref{d1term} to $\frac12\sum_rA^r\bullet A_r\bullet
 \ph$. In view of Proposition \ref{prop2.3}, we have exactly obtained the
 action of half of the relative curved Casimir so far.

 Again by part (1) of Lemma 3.3.2 of \cite{book}, the anti--commutator of
 the second sum in the right hand side of \eqref{partform} with
 $\partial^*$ can be rewritten as
 $$
 -\tfrac12\textstyle\sum_\al \Cal L_{\eta_\al}\o\Cal L^V_{\xi_\al}-
 \frac12\sum_\al \ep_{\eta_\al}\o(\partial^*\o\Cal L^V_{\xi_\al}-\Cal
 L^V_{\xi_\al}\o\partial^*). 
 $$
 Using parts (2) and (3) of Lemma 3.3.2 of \cite{book}, one immediately
 computes that the last sum (including the sign) equals
 $$
 \tfrac12\textstyle\sum_\al (\Cal L_{\eta_\al}-\Cal L^V_{\eta_\al})\o\Cal
 L^V_{\xi_\al}, 
 $$ 
 so this part just gives $-\frac12$ times the tensor product of the
 action of the Casimir element on $V$ with the identity.
 \end{proof}

 \subsection{Algebraic properties of splitting operators}\label{4.8}
 The final important aspect of the BGG--machinery are results ensuring
 that the splitting operators have values in certain subbundles of the
 bundles of differential forms. The corresponding results for usual
 BGG sequences were proved in \cite{twistor} and they are a crucial
 ingredient for the results on subcomplexes in BGG sequences in
 \cite{subcomplexes}.

 Consider a representation $\Bbb V$ of $P$, let $\Bbb E\subset
 \La^k(\frak q_+/\frak p_+)\otimes\Bbb V$ be a $Q$--submodule for some
 $k$, and put $\Bbb E_0:=\Bbb E\cap\ker(\square_\rho)$. Since the
 latter is a $Q_0$--invariant subspace, it corresponds to a smooth
 subbundle $\Cal E_0\subset\Cal H_k(T^*M,\Cal VM)$. The kind of
 statement we want to prove is that applying a splitting operator to a
 section of $\Cal E_0$, we obtain a section of the bundle $\Cal
 E\subset\La^rT^*_\rho M\otimes\Cal VM$ corresponding to $\Bbb E$. As
 we shall see soon, this needs only minimal assumption in the case of
 the splitting operators constructed from the relative twisted
 exterior derivative. However, for some applications, we have to study
 the splitting operators corresponding to the relative covariant
 exterior derivative. To deal with those, an additional concept is
 needed.

 \begin{definition}\label{def4.8}
   Suppose we have given $Q$--submodules $\Bbb E\subset \La^k(\frak
   q_+/\frak p_+)\otimes\Bbb V$ and $\Bbb F\subset\La^2(\frak q_+/\frak
   p_+)\otimes (\frak p/\frak p_+)$. Then we say that \textit{$\Bbb E$
     is stable under $\Bbb F$--insertions} if and only if for any
   $\ph\in\Bbb E$ and $\ps\in\Bbb F$ (both viewed as alternating
   multilinear maps on $\frak p/\frak q$) the image under
   $\partial^*_\rho$ of the total alternation of the map
 $$
 (X_0,\dots,X_k)\mapsto \ph(\ps(X_0,X_1)+\frak q,X_2,\dots,X_k)
 $$  
   lies again in $\Bbb E$. 
 \end{definition}

 The proofs of the following results are significantly simpler than the
 proofs for usual BGG sequences in Section 3.2 of \cite{twistor}.

 \begin{prop}\label{prop4.8}
   Suppose that $\Bbb E\subset \La^k(\frak q_+/\frak p_+)\otimes\Bbb
   V$ and $\Bbb F\subset\La^2(\frak q_+/\frak p_+)\otimes (\frak
   p/\frak p_+)$ are $Q$--submodules which are invariant under
   $\id\otimes C_0$, the action of the Casimir element on $\Bbb V$, 
   respectively $\frak p/\frak p_+$. Put $\Bbb E_0=\Bbb
   E\cap\ker(\square_\rho)$ and likewise for $\Bbb F$ and let us
   denote the corresponding natural subbundles by
   \begin{gather*}
     \Cal E_0M\subset\Cal H_k(T^*_\rho M,\Cal VM)\quad \Cal
     EM\subset\La^rT^*_\rho M\otimes\Cal VM \\
     \Cal F_0M\subset\Cal H_2(T^*_\rho M,\Cal A_\rho M)\quad \Cal
     FM\subset\La^rT^*_\rho M\otimes\Cal A_\rho M.
   \end{gather*}

 (1) The splitting operator associated to $d^{\Cal V}$ maps $\Ga(\Cal
 E_0M)$ to $\Ga(\Cal EM)$. 

 (2) If $T_\rho M\subset TM$ is involutive, $\Bbb E$ is stable under
 $\Bbb F$--insertions, and the relative curvature $\ka_\rho$ of the
 geometry is a section of $\Cal FM$, then the splitting operator
 associated to $d^\nabla$ maps $\Ga(\Cal E_0M)$ to $\Ga(\Cal EM)$.

 (3) Suppose that $\frak p=\frak g$, so $\Bbb F\subset\La^2\frak
 q_+\otimes\frak g$, and that the geometry in question is regular and
 normal. Then if $\Bbb F$ is stable under $\Bbb F$--insertions and the
 harmonic curvature $\ka^h$ of the geometry is a section of $\Cal
 F_0M$, its curvature $\ka$ is a section of $\Cal FM$.
 \end{prop}
 \begin{proof}
 If necessary, we can replace $\Bbb E$ and $\Bbb F$ by their
 intersections with $\ker(\partial^*_\rho)$, which does not change the
 intersection with $\ker(\square_\rho)$. 

 (1) Let $\square^{\Cal V}$ be the Laplacian associated to $d^{\Cal
   V}$. Since the relative curved Casimir preserves sections of any
 natural subbundle, we conclude from Proposition \ref{prop4.7} that
 our assumptions imply that $\square^{\Cal V}$ maps sections of $\Cal
 EM$ to sections of $\Cal EM$. Of course, the same is then true for
 any polynomial in $\square^{\Cal V}$. The construction of the
 splitting operator in Section \ref{3.3} shows that its value on a
 section of $\Cal E_0M$ can be obtained from applying such a
 polynomial to a representative section of
 $\ker(\partial^*_\rho)$. But by construction, we can choose a
 representative section in $\Cal EM$, which implies the result.

 (2) Our assumptions imply that the relative covariant exterior
 derivative $d^\nabla$ is defined and we denote by $\square^{\nabla}$
 the associated Laplacian. Then Proposition \ref{prop4.5} implies that
 for $\ph\in\Ga(\ker(\partial^*_\rho))$, we have
 $\square^{\nabla}\ph=\square^{\Cal
   V}\ph-\partial^*(i_{\tau_\rho}\ph)$.  Since $\tau_\rho$ is just the
 projection of $\ka_\rho\in\Ga(\Cal FM)$, the fact that $\Bbb E$ is
 stable under $\Bbb F$--insertions implies that for $\ph\in\Ga(\Cal
 EM)$, we have $\square^{\nabla}\ph\in\Ga(\Cal EM)$, and we can
 conclude the proof as in part (1).

 (3) Here we deal with the adjoint tractor bundle $\Cal AM$ and the
 adjoint tractor connection $\nabla^{\Cal A}$. By part (2) of
 Proposition \ref{prop4.4.2}, the curvature of $\nabla^{\Cal A}$ is
 given by $R^{\Cal A}(\xi,\eta)(s)=\{\ka(\xi,\eta),s\}$ for
 $\xi,\eta\in\frak X(M)$ and $s\in\Ga(\Cal AM)$. Together with the
 Bianchi--identity for linear connections, this easily implies that
 $\ka\in\Om^2(M,\Cal AM)$ satisfies $d^{\nabla}\ka=0$. By normality,
 we have $\partial^*\ka=0$ and by definition
 $\ka^h=\pi_H(\ka)\in\Ga(\Cal H_2(TM,\Cal AM))$. Hence denoting by
 $S^{\nabla}$ the splitting operator corresponding to $d^\nabla$, part
 (2) of Theorem \ref{thm3.3} shows that $\ka=S^\nabla(\ka^h)$.

 Now let $S^{\Cal A}$ be the splitting operator constructed from
 $d^{\Cal A}$. Then by part (1), $\ph:=S^{\Cal A}(\ka^h)$ is a section
 of $\Cal FM\subset\Cal AM$. As in the proof of part (2), we get
 $\square^{\nabla}\ph=-\partial^*(i_{\tau}\ph)$. Denote by $Q^\nabla$
 the operator constructed from $d^\nabla$ as in Theorem \ref{thm3.4},
 and put $\tilde\ph:=\ph+Q^\nabla\partial^*(i_{\tau}\ph)$. Then by
 construction $\square^\nabla(\tilde\ph)=0$ and since the image of
 $Q^\nabla$ lies in the image of $\partial^*$, we get
 $\pi_H(\tilde\ph)=\pi_H(\ph)=\ka^h$. Hence we see that
 $\tilde\ph=S^{\nabla}(\ka^h)=\ka$.

 Now for $i\geq 1$, assume that $\ka$ is congruent to a section of
 $\Cal FM$ modulo elements which are homogeneous of degree $\geq
 i+1$. This is certainly satisfied for $i=1$: By regularity $\tau$ and
 $\ka^h$ are of homogeneity $\geq 1$, so $\ph$ is of homogeneity $\geq
 1$. Thus $i_{\tau_\rho}\ph$ is of homogeneity $\geq 2$, and
 $Q^\nabla$ and $\partial^*$ preserve homogeneities.

 But if we assume that this is satisfied for some $i$, then we can
 write $\ka=\ka^1+\ka^2$ with $\ka^1\in\Ga(\Cal FM)$ and $\ka^2$
 homogeneous of degree at least $i+1$. Defining $\tau^1$ and $\tau^2$
 as the corresponding images in $\Om^2(M,TM)$, we have
 $i_{\tau}\ph=i_{\tau^1}\ph+i_{\tau^2}\ph$. Since $\Bbb F$ is stable
 under $\Bbb F$--insertions, $\partial^*(i_{\tau^1}\ph)$ is a section
 of $\Cal FM$, and as in (2), $Q^\nabla$ preserves the space of
 sections of this subbundle. On the other hand, $i_{\tau^2}\ph$ is
 homogeneous of degree at least $i+2$, and again this property is
 preserved by $Q^\nabla\partial^*$. Hence we conclude that the
 property is satisfied for $i+1$ and by induction the claim follows.
 \end{proof}

\begin{remark}\label{rem4.8}
  We remark briefly here that there is a relative version of normality
  and harmonic curvature and a corresponding extension of part (3) of
  Proposition \ref{prop4.8}. Consider a parabolic geometry $(p:\Cal
  G\to M,\om)$ of type $(G,Q)$ with involutive relative tangent bundle
  $T_\rho M$. Then one has the relative curvature
  $\ka_\rho\in\Om^2_\rho(M,\Cal A_\rho M)$ as defined in Definition
  \ref{def4.4}. The relative version of normality then is to require
  that $\partial^*_\rho(\ka_\rho)=0$, which allows one to define a
  \textit{relative harmonic curvature}
  $\ka^h_\rho=\pi_H(\ka_\rho)\in\Ga(\Cal H_2(T_{\rho}M,\Cal A_\rho
  M))$.

  Consider the relative tractor connection $\nabla^{\rho,\Cal A}$ on
  the relative tractor bundle $\Cal A_\rho M$ and let $d^\nabla$ be
  the induced relative covariant exterior derivative. Then similarly
  as in the proof of Proposition \ref{prop4.8}, one can use the
  Bianchi identity for linear connections to prove that
  $d^{\nabla}\ka_\rho=0$. Denoting by $S^\nabla$ the splitting
  operator constructed from $d^\nabla$, this immediately implies
  $\ka_\rho=S^\nabla(\ka^h_\rho)$, which can be viewed as a strong
  version of the Bianchi identity. In particular, vanishing of
  $\ka^h_\rho$ implies vanishing of $\ka_\rho$.

  The proof of part (3) of Proposition \ref{prop4.8} then extends
  without changes to the general setting. Consider a $Q$--submodule
  $\Bbb F\subset\La^2(\frak q_+/\frak p_+)\otimes(\frak p/\frak p_+)$,
  put $\Bbb F_0:=\Bbb F\cap\ker(\square_\rho)$, and consider the
  associated bundles $\Cal F_0M$ and $\Cal FM$. Assuming that $\Bbb F$
  is stable under $\id\otimes C_0$ and under $\Bbb F$--insertions, the
  fact that $\ka^h_\rho\in\Ga(\Cal F_0M)$ then implies
  $\ka_\rho\in\Ga(\Cal FM)$.

  We have not discussed this topic in more detail, since it is unclear
  how relevant the concept of relative normality is for interesting
  examples of parabolic geometries.
\end{remark}

 \subsection{Twistor spaces and harmonic curvature}\label{4.9}
 Using Proposition \ref{prop4.8}, we can now give sufficient
 conditions for some of the properties studied so far in terms of the
 harmonic curvature. These are then easy to verify, since the harmonic
 curvature is usually easy to understand explicitly. The harmonic
 curvature $\ka^h$ of a geometry of type $(G,Q)$ is a section of the
 bundle $\Cal G\x_P H_2(\frak q_+,\frak g)$ which can be viewed as a
 subbundle of the bundle $L(\La^2\gr(TM),\gr(\Cal AM))$. So in
 particular, one can naturally formulate conditions on insertions of
 one or two elements from $\gr(T_\rho M)\subset\gr(TM)$ into the
 harmonic curvature, respectively on harmonic curvatures being of
 torsion type (i.e.~having values in $\gr_i(\Cal AM)$ with
 $i<0$). Using this, we can now formulate:

 \begin{prop}\label{prop4.9}
 Let $(p:\Cal G\to M,\om)$ be a parabolic geometry of type $(G,Q)$ with
 curvature $\ka$ and harmonic curvature $\ka^h$. 

 (1) Suppose that $\ka^h$ vanishes upon insertion of two elements of
 $\gr(T_\rho M)$ and that the torsion--type components of $\ka^h$ even
 vanish, if one of their entries is from $\gr(T_\rho M)$. Then $T_\rho
 M$ is involutive, and $\ka$ vanishes on $T_\rho M\x T_\rho M$. Hence
 the relative curvature $\ka_\rho$ vanishes and we are in the setting
 of part (1) of Theorem \ref{thm4.6}.

 (2) Suppose that $\ka^h$ vanishes upon insertion of one element of
 $\gr(T_\rho M)$. Then the same is true for $\ka$, so we are in the
 setting of part (2) of Theorem \ref{thm4.6}.
 \end{prop}
 \begin{proof}
   We apply part (3) of Proposition \ref{prop4.8} to modules depending
   on $\frak p$. So we have to look at a submodule $\Bbb F\subset
   \La^2\frak q_+\otimes\frak g$, which is stable under the
   $\id\otimes C_0$, where $C_0$ is the Casimir element of $\frak
   g$. Provided that $\Bbb F$ is stable under $\Bbb F$--insertion we
   can then conclude that $\ka\in\Ga(\Cal FM)$ from $\ka^h\in\Ga(\Cal
   F_0M)$.

   Now from the $\frak q$--submodule $\frak p_+\subset\frak q_+$ we
   get submodules $\La^2\frak p_+\subset \frak p_+\wedge\frak
   q_+\subset\La^2\frak q_+$. Since $\frak p_+$ is the annihilator of
   $\frak p/\frak q\subset\frak g/\frak q$, these are the spaces of
   those maps which vanish upon insertion of one, respectively two,
   elements of $\frak p/\frak q$.

   (1) Putting $\Bbb F:=\La^2\frak p_+\otimes\frak g +\frak
   p_+\wedge\frak q_+\otimes\frak q$, the assumption in (1) is exactly
   that $\ka^h\in\Ga(\Cal F_0M)$, while the conclusion of (1) is that
   $\ka\in\Ga(\Cal FM)$. So we only have to show that $\Bbb F$ is
   stable under $\id\otimes C_0$ and under $\Bbb F$--insertions. For
   the first property it suffices to show that $\frak q\subset\frak g$
   is stable under the action of the Casimir element $C_0$. This
   follows by taking appropriate dual bases to compute the
   Casimir. Take a basis $\xi_\ell$ which fills up step by step the
   $\frak q$--invariant filtration components $\frak g^i\subset\frak
   g$. Then the dual basis $\eta_\ell$ has the property that for
   $\xi_\ell\in\frak g^i$ we have $\eta_\ell\in\frak g^{-i}$. Hence
   for $A\in\frak q=\frak g^0$, we always have
   $[\xi_\ell,[\eta_\ell,A]]\in\frak g^0=\frak q$.

   On the other hand, taking $\ph,\ps\in \Bbb F$, the map
   $(X_1,X_2,X_3)\mapsto \ps(\ph(X_1,X_2),X_3)$ vanishes, if $X_1$ or
   $X_2$ lies in $\frak p/\frak q$ and has values in $\frak q$ if
   $X_3\in\frak p/\frak q$. This means that the complete alternation
   of this map has values in $\frak q$, if one inserts one element from
   $\frak p/\frak q$ and vanishes upon insertion of two such
   elements. Hence it lies in
 $$
 \La^2\frak p_+\wedge\frak q_+\otimes\frak q+\La^3\frak
 p_+\otimes\frak g.  
 $$
 Since $[\frak q_+,\frak p_+]\subset\frak p_+$ it follows directly
 from the definition that $\partial^*_\rho$ maps this module into
 $\Bbb F$, so $\Bbb F$ is stable under $\Bbb F$--insertions.

 (2) This is proved in Theorem 3.3 of \cite{twistor}, but since the
 proof becomes very simple with our tools, we provide the
 argument. Putting $\Bbb F:=\La^2\frak p_+\otimes\frak g$, part (2)
 exactly says that $\ka^h\in\Ga(\Cal F_0M)$ implies $\ka\in\Ga(\Cal
 FM)$. But in this case stability under $\id\otimes C_0$ is not an
 issue and the insertion just produces elements of $\La^3\frak
 p_+\otimes\frak g$, which are mapped to $\Bbb F$ by $\partial^*$. Thus
 $\Bbb F$ is stable under $\Bbb F$--insertion and the result follows
 from part (3) of Proposition \ref{prop4.8}.
 \end{proof}

 \section{Examples and Applications}\label{5}
 We will mainly discuss applications involving the BGG sequences
 associated to the relative twisted exterior derivative. These
 applications naturally split into two different groups, which are
 distinguished by representation theory data. On the one hand, there
 are cases in which the bundles occurring in a relative BGG sequence
 also occur in a standard BGG sequence. For this case, our main result
 is that the operators on these spaces coming from the two sequences
 actually agree. Hence the relative BGG sequence is a subsequence in a
 standard BGG sequence in these cases. Under weak curvature
 conditions, these subsequences are subcomplexes and fine resolutions
 of certain sheaves.

 On the other hand, there are cases in which the bundles contained in
 a relative BGG sequence cannot occur in a standard BGG sequence. In
 these cases, the constructions of invariant differential operators
 which are available in the literature, need a lot of case--by--case
 considerations (even to decide whether they are applicable). Hence in
 these cases, we obtain the first general uniform construction for
 invariant differential operators between the bundles in question. Of
 course, the results on curvature conditions ensuring that relative
 BGG sequences are complexes or even fine resolutions of sheaves from
 Section \ref{4} also apply in these cases.

 We will make both kinds of examples explicit for generalized path
 geometries. These form an example of broader interest, because of their
 relation to systems of second order ODEs, and at the same time they
 nicely expose the features of relative BGG sequences.

 \subsection{The two types of examples}\label{5.1}
 As mentioned above, the distinction between the two classes of
 examples we obtain is related to representation theory, in particular
 to the distinction between regular and singular infinitesimal
 character. We will formulate things in a direct way first and then
 discuss the interpretation in terms of weights.

 Given nested parabolic subalgebras $\frak q\subset\frak p$ in a
 semi--simple Lie algebra $\frak g$ and corresponding groups $Q\subset
 P\subset G$, there are two ways to apply the theory developed in this
 article to the construction of invariant differential
 operators. Either we can construct ``absolute'' BGG sequences as known
 from \cite{CSS-BGG} and \cite{Calderbank--Diemer} in the way presented
 in Sections \ref{3} and \ref{4} (so this corresponds to the pair
 $\frak q\subset\frak g$). For this construction, the starting point is
 a representation $\tbv$ of $G$ and the bundles in the resulting BGG
 sequence are induced by the summands of the completely reducible
 representation $H_*(\frak q_+,\tbv)$ of $Q$.

 On the other hand, we can apply the relative BGG construction for the
 pair $\frak q\subset\frak p$. Here we start with a completely
 reducible representation $\Bbb V$ of the group $P$, and the bundles in
 the resulting relative BGG sequence are associated to the summands of
 the completely reducible representation $H_*(\frak q_+/\frak p_+,\Bbb
 V)$ of $Q$. 

 Now there obviously is some overlap between the two cases. The
 simplest example of this is that an irreducible representation $\tbv$
 of $G$ has a unique $P$--irreducible quotient as well as a unique
 $Q$--irreducible quotient. These can be nicely described as $\Bbb
 V:=H_0(\frak p_+,\tbv)$ and as $H_0(\frak q_+,\tbv)$, respectively.
 It is easy to see that $H_0(\frak q_+/\frak p_+,\Bbb V)$ is the
 $Q$--irreducible quotient of $\Bbb V$ and hence isomorphic to
 $H_0(\frak q_+,\tbv)$. This is vastly generalized in Theorem 3.3 of
 \cite{part1}, where it is shown that
 $$
 H_k(\frak q_+,\tbv)\cong \oplus_{i+j=k}H_i(\frak q_+/\frak
 p_+,H_j(\frak p_+,\tbv)). 
 $$
 Since $H_*(\frak p_+,\tbv)$ is always a completely reducible
 representation, we conclude that the absolute BGG sequence is (as far
 as the bundles are concerned) the union of the relative BGG sequences
 induced by the components of $H_*(\frak p_+,\tbv)$. Of course, in
 these cases, the main question is relating the two BGG sequences.

\smallskip

The importance of bundles induced by Lie algebra homology spaces is
explained by Kostant's theorem (see \cite{Kostant}), and its relative
analog, Theorem 2.7 of \cite{part1}. In the complex setting,
irreducible representations of a parabolic subalgebra of $\frak g$ can
be described in terms of weights, which are linear functionals on a
Cartan subalgebra $\frak h\subset\frak g$. Now the relation between
invariant differential operators on the homogeneous model of a
parabolic geometry and representation theory (in particular
homomorphisms between generalized Verma modules) heavily restricts the
representations inducing bundles which allow a non--zero invariant
differential operator between their spaces of sections. Namely, the
negatives of the lowest weights of the inducing representations have
to lie in the same orbit for the so--called affine action of the Weyl
group $W$ of $\frak g$ on the space of weights. Since each such orbit
is finite, this is a very strong restriction.

 Given a finite dimensional irreducible representation $\tbv$ of $\frak
 g$, Kostant's original theorem first shows that each summand of the
 homology $H_*(\frak q_+,\tbv)$ corresponds to a weight in the affine
 Weyl orbit determined by $\tbv$. Moreover, the weights obtained by
 these summands are the only ones in the orbit that can be realized by
 finite dimensional irreducible representations of $\frak q$. Hence the
 BGG sequence contains all candidates for targets and domains of
 invariant differential operators corresponding to that affine Weyl
 orbit. 

 The relative version of Kostant's theorem from \cite{part1} proves a
 similar statement starting from a finite dimensional, complex
 irreducible representation $\Bbb V$ of $\frak p$. Again the first
 statement is that all summands of $H_*(\frak q_+/\frak p_+,\Bbb V)$
 correspond to weights in the affine Weyl orbit determined by $\Bbb
 V$. Moreover, they are the only weights in the orbit under a subgroup
 $W_{\frak p}\subset W$, which can be realized by finite dimensional
 irreducible representations of $\frak q$.

 The description via weights also shows that there are many
 $P$--irreducible representations $\Bbb V$ such that no summand of
 $H_*(\frak q_+/\frak p_+,\Bbb V)$ can occur in $H_*(\frak q_+,\tbv)$
 for a representation $\tbv$ of $\frak g$. The point here is that
 realizability of a weight by a finite dimensional representation of
 one of the Lie algebras or groups in question depends on dominancy
 and integrality conditions for the weights. Here we have to consider
 the three conditions of $\frak g$--dominancy, $\frak p$--dominancy,
 and $\frak q$--dominancy, which imply each other in that
 sequence. Now for any irreducible representation $\tbv$ of $\frak g$
 the corresponding weight is $\frak g$--dominant and integral. For
 irreducible representations of $\frak p$ and $\frak q$ (and even of
 $P$ and $Q$) both the dominancy and integrality conditions are
 weaker. In particular, there are many examples of finite dimensional,
 complex, $P$--irreducible representations $\Bbb V$, which correspond
 to a weight whose affine Weyl--orbit does not contain any $\frak
 g$--dominant integral weight.  In particular, the summands of
 $H_*(\frak q_+/\frak p_+,\Bbb V)$ then give rise to bundles that do
 not show up in a standard BGG sequence.  This is always the case in
 singular infinitesimal character (compare with Section 3.1 of
 \cite{part1}), but also in regular character there are many examples
 of representations $\Bbb V$ corresponding to a non--integral weight.

 In this second situation we obtain ``new'' invariant differential
 operators and a new relation to relative forms and the relative
 twisted exterior derivative. The main question here is to describe
 the bundles on which these operators act, the order of the operators,
 and the complexes and resolutions obtained by the relative BGG
 construction.

 \subsection{The case of subsequences}\label{5.2}
 To discuss the first type of results described in Section \ref{5.1},
 we consider a representation $\tbv$ of $G$ and compare the absolute
 BGG sequence determined by $\tbv$ to the relative BGG sequences
 determined by summands of $H_*(\frak p_+,\tbv)$. Initially, the
 isomorphism
 $$
 H_k(\frak q_+,\tbv)\cong \oplus_{i+j=k}H_i(\frak q_+/\frak
 p_+,H_j(\frak p_+,\tbv))
 $$
 is obtained in \cite{part1} abstractly from comparing
 weights. However, in \cite{part1} we have also given a more explicit
 description of the relation between the two homologies, which can be
 translated to geometry.

 Namely, for each $\ell\leq k$ we have constructed there $Q$--invariant
 subspaces $\tcf^\ell_k\subset\Cal F^\ell_k\subset C_k(\frak
 q_+,\tbv)$. These submodules correspond to natural subbundles
 $$
 \tilde F^\ell_kM\subset F^\ell_kM\subset\La^kT^*M\otimes\tcv M . 
 $$
 The description of $\Cal F_k^\ell$ in \cite{part1} readily translates
 to geometry. A $\tcv M$--valued $k$--form lies in $F^\ell_kM$ if and
 only if it vanishes upon insertion of $k-\ell+1$ sections of the
 subbundle $T_\rho M\subset TM$. Moreover, there is a $\frak
 q$--equivariant surjection $\pi:\tcf^\ell_k\to\La^{k-\ell}(\frak
 q_+/\frak p_+)\otimes H_\ell(\frak p_+,\tbv)$. Denoting by $\Bbb
 V_\ell$ the completely reducible representation $H_\ell(\frak
 p_+,\tbv)$ and by $\Cal V_\ell M$ the corresponding relative tractor
 bundle, we obtain, for each $\ell\leq k$, an induced bundle map
 \begin{equation}\label{pidef}
 \pi:\tilde F^\ell_k M\to\La^{k-\ell}T^*_\rho M\otimes \Cal V_\ell M. 
 \end{equation}

 \begin{prop}\label{prop5.2}
   For each $\ell\leq k$, the twisted exterior derivative $d^{\tcv}$
   maps the subspace $\Ga(\tilde F^\ell_kM)\subset\Om^k(M,\tcv M)$ to
   $\Ga(\tilde F_{k+1}^\ell M)\subset\Om^{k+1}(M,\tcv M)$. Moreover,
   denoting by $d^{\Cal V_\ell}$ the relative twisted exterior
   derivative on $\Om^*_\rho(M,\Cal V_\ell M)$, we get $\pi\o
   d^{\tcv}\ph=d^{\Cal V_\ell}(\pi\o\ph)$ for all $\ph\in\Ga(\tilde
   F^\ell_kM)$.
 \end{prop}
 \begin{proof}
   Observe first that $d^{\tcv}$ is the twisted exterior derivative
   giving rise to an absolute BGG sequence. Hence to follow the
   construction from Section \ref{4.1}, we have to start by viewing
   $\La^kT^*M\otimes\tcv M$ as subbundle of $\La^k\Cal A^*M\otimes\tcv
   M$. Moreover, the Lie algebra cohomology differential used in the
   construction is the differential $\partial_{\frak g}$ from the
   standard complex computing the Lie algebra cohomology $H^*(\frak
   g,\tbv)$.

   Now $\La^k\Cal A^*M\otimes\tcv M$ carries an obvious analog of the
   filtration $\{F^\ell_k M\}$. For $k\geq\ell$, we define
   $E^\ell_kM\subset \La^k\Cal A^*M\otimes\tcv M$ to consist of those
   $k$--linear maps which vanish upon insertion of $k-\ell+1$ entries
   from the subbundle $\Cal A_{\frak p}M:=\Cal G\x_Q\frak p\subset\Cal
   AM$.

   Taking $\ph\in E^k_\ell M$ and inserting $k-\ell$ elements of $\Cal
   A_{\frak p}M$, we obtain an $\ell$--linear alternating map on $\Cal
   AM$ with values in $\tcv M$. By construction, the latter map vanishes
   upon insertion of a single element of $\Cal A_{\frak p}M$, so it
   descends to an element of $\La^\ell(\Cal AM/\Cal A_{\frak
     p}M)^*\otimes\tcv M$. Since $(\frak g/\frak p)^*\cong\frak p_+$,
   the latter bundle is induced by the representation $\La^\ell\frak
   p_+\otimes\tbv$. Denoting by $C_\ell$ the bundle $\La^\ell(\Cal
   AM/\Cal A_{\frak p}M)^*\otimes\tcv M$, we have obtained a bundle map
 \begin{equation}\label{Psidef}
 \Psi:E^\ell_k M\to \La^{k-\ell}(\Cal A_{\frak p}M)^*\otimes C_\ell,  
 \end{equation}
 whose kernel evidently equals $E^{\ell+1}_kM$.  Now since the bundle
 $C_\ell$ is induced by a representation of $\frak p$, there is a well
 defined Lie algebra cohomology differential $\partial_{\frak p}$
 mapping $\La^{k-\ell}(\Cal A_{\frak p}M)^*\otimes C_\ell$ to
 $\La^{k-\ell+1}(\Cal A_{\frak p}M)^*\otimes C_\ell$. On the other
 hand, given a section of $\La^{k-\ell}(\Cal A_{\frak p}M)^*\otimes
 C_\ell$ we can form the fundamental derivative, restrict to entries
 from $\Cal A_{\frak p}M$ and then form the complete alternation to
 obtain a section of $\La^{k-\ell+1}(\Cal A_{\frak p}M)^*\otimes
 C_\ell$. We write this operation as $\tau\mapsto \Alt(D\tau)$.

 We next claim that for $\ph\in\Ga(E^\ell_k M)$ we have
 $\tilde{d}^{\tcv}\ph\in\Ga(E^\ell_{k+1}M)$ and that
 $$
 \Ps(\tilde{d}^{\tcv}\ph)=\Alt(D(\Psi(\ph)))+\partial_{\frak p}(\Psi(\ph)). 
 $$ 
 To verify that $\tilde{d}^{\tcv}_1\ph=\Alt(D\ph)$ lies in
 $E^\ell_{k+1}$ take sections $s_0,\dots, s_{k-\ell}\in\Ga(\Cal
 A_{\frak p}M)$ and $t_1,\dots,t_\ell$ in $\Ga(\Cal AM)$. Inserting
 these into $\tilde{d}^{\tcv}_1\ph$, we obtain
 $$
 \tsum_{i=0}^{k-\ell}(-1)^i
 (D_{s_i}\ph)(s_0,\dots,\widehat{s_i},\dots,t_\ell)+
 \tsum_{j=1}^{\ell}(D_{t_j}\ph)(s_0,\dots,\widehat{t_j},\dots,t_\ell). 
 $$ 
 Naturality of the fundamental derivative implies that for each
 $s\in\Ga(\Cal AM)$, we have $D_s\ph\in\Ga(E^\ell_k M)$. This readily
 implies that each summand in the second sum vanishes, since there are
 $k-\ell+1$ entries from $\Cal A_{\frak p}M$. For the same reason, the
 summands in the first sum vanish if one of the $t_j$ is a section of
 $\Cal A_{\frak p}M$. Hence we see that
 $\tilde{d}^{\tcv}_1\ph\in\Ga(E^{\ell}_{k+1})$.

 On the other hand, expanding
 $(D_{s_i}\ph)(s_0,\dots,\widehat{s_i},\dots,t_\ell)$, we obtain a term
 in which $D_{s_i}$ acts on $\ph(s_0,\dots,\widehat{s_i},\dots,t_\ell)=
 (\Ps(\ph)(s_0,\dots,\widehat{s_i},\dots,s_{k-\ell}))(t_1,\dots,t_\ell)$. On
 the other hand, there are terms in which $D_{s_i}$ acts on one of the
 $t$'s and adding those, we obtain
 $(D_{s_i}(\Ps(\ph)(s_0,\dots,\widehat{s_i},\dots,s_{k-\ell})))(t_1,\dots,t_\ell)$.
 Finally, there are terms in which the $D_{s_i}$ hits another $s$, and
 we can rewrite those terms in the form
 $-(\Ps(\ph)(s_0,\dots,\widehat{s_i},\dots,D_{s_i}s_j,\dots,s_{k-\ell}))
 (t_1,\dots,t_\ell)$. Adding these in, one obtains
 $(D_{s_i}(\Ps(\ph)))(s_0,\dots,\widehat{s_i},\dots,s_{k-\ell})$
 evaluated on the $t_j$. This shows that
 $\Ps(\tilde{d}^{\tcv}_1\ph)=\Alt(D(\Ps(\ph)))$.

 The argument for $\partial_{\frak g}\ph$ is similar. Inserting the
 sections $s_i$ and $t_j$ and expanding the definitions, terms of the
 form $t_j\bullet (\ph(\dots))$ and $\ph(\{t_j,t_r\},\dots)$ vanish
 identically, since $k+\ell-1$ entries from $\Cal A_{\frak p}M$ get
 inserted into $\ph$. Moreover, if at least one $t_j$ is a section of
 $\Cal A_{\frak p}M$ then the same applies to the remaining
 terms. (Here one uses that $\frak p\subset\frak g$ is a subalgebra to
 deal with the terms involving $\{s_i,t_j\}$ for this fixed $t_j$.)
 This implies that $\partial_{\frak g}\ph$ is a section of
 $E^\ell_{k+1}M$, too, hence proving the first part of the claim. To
 describe $\Ps(\partial_{\frak g}\ph)$, we observe that
 \begin{align*}
   (-1)^is_i&\bullet(\ph(s_0,\dots,\widehat{s_i},\dots,t_\ell))+\tsum_{j=1}^\ell
   (-1)^{i+j+k-\ell}\ph(\{s_i,t_j\},\dots)\\ &=(-1)^i(s_i\bullet
   (\Ps(\ph)(s_1,\dots,\widehat{s_i},\dots,s_{k-\ell})))(t_1,\dots,t_\ell).
 \end{align*}
 Likewise, we can write the term involving the bracket $\{s_i,s_j\}$ as
 $$
 (-1)^{i+j}(\Ps(\ph)(\{s_i,s_j\},\dots
 ))(t_1,\dots,t_j).
 $$ Hence we conclude that $\Ps(\partial_{\frak g}\ph)=\partial_{\frak
   p}(\Ps(\ph))$, and this completes the proof of the claim.

 \smallskip

 Consider the subbundles $A^\ell_kM:=\La^{k-\ell}\Cal A_{\frak
   p}^*M\otimes\ker(\partial^*)\subset \La^{k-\ell}\Cal A_{\frak
   p}^*M\otimes C_\ell$, which correspond to $\frak q$--invariant
 subspaces in the inducing representations. From the form of these
 subspaces it is evident, that $\partial_{\frak p}$ maps sections of
 $A^\ell_kM$ to sections of $A^\ell_{k+1} M$. Moreover, by naturality
 of the fundamental derivative, we see that for $\ps\in\Ga(A^\ell_kM)$
 and $s\in\Ga(\Cal AM)$ we have $D_s\ps\in\Ga(A^k_\ell M)$. Applying
 this to $s\in\Ga(\Cal A_{\frak p}M)$ and forming the complete
 alternation, we conclude that also $\Alt\o D$ maps $\Ga(A^\ell_k M)$
 to $\Ga(A^\ell_{k+1}M)$.

 But putting $\tilde E^\ell_kM:=\Ps^{-1}(A^\ell_kM)\subset E^\ell_k
 M$, it is clear from the definitions that $\tilde F^\ell_kM=\tilde
 E^\ell_kM\cap F^\ell_k M$. Hence for $\ph\in\Ga(\tilde F^\ell_kM)$,
 we get $\Ps(\ph)\in\Ga(A^\ell_k M)$, so from our claim we conclude
 that $\Ps(d^{\tcv}\ph)\in\Ga(A^\ell_{k+1} M)$. Since this implies
 $d^{\tbv}\ph\in\Ga(\tilde F^\ell_{k+1})$, the first part of the
 proposition is proved.

\smallskip

On the other hand, the bundle map $\pi$ from \eqref{pidef} evidently
extends to a bundle map $\tilde E^\ell_kM\to \La^{k-\ell}\Cal A_{\frak
  p}^*\otimes \tcv_\ell M$, which we denote by the same symbol. Now
for a section $\ph\in\Ga(\tilde F^\ell_k M)$, we know from Theorem
\ref{thm4.1} that $\tilde{d}^{\tcv}\ph$ vanishes upon insertion of a
single section of the subbundle $\Cal G\x_Q\frak q$. Since $\frak
p_+\subset\frak q$, we in particular conclude that
$\Ps(\tilde{d}^{\tcv}\ph)\in\Ga(\La^{k-\ell}\Cal A_{\frak
  p}^*M\otimes\ker(\partial^*))$ vanishes upon insertion of a single
section of $\Cal G\x_Q\frak p_+\subset A_{\frak p}M$. Hence
$\Ps(\tilde{d}^{\tcv}\ph)$ naturally descends to a section of
$\La^{k-\ell}A_\rho^*M\otimes\ker(\partial^*)$, so projecting to
homology, we see that $\pi(\tilde{d}^{\tcv}\ph)$ naturally descends to
$\La^{k-\ell}A_\rho^*M\otimes\Cal V_\ell M$. Moreover, $\Alt\o D$
descends to $\Alt\o D^\rho$ on that bundle. On the other hand, the Lie
algebra cohomology differential $\partial_{\frak p}$ on
$\La^{k-\ell}A_{\frak p}^*M\otimes\ker(\partial^*)$ by construction
descends to $\partial_{\frak p/\frak p_+}$ on
$\La^{k-\ell}A_\rho^*M\otimes\Cal V_\ell$. Thus the second part of the
proposition follows.
 \end{proof}

 \subsection{Absolute vs.\ relative BGG sequences}\label{5.3}
 We can now complete the general discussion of the first kind of
 examples. We continue using the notation of Section \ref{5.2}, so
 $\tbv$ is a finite dimensional irreducible representation of $G$ and
 for some $\ell\leq k$, we denote by $\Bbb V_\ell$ the completely
 reducible representation $H_\ell(\frak p_+,\tbv)$ of $P$. We want to
 compare the BGG sequence corresponding to the tractor bundle $\tcv M$
 induced by $\tbv$ to the relative BGG sequence corresponding to the
 relative tractor bundle $\Cal V_\ell M$ induced by $\Bbb V_\ell$. To
 explicitly relate the bundles showing up in the two sequences, we
 have to use another property of $Q$--invariant subspaces
 $\tcf_k^\ell\subset\Cal F_k^\ell\subset C_k(\frak q_+,\tbv)$ from
 Section \ref{5.2}. Namely, we get $\Cal
 F^{\ell+1}_k\subset\tcf^\ell_k$ and by Proposition 3.6 of
 \cite{part1}, $\ker(\partial^*_{\frak q})\cap\Cal
 F^\ell_k\subset\tcf^\ell_k$. In that proposition it is also shown
 that the $\frak q$--equivariant map $\pi$ inducing the bundle map
 \eqref{pidef} vanishes on $\Cal F^{\ell+1}_k$ and has the property
 that, up to sign, $\partial^*_\rho\o\pi$ coincides with
 $\pi\o\partial^*_{\frak q}$.

 The obvious consequences of these properties for induced bundles and
 induced bundle maps imply that $\pi$ is defined on
 $\ker(\partial^*_{\frak q})\cap F^\ell_kM$ and its values on that
 space are contained in
 $\ker(\partial^*_\rho)\subset\La^{k-\ell}T^*_\rho M\otimes\Cal V_\ell
 M$. Hence we can project to the relative homology bundle and obtain a
 bundle map
 \begin{equation}\label{Pidef}
 \Pi:\ker(\partial^*_{\frak q})\cap \tilde F^\ell_k M\to \Cal
 H_{k-\ell}(T^*_\rho M,\Cal V_\ell M). 
 \end{equation} 
 By Theorem 3.7 of \cite{part1}, this bundle map vanishes on
 $\im(\partial^*_{\frak q})\cap \tilde F^\ell_k M$, so it descends to
 a bundle map on $\Cal H_k(T^*M,\tcv M)$. That theorem also implies
 that the result is a surjection $\Cal H_k(T^*M,\tcv M)\to\Cal
 H_{k-\ell}(T^*_\rho M,\Cal V_\ell M)$. Hence we can use it to
 identify the copies of the bundles showing up in the relative BGG
 sequence induced by $\Bbb V_\ell$ with their counterparts in the
 absolute BGG sequence determined by $\tbv$. We refer to the part of
 the absolute BGG sequence determined by $\tbv$ formed by these
 bundles and the operators mapping between them, as the
 \textit{subsequence determined by} $\Bbb V_\ell$.

 \begin{thm}\label{thm5.3}
   For any $\ell\leq k$, the identification between the relative BGG
   sequence determined by $\Bbb V_\ell=H_\ell(\frak p_+,\tbv)$ and the
   subsequence of the absolute BGG sequence determined by $\tbv$
   constructed above is compatible with the operators in the two
   sequences.

   In particular, under the curvature conditions from Theorem
   \ref{thm4.6} respectively Proposition \ref{prop4.9}, this
   subsequence is a subcomplex and a fine resolution of a sheaf as
   described there.
 \end{thm}
 \begin{proof}
   Consider the $Q$--submodule $\tcf^\ell_k\subset\La^k\frak
   q_+\otimes\tbv$. From the proof of Theorem 3.7 in \cite{part1}, we
   see that all irreducible components of $H_k(\frak q_+,\tbv)$ which
   also occur in $H_{k-\ell}(\frak q_+/\frak p_+,H_\ell(\frak p_+,\Bbb
   V))$ are contained in $\ker(\square)\cap\tcf^\ell_k$. Since $\tbv$
   is an irreducible representation of $\frak g$, the Casimir acts by
   a scalar on $\tbv$, so part (1) of Proposition \ref{prop4.8}
   applies. This shows that for any section $\al$ of $\Cal
   H_k(T^*M,\tcv M)$ which is contained in the subsequence determined
   by $\Bbb V_\ell$, the image $S(\al)$ under the splitting operator
   lies in $\Ga(\tilde F^\ell_kM)\subset\Om^k(M,\tcv M)$. Hence we can
   apply the bundle map $\pi$ from \eqref{pidef} to obtain
   $\ph:=\pi(S(\al))\in \Om^{k-\ell}_\rho(M,\Cal V_\ell M)$. Since
   $\partial^*_{\frak q}(S(\al))=0$, we conclude that
   $\partial^*_\rho(\ph)=0$ so we can project to a section of $\Cal
   H_{k-\ell}(T^*_\rho M,\Cal V_\ell M)$. By construction, the
   resulting section coincides with $\Pi(S(\al))$ so from the above
   discussion we see that this is the section corresponding to $\al$
   under the identification of the subsequence with the relative BGG
   sequence.

   We next claim that $\ph=\pi(S(\al))$ coincides with $S_\rho(\al)$,
   where
 $$
 S_\rho:\Ga(H_{k-\ell}(T^*_\rho M,\Cal V_\ell M))\to
   \Om^{k-\ell}_\rho(M,\Cal V_\ell M)
 $$ 
 is the splitting operator coming from the relative BGG
 construction. We already know that $\partial^*_\rho (\ph)=0$ and the
 projection of $\ph$ to cohomology coincides with $\al$. In view of
 part (2) of Theorem \ref{thm3.3}, it thus suffices to verify that
 $\partial^*_\rho d^{\tcv_\ell}\ph=0$ to complete the proof of the
 claim. But since $\ph=\pi(S(\al))$, Proposition \ref{prop5.2} shows
 that $d^{\tcv_\ell}\ph=\pi(d^{\tcv}S(\al))$ and the compatibility of
 $\pi$ with the Lie algebra homology differentials then implies that
 $\partial^*_\rho d^{\tcv_\ell}\ph=\pm\pi(\partial^*_{\frak
   q}d^{\tcv}S(\al))=0$.

 Knowing that $\pi(S(\al))=S_\rho (\al)$ we can again use Proposition
 \ref{prop5.2} to conclude that $d^{\tcv}S(\al)\in\Ga(\tilde
 F^\ell_{k+1}M)$ and that $\pi(d^{\tcv}S(\al))=d^{\Cal V_\ell}S_\rho
 (\al)$. This is a section of
 $\ker(\partial^*_\rho)\subset\Om^{k+1}_\rho (M,\Cal V_\ell M)$ and
 projecting to cohomology we obtain $D_\rho(\al)$, the relative
 BGG--operator. But this can be equivalently written as
 $\Pi(d^{\tcv}S(\al))$, so under the identification of the two
 sequences, it coincides with the components of the projection of
 $d^{\tcv}S(\al)$ to cohomology, which lie in the subsequence. But
 denoting by $D$ the absolute BGG operator, this is exactly the
 component of $D(\al)$ contained in the subsequence.
 \end{proof}

 \begin{remark}\label{rem5.3}
   The proof of the theorem shows how the machineries corresponding to
   the absolute and relative BGG sequences are related. If
   $\al\in\Ga(\Cal H_k(T^*M,\tcv M))$ lies in the subsequence
   determined by $\Cal V_\ell$, then the absolute splitting operator
   $S$ has the property that $S(\al)\in\Ga(\tilde
   F^\ell_kM)\subset\Om^k(M,\tcv M)$. Hence we can apply $\pi$ and
   $\pi(S(\al))=S_\rho(\al)\in\Om^{k-\ell}_\rho (M,\Cal V_\ell M)$,
   where $S_\rho$ is the relative splitting operator. 
 \end{remark}

 \subsection{Generalized path geometries}\label{5.4}
 To obtain explicit examples of relative BGG sequences, we consider
 the example of generalized path geometries, see Section 4.4.3 of
 \cite{book}. A \textit{path geometry} on a smooth manifold $N$ of
 dimension $n+1$ is given by a smooth family of 1--dimensional
 submanifolds in $N$, such that, given any point $x\in N$ and any line
 $\ell$ in $T_xM$, there is a unique submanifold in the family which
 contains $x$ and is tangent to $\ell$ in $x$. The importance of path
 geometries comes from their relation to systems of second order
 ODEs. Given such a system on an open subset $U\subset\Bbb R^n$ (or on
 some manifold), the graphs of all solutions define a path geometry on
 (an open subset of) $U\x\Bbb R$, so this gives a
 coordinate--independent way to study such systems.

 To encode a path geometry, one passes to the projectivized tangent
 bundle $M:=\Cal P(TN)$ of $N$. Viewing the submanifolds in $N$ as
 regularly parametrized curves, it is evident that they lift to $M$,
 and taking tangents, one obtains a smooth line subbundle $E\subset
 TM$, which is transversal to the vertical subbundle $VM$ of the
 projection $M\to N$. The sum $H:=E\oplus V\subset TM$ is the
 so--called tautological subbundle in $TM$, whose fiber in a point
 $\ell$ consist of those tangent vectors which project to elements of
 $\ell$. One can then recover the submanifolds in the initial family
 as the projections to $N$ of the leaves of the foliation defined by
 the line subbundle $E\subset TM$. Hence a path geometry can be
 equivalently defined as specifying a line subbundle $E$ in the
 tautological bundle which is complementary to the vertical subbundle.

 The concept of a \textit{generalized path geometry} is then obtained
 by requiring an abstract version of the properties of the subbundles
 defining a path geometry. If $M$ is any smooth manifold of dimension
 $2n+1$, then a generalized path geometry on $M$ is given by two
 subbundles $E$ and $V$ in $TM$ of rank $1$ and $n$, respectively,
 which intersect only in zero. Moreover, one requires that the Lie
 bracket of two sections of $V$ is a section of $H:=E\oplus V$, while
 projecting the bracket of a section of $E$ and a section of $V$ to the
 quotient induces an isomorphism $E\otimes V\to TM/H$ of vector
 bundles.

 It turns out (see again Section 4.4.3 of \cite{book}) that
 generalized path geometries can be equivalently described as
 parabolic geometries of type $(G,Q)$, where $G=PGL(n+2,\Bbb R)$ and
 $Q$ is a the stabilizer of a flag in $\Bbb R^{n+2}$ consisting of a
 line contained in a plane. Now there are two obvious intermediate
 parabolics lying between $Q$ and $G$, namely the stabilizer $P$ of
 the line and the stabilizer $\tilde P$ of the plane, so
 $Q=P\cap\tilde P$. Hence on a generalized path geometry, there are
 two kinds of relative BGG sequences available, namely the ones
 corresponding to $\frak p\supset\frak q$ and the ones corresponding
 to $\tilde{\frak p}\supset\frak q$. Since the latter consist of a
 single operator, we will focus on describing the former class.

 \subsection{Relative BGG sequences on generalized path
   geometries}\label{5.5} 
 To make our results explicit for generalized path geometries, it
 mainly remains to connect representation theory data to geometric
 objects. On the one hand, we have to ensure that there are
 sufficiently many relative tractor bundles available to start the
 construction. On the other hand, we have to discuss how weights are
 realized in terms of natural bundles. We do the second step in detail
 only in the case $n=2$, i.e.\ for generalized path geometries in
 dimension $5$. Here the representation theory information is available
 in \cite{part1}. Higher dimensions can be dealt with in a similar
 way.

 For the question of existence of relative tractor bundles, we have to
 construct completely reducible representations of the group $P$, the
 stabilizer of a line (i.e.~of a point in projective space) in
 $G=PGL(n+2,\Bbb R)$. As discussed in Section 4.1.5 of \cite{book},
 the Levi component $P_0\subset P$ is given by the classes of block
 diagonal matrices of the form $\left(\begin{smallmatrix} c & 0 \\ 0 &
     C \end{smallmatrix}\right)\in GL(n+2,\Bbb R)$ with $0\neq
 c\in\Bbb R$ and $C\in GL(n+1,\Bbb R)$. It is also shown there that
 the representation of $P_0$ on $\Bbb R^{n+1}$ defined by $X\mapsto
 c^{-1}CX$ can be realized on $\frak g/\frak p$ via the adjoint
 representation.

 Forming exterior powers of this basic representation, one obtains the
 fundamental representations of $GL(n+1,\Bbb R)$ up to a twist by a
 multiplication by some power of $c$. On the other hand, the center of
 $P$ is isomorphic to $\Bbb R\setminus\{0\}$, and the top exterior
 power of the representation on $\frak g/\frak p$ from above gives a
 non--trivial representation of the center. Forming the square of this
 representation, the action depends only on the absolute value, so one
 can take arbitrary real roots of the resulting representation. By
 tensorizing with such representations, the action of the center can
 be changed arbitrarily. Hence we conclude that any weight which is
 $\frak p$--dominant and $\frak p$--integral can be realized by a
 finite dimensional representation of $P_0$ and hence by a completely
 reducible representation of $P$. (Initially weights are considered
 for complex representations, but there is no problem to use them in a
 real setting here, since the real Lie algebra $\frak{sl}(n+2,\Bbb R)$
 we are dealing with is a split real form of its complexification.) We
 will make this more explicit in the case $n=2$ in Section \ref{5.6}
 below.

 As a second step, let us discuss the bundles induced by
 representations of $Q_0$ in the case $n=2$. In the Dynkin diagram
 notation used in \cite{part1}, we have to consider weights of the form
 $\xxd{a}{b}{c}$. Again it is no problem to work with weights in the
 real setting here. From the Dynkin diagram it is clear that the
 fundamental representations corresponding to the first two (crossed)
 nodes will be one--dimensional, while the fundamental representation
 corresponding to the last node has dimension two. Correspondingly, we
 obtain a two parameter family of line bundles and one basic rank two
 bundle. For our purposes, there is no need to discuss existence of
 representations of $Q$ realizing a given weight. We have discussed
 existence of representations inducing relative tractor bundles
 above. These give rise to representations of $Q_0$ on relative Lie
 algebra homology groups, which induce the completely reducible
 natural bundles showing up in relative BGG sequences.

 Hence we just briefly discuss the relation between weights and the
 basic bundles available for the geometry. For $w,w'\in\Bbb R$, we
 denote by $\Cal E(w,w')$ the bundle corresponding to the weight
 $\xxd{w}{w'}{0}$ (not worrying about existence). The correspondence
 between natural bundles and the Lie algebra $\frak{sl}(n+2,\Bbb R)$
 (c.f.\ Section 4.4.3 of \cite{book}) allows us to read off the
 weights corresponding to the constituents of the associated graded of
 the tangent-- and cotangent bundle, as well as bundles constructed
 from those.  This shows that $E=\Cal E(2,-1)$, $\La^2V=\Cal E(2,3)$,
 and $\La^2(TM/H)=\Cal E(2,1)$. Next, the bundle $V$ corresponds to
 $\xxd{-1}{1}{1}$, its dual $V^*$ corresponds to $\xxd{1}{-2}{1}$,
 while $TM/H\cong E\otimes V$ and its dual correspond to
 $\xxd{1}{0}{1}$ and $\xxd{-1}{-1}{1}$, respectively. Together with
 the line bundles $\Cal E(w,w')$, any of these for bundles can be used
 to construct a bundle corresponding to any given weight. It will be
 most convenient to take $V^*$ as the basic ingredient, and we will
 follow the usual convention that adding ``$(w,w')$'' to the name of a
 bundle indicates a tensor product with $\Cal E(w,w')$. Then for
 $a,b\in\Bbb R$ and $c\in\Bbb N$, the weight $\xxd{a}{b}{c}$ is
 realized by the bundle $S^cV^*(a-c,b+2c)$.

 Having this at hand, we can describe the basic form of relative BGG
 sequences corresponding to $\frak p\supset\frak q$ on a generalized
 path geometry in dimension five.

 \begin{thm}\label{thm5.5}
   Let $H=E\oplus V\subset TM$ be a generalized path geometry of
   dimension $5$. Then for each $w\in\Bbb R$ and $k,\ell\in\Bbb N$,
   there is a sequence of invariant differential operators
 \begin{equation}\label{pathbgg}
\begin{gathered}
  \Ga(\Cal W_0)\overset{D_1}{\longrightarrow} \Ga(\Cal W_1)
  \overset{D_2}{\longrightarrow}\Ga(\Cal W_2),\text{\
    with\ }\Cal W_0=S^kV^*(w,2k+\ell),\\
  \Cal W_1=S^{k+\ell+1}V^*(w,2k+\ell),\text{\ and\ } \Cal W_2=S^\ell
  V^*(w+2k+2,\ell-k-3).
 \end{gathered}
 \end{equation}
 This sequence is contained in a standard BGG sequence if and only if
 $w\in\Bbb Z$ and one of the following four mutually exclusive
 conditions is satisfied:
 \begin{itemize}
 \item $w+k\geq 0$
 \item $w+k\leq -2$ and $w+k+\ell\geq -1$
 \item $w+k+\ell\leq -3$ and $w+2k+\ell\geq -2$ 
 \item $w+2k+\ell\leq -4$
 \end{itemize}
 If either $w=-1-k$, or $w=-2-k-\ell$, or $w=-3-2k-\ell$, then the
 representations in the sequence have singular infinitesimal
 character.
 \end{thm}
 \begin{proof}
   We start from the relative tractor bundle $\Cal VM$ induced by the
   representation $\Bbb V$ corresponding to the weight $\xdd{a}{b}{c}$
   with $a=w+k$, $b=\ell$ and $c=k$. Then Theorem 2.7 of \cite{part1}
   shows that the homology groups $H_i(\frak q_+/\frak p_+,\Bbb V)$
   for $i=0,1,2$ correspond to the weights listed in formula (3.1) in
   Example 3.2 of \cite{part1}. Expressing these weights in terms
   of $w$, $k$ and $\ell$, the discussion above this theorem then
   shows that the three bundles in the sequence \eqref{pathbgg} are
   $\Cal H_i(T^*_\rho M,\Cal VM)$ for $i=0,1,2$. Thus existence of the
   sequence of invariant differential operators follows directly from
   Theorem \ref{thm4.1}.

   The standard BGG sequences on $M$ are indexed by irreducible
   representations of $G$ and thus by $\frak g$--dominant integral
   weights. Since all weights in the affine Weyl orbit of an integral
   weight are integral, too, we see that the condition $a=w+k\in\Bbb
   Z$ is necessary. The bundles occurring in the standard BGG sequence
   induced by a representation $\tbv$ of $\frak g$ correspond to the
   representations $H_*(\frak q_+,\Bbb V)$. The relation between
   absolute and relative homology groups is discussed in detail in
   Section 3.2 of \cite{part1}, and in Example 3.2 of that reference,
   this is made explicit in the case we consider here. Starting from a
   dominant integral weight $\ddd{a'}{b'}{c'}$, the corresponding
   absolute BGG sequence contains four relative BGG sequences as
   subsequences. The initial weights of these four sequences are
   listed in formula (3.2) of \cite{part1}. It is elementary to verify
   that, for $w\in\Bbb Z$ and hence $a=w+k\in\Bbb Z$, the condition
   that $\xdd{a}{b}{c}$ equals one of these four weights is equivalent
   to $a\geq 0$, respectively $a\leq -2$ and $a+b\geq -1$,
   respectively $a+b\leq -3$ and $a+b+c\geq -2$, respectively
   $a+b+c\leq -4$. The conditions in the theorem just express these in
   terms of $w$, $k$, and $\ell$.

   The cases in which the representations corresponding to the
   relative homology groups have singular infinitesimal character are
   also listed in Example 3.2 of \cite{part1}. The conditions in the
   theorem just equivalently express these in terms of $w$, $k$ and
   $\ell$.
\end{proof}

\begin{remark*}
  At this stage, it is not clear whether the operators in the sequence
  actually are non--zero. This will follow from the results that we
  obtain resolutions in the case of path geometries below. One can
  actually go much further in that direction and obtain a description
  of the principal parts of the operators using only representation
  theory.

  In order to do this in the case discussed here, it suffices to
  verify that the orders of the two operators in the sequences are
  $\ell+1$ for $D_1$ and $k+1$ for $D_2$. Now for $D_1:\Ga(\Cal
  W_0)\to \Ga(\Cal W_1)$, the target bundle $\Cal
  W_1=S^{k+\ell+1}(w,2k+\ell)$ evidently is included in
  $S^{\ell+1}V^*\otimes\Cal W_0$, since $\Cal
  W_0=S^k(w,2k+\ell)$. Indeed it corresponds to the highest weight
  component in the tensor product of the inducing
  representations. Since $D_1$ is an invariant differential operator,
  on the homogeneous model its symbol must be an equivariant map of
  homogeneous vector bundles. Hence it is induced by a
  $Q_0$--equivariant map between the inducing representations.

  Taking into account that the operators are constructed from vertical
  derivatives, it follows that the symbol is defined on
  $S^{\ell+1}V^*\otimes\Cal W_0$, so $Q_0$--equivariancy pins it down up
  to a constant multiple.

  For the second operator $D_2:\Cal W_1\to\Cal W_2$, the situation is
  only slightly more complicated. Here the dual bundles are $\Cal
  W_2^*=S^\ell V(-w-2k-2,k-\ell+3)$ and $\Cal
  W_1^*=S^{k+\ell+1}V(-w,-\ell-2k)$. From the discussion of the
  relation between representations and bundles above, we see that
  $V^*\cong V(2,-3)$, so $S^{k+1}V^*\cong S^{k+1}V(2k+2,-3k-3)$. Hence
  $\Cal W_1^*$ is naturally contained in the tensor product
  $S^{k+1}V^*\otimes\Cal W_2^*$ corresponding to the highest weight
  component in the tensor product of the inducing representations. Now
  the unique (up to scale) $Q_0$--homomorphism $S^{k+1}V^*\otimes\Cal
  W_2^*\to\Cal W_1^*$ dualizes to a unique $Q_0$--homomorphism
  $S^{c+1}V^*\otimes\Cal W_1\to\Cal W_2$. Knowing the order, one again
  obtains the symbol on the homogeneous model, up to a constant
  factor. Finally, one argues that passing to a curved geometry does
  not change the principal part of the operator.
\end{remark*}

\subsection{Relative BGG resolutions on path geometries}\label{5.6}
The concepts of correspondence spaces and local twistor spaces as
discussed in Sections \ref{4.4}, \ref{4.6} and \ref{4.6a} arise very
naturally in the case of generalized path geometries. To discuss
correspondence spaces, recall that a regular normal parabolic geometry
of type $(G,P)$ on a manifold $N$ of dimension $n+1$ is equivalent to
a projective structure on $N$. Such a structure is given by an
equivalence class of torsion--free linear connections on $TN$, which
share the same geodesics up to parametrization. The unparametrized
geodesics of the connections in the class define a path geometry on
$N$ and the correspondence space $\Cal CN$ is the associated geometry
on $\Cal P(TN)$. In the language of systems of second order ODEs,
local isomorphism to a correspondence space thus is related to
realizability of a system as a geodesic equation.

On the other hand, let us recall the description of harmonic curvature
components for generalized path geometries from Section 4.4.3 of
\cite{book}. In all dimensions $n\geq 2$, there is one harmonic
torsion $\tau_E\in\Ga(E^*\otimes (TM/H)^*\otimes V)$ and a curvature,
which we denote by $\ga\in\Ga(V^*\otimes (TM/H)^*\otimes\End(V))$. For
$n=2$, there is an additional torsion $\tau_V\in\Ga(\La^2V^*\otimes
E)$. Using these, we can formulate the conditions for existence of
local twistor spaces and local isomorphism to a correspondence space.

\begin{lemma}\label{lemma5.6}
  Let $(M,E,V)$ be a generalized path geometry of dimension $2n+1$
  with $n\geq 2$. Then we have

  (1) This relative tangent bundle $T_\rho M$ is the bundle $V$. If
  $n>2$, this bundle is always involutive, for $n=2$, this is the case
  if and only if $\tau_V=0$.

  (2) Involutivity of $T_\rho M=V$ is equivalent to the fact that the
  geometry is locally isomorphic to a path geometry on a local twistor
  space.

  (3) If $T_\rho M=V$ is involutive, then the geometry is locally
  isomorphic to a correspondence space for a projective structure on a
  local twistor space if and only if the harmonic curvature component
  $\ga$ vanishes identically.
\end{lemma}
\begin{proof} 
  Section 4.4.4 of \cite{book} contains a proof of (1). If $V$ is
  involutive, then for a local leaf space $\ps:U\to N$, define
  $\tilde\ps:U\to\Cal PTN$ by mapping $x\in U$ to the line
  $T_x\ps\cdot E_x\subset T_{\ps(x)}N$. Then it is shown in
  Proposition 4.4.4 of \cite{book} that for sufficiently small $U$,
  the map $\tilde\ps$ is an open embedding whose tangent map sends $V$
  to the vertical bundle and $E\oplus V$ to the tautological bundle,
  so (2) follows.

  If $\ga=0$ and (if $n=2$), also $\tau_V=0$, then the harmonic
  curvature $\ka^h$ evidently satisfies the assumptions of part (2) of
  Proposition \ref{prop4.9}, which then implies (3).
\end{proof}

The last ingredient we need to make our results on resolutions
explicit is notation for some tensor bundles. For a smooth manifold
$N$ of dimension $n$, we define $\Cal E[w]$ to be the bundle of
densities of weight $\frac{-w}{n+1}$ on $N$ (so $\Cal E[-n-1]$ is the
bundle of volume densities). Adding ``$[w]$'' to the name of a tensor
bundle will indicate a tensor product with $\Cal E[w]$. Now for
$k,\ell\in\Bbb N$, consider the tensor product $S^kT^*N\otimes S^\ell
TN$. If both $k$ and $\ell$ are positive, then there is a unique
contraction from this bundle to $S^{k-1}T^*N\otimes S^{\ell-1}TN$, and
we denote by $\Cal T_k^\ell$ the kernel of this contraction. We
further define $\Cal T^0_k:=S^kT^*N$ and $\Cal T^\ell_0:=S^\ell TN$.

\begin{prop}\label{prop5.6}
  Let $(M,E,V)$ be a generalized path geometry of dimension $2n+1$
  with $n\geq 2$ such that the relative tangent bundle $T_\rho M=V$ is
  involutive. Let $\Bbb V$ be a completely reducible representation of
  $P$, $\Cal VM\to M$ the corresponding relative tractor bundle and
  $\nabla^{\rho,\Cal V}$ the relative tractor connection on $\Cal VM$.

  (1) The relative BGG sequence induced by $\Bbb V$ is a complex and a
  fine resolution of the sheaf $\ker(\nabla^{\rho,\Cal V})$.  In
  particular, if $\Bbb V$ is an irreducible component of $H_*(\frak
  p_+,\tbv)$ for a representation $\tbv$ of $\frak g$, then one
  obtains a subcomplex in a curved BGG sequence.

  (2) If $M$ is the correspondence space $\Cal CN$ for a projective
  structure on a manifold $N$ of dimension $n+1$, then the sheaf
  $\ker(\nabla^{\rho,\Cal V})$ is globally isomorphic to the pullback
  of the sheaf of smooth sections of the tensor bundle over $N$
  induced by the representation $\Bbb V$.

  (3) If $M\cong P(TN)$ is a path geometry on a manifold $N$ of
  dimension $n+1$, then the isomorphism of sheaves from (2) holds
  locally (with the same tensor bundle).

  (4) If $n=2$ and $\tau_V=0$, then for the sequence \eqref{pathbgg}
  of invariant differential operators in Theorem \ref{thm5.5} the
  tensor bundle from parts (2) and (3) is $\Cal T^k_\ell[w+2\ell]$.
\end{prop}
\begin{proof}
  If $V$ is involutive, then $\tau_V=0$, and the harmonic curvature
  $\ka^h$ visibly satisfies the assumptions of part (1) of Proposition
  \ref{prop4.9}. Hence we conclude that $\ka_{\rho}=0$ and we may
  apply Theorem \ref{thm4.6} to obtain (1). As noted in the proof of
  Lemma \ref{lemma5.6}, in the case of a correspondence space, we can
  apply part (2) of Proposition \ref{prop4.9}, so (2) again follows
  from Theorem \ref{thm4.6}.

  Assuming that $M$ is a path geometry over $N$ (i.e.~that $N$ is a
  global twistor space for $M$), the fact that $\ka_{\rho}=0$ implies
  that we can apply Lemma \ref{lemma4.6a}. This shows that locally the
  Cartan bundle $\Cal G\to M$ is isomorphic to a principal $P$--bundle
  $\Cal F\to N$, and the sheaf $\ker(\nabla^{\rho,\Cal V})$ can be
  identified with $\Ga(\Cal F\x_P\Bbb V)$.

  Finally, we can apply the well known result that lowest non--zero
  homogeneous component of the curvature of any regular normal
  parabolic geometry is harmonic. Since $\tau_V$ vanishes identically,
  the list of harmonic components above shows that the next lowest
  possible homogeneity is two, and this is represented by
  $\tau_E$. Consequently, only components of homogeneity at least
  three contribute to values of $\ka$ if one of the entries is from
  $T_\rho M$. But this immediately implies that for $\xi\in \Ga(T_\rho
  M)$ and $\eta\in\frak X(M)$, we get $\ka(\xi,\eta)\in\Cal G\x_Q\frak
  q\subset\Cal G\x_Q\frak p$. Hence we can apply part (2) of Theorem
  \ref{thm4.6a}, which shows that the Cartan connection induces a
  soldering form $\th\in\Om^1(\Cal F,\frak g/\frak p)$. This implies
  that $\Cal F\x_P(\frak g/\frak p)\cong TN$, and hence the
  correspondence between completely reducible representations of $P$
  and tensor bundles is the same as in the case of a projective
  structure on $N$. This completes the proof of (3).

\smallskip

(4) The representation $\frak g/\frak p$ is the $\frak p$--irreducible
quotient of the adjoint representation, so the two representations
have the same lowest weight. Hence $\frak g/\frak p$ corresponds to
the weight $\xdd101$. Similarly, one verifies that the dual $(\frak
g/\frak p)^*\cong\frak p_+$ corresponds to the weight
$\xdd{-2}10$. Hence the tensor bundle $\Cal T^k_\ell$, which is
induced by the highest weight component in $S^\ell(\frak g/\frak
p)^*\otimes S^k(\frak g/\frak p)$ corresponds to the weight
$\xdd{k-2\ell}{\ell}{k}$, which implies the result.
\end{proof}

\end{document}